\documentclass[a4paper,leqno,11pt]{amsart}
\usepackage{lmodern}
\usepackage{amsmath, amsfonts, amssymb, amsthm}
\usepackage{esint}
\usepackage{mathtools}
\usepackage{bbm}
\usepackage{standalone}
\usepackage{bm}
\usepackage{xparse}
\usepackage[inline]{enumitem}
\usepackage{verbatim}
\usepackage{mathrsfs}
\usepackage{tikz-cd}
\usepackage{mathrsfs}
\usepackage[margin=2.54cm]{geometry}
\usepackage[normalem]{ulem}
\usepackage{graphicx} 

\usepackage{float}
\usepackage{stmaryrd}

\newtheorem{thm}{Theorem}[section]
\newtheorem{lem}[thm]{Lemma}

\newtheorem{prop}[thm]{Proposition}
\theoremstyle{definition}

\newtheorem{rem}[thm]{Remark}

\def\Xint#1{\mathchoice
{\XXint\displaystyle\textstyle{#1}}
{\XXint\textstyle\scriptstyle{#1}}
{\XXint\scriptstyle\scriptscriptstyle{#1}}
{\XXint\scriptscriptstyle\scriptscriptstyle{#1}}
\!\int}
\def\XXint#1#2#3{{\setbox0=\hbox{$#1{#2#3}{\int}$ }
\vcenter{\hbox{$#2#3$ }}\kern-.57\wd0}}

\def\dashint{\Xint-}

\def\YYint#1#2#3{{\setbox0=\hbox{$#1{#2#3}{\iint}$}
    \vcenter{\hbox{$#2#3$}}\kern-.51\wd0}}
 
\newcommand*\dd{\mathop{}\!\mathrm{d}}

\newcommand{\Rn}{\mathbb{R}^{n}}
\newcommand{\R}{\mathbb{R}}
\newcommand{\Z}{\mathbb{Z}}

\newcommand{\N}{\mathbb{N}}
\newcommand{\abs}[1]{\left| #1\right|}
\newcommand{\smallabs}[1]{| #1 |}
\newcommand{\norm}[1]{\left\lVert #1\right\rVert}
\newcommand{\smallnorm}[1]{\lVert #1 \rVert}
\newcommand{\meas}[1]{\left\lvert #1\right\rvert}
\newcommand{\ind}[1]{\mathbbm{1}_{#1}}
\newcommand{\RNum}[1]{\textup{\uppercase\expandafter{\romannumeral#1}}}
\newcommand{\Hn}{\mathbb{R}^{1+n}_{+}}
\newcommand{\diam}[1]{\textup{diam}(#1)}
\newcommand{\dist}{\textup{dist}}
\newcommand{\clos}[1]{\overline{#1}}
\newcommand{\supp}[1]{\textup{supp}\left(#1\right)}
\newcommand{\eps}{\varepsilon}

\newcommand{\Ell}[1]{\mathrm{L}^{#1}(\Rn)}
\newcommand{\El}[1]{\mathrm{L}^{#1}}
\newcommand{\bb}[1]{\mathbb{#1}}
\newcommand{\set}[1]{\left\{#1\right\}}
\newcommand{\loc}{\textup{loc}}
\newcommand{\dvg}{\textup{div}}

\numberwithin{equation}{section}
\usepackage[colorlinks,citecolor=red,linkcolor=blue]{hyperref}  
\title[Tangential approach]{Tangential approach in the Dirichlet problem for elliptic equations}

\author{Jonathan Bennett}
\author{Arnaud Dumont}
\author{Andrew J. Morris}

\address{J. Bennett, School of Mathematics, University of Birmingham, Edgbaston, B15 2TT, UK}
\email{j.bennett@bham.ac.uk}
\address{A. Dumont, School of Mathematics, University of Birmingham, Edgbaston, B15 2TT, UK}
\email{axd461@bham.ac.uk}
\address{A. J. Morris, School of Mathematics, University of Birmingham, Edgbaston, B15 2TT, UK}
\email{a.morris.2@bham.ac.uk}

\date{\today}
\thanks{The first author was supported by EPSRC Grant EP/W032880/1.}
\thanks{No data were created or analysed in this study.}

\begin{document}
\begin{abstract}

It is well-known that solvability of the $\mathrm{L}^{p}$-Dirichlet problem for elliptic equations $Lu:=-\mathrm{div}(A\nabla u)=0$ with real-valued, bounded and measurable coefficients $A$ on Lipschitz domains $\Omega\subset\mathbb{R}^{1+n}$ is characterised by a quantitative absolute continuity of the associated $L$-harmonic measure. We prove that this local $A_{\infty}$ property is sufficient to guarantee that the nontangential convergence afforded to $\mathrm{L}^{p}$ boundary data actually improves to a certain \emph{tangential} convergence when the data has additional (Sobolev) regularity. Moreover, we obtain sharp estimates on the Hausdorff dimension of the set on which such convergence can fail. This extends results obtained by Dorronsoro, Nagel, Rudin, Shapiro and Stein for classical harmonic functions in the upper half-space. 
\end{abstract}
\maketitle
\setcounter{tocdepth}{1}
\section{Introduction}
\subsection{Background: Laplace's equation in the upper half-space}
A classical theorem of Fatou states that for any $p\in[1,\infty]$ and $f$ belonging to the Lebesgue space $\Ell{p}$, the harmonic extension $$u_{f}(t,x):=(P_{t}\ast f)(x)$$ of $f$ to the upper half-space $\Hn=\set{(t,x)\in\R\times \Rn : t>0}$ admits nontangential limits almost everywhere on its boundary $\Rn$. This means that
\begin{equation*}
    \lim_{\Gamma_{a}(x_0)\ni(t,x)\to (0,x_0)}u_{f}(t,x)=f(x_0) 
\end{equation*}
for almost every $x_0\in\mathbb{R}^n$,
where $$\Gamma_{a}(x_0):=\set{(t,x)\in\Hn : \abs{x-x_0}<at}.$$ In this context the set $\Gamma_a(x_0)$ is referred to as a \textit{nontangential approach region}, being a cone with vertex $x_0$ and aperture parameter $a>0$. We recall that $P_t:\mathbb{R}^n\rightarrow\mathbb{R}_+$ is the classical Poisson kernel, and that $u_f(t,x)$, being a harmonic function on $\Hn$, solves Laplace's equation $$\Delta u:=\frac{\partial^2 u}{\partial t^2}+\sum_{j=1}^n\frac{\partial^2 u}{\partial x_j^2}=0$$ on $\Hn$. We note that the original version of this convergence result is stated in the context of harmonic extensions to the unit disc rather than the half space; see \cite{Fatou_1906}. 

It is well known that taking a \textit{tangential} approach to the boundary can lead to the failure of such convergence. A curve $\gamma\subset\Hn$ for which $(0,0)\in\clos{\gamma}$ (the closure of $\gamma$) is said to be \emph{tangential} if for all $a>0$ there exists $\eps>0$ such that 
\begin{equation}\label{tangentiality condition on curve}
    \set{(t,x)\in\gamma : 0<t\leq \eps}\subseteq \Hn\setminus\Gamma_{a}(0).
\end{equation}
For any $x_0\in\Rn$ let $\gamma+x_0=\set{(t,x+x_0)\in\Hn : (t,x)\in\gamma}$ be the corresponding translate of $\gamma$. 
The following result shows that Fatou's theorem is in some sense best possible.
\begin{thm}\label{Aikawa Littlewood theorem}
    Let $\gamma\subset\Hn$ be a tangential curve such that $(0,0)\in\clos{\gamma}$. There exists $f\in\Ell{\infty}$ with Poisson extension $u_{f}(t,x)=(P_{t}\ast f)(x)$ such that the limit
    \begin{equation*}
        \lim_{\gamma+x_0\ni (t,x)\to (0,x_0)} u_{f}(t,x)
    \end{equation*}
    does not exist for any $x_0\in\Rn$.
\end{thm}
This theorem follows from results due to Aikawa \cite{Aikawa_PAMS_Harmonic_Functions, Aikawa_JLMS_Harmonic_Functions_Green}. Much earlier, Littlewood \cite{Littlewood_JLMS_Fatou_Thm} obtained the weaker conclusion of almost everywhere non-existence of the limit (in the context of bounded holomorphic functions on the unit disk). See \cite{Lohwater_Piranian_boundary_behavior_1957} and \cite[Theorem~2.22]{Lohwater_Collingwood_Cluster_sets_1966} for similar results that predate the aforementioned work of Aikawa. Nagel and Stein \cite{Nagel_Stein_1984}, however, constructed approach regions which are not contained in any nontangential region $\Gamma_{a}(0)$ (which is weaker than the tangentiality condition \eqref{tangentiality condition on curve}) for which the conclusion of Fatou's theorem still holds. 

Theorem~\ref{Aikawa Littlewood theorem} essentially states that no degree of integrability of $f$ can help with the convergence properties of $u$ at the boundary along tangential approach paths. It is therefore natural to ask if any additional \emph{regularity} assumptions on the function $f$ can allow for the convergence of the Poisson extension $u_{f}(t,x)=(P_{t}\ast f)(x)$ through tangential approach regions. For any $\beta\in(0,1]$, $a>0$ and $x_{0}\in\Rn$ we follow \cite[Section 4]{Nagel_Stein_1984} and consider the \emph {tangential} approach region
 \begin{equation}\label{first mention of the tangential approach regions in upper half space}
     \Gamma^{\beta}_{a}(x_0)=\set{(t,x)\in\Hn : \abs{x-x_{0}}<at^{\beta} \textup{ if }0<t\leq 1, \textup{ and } \abs{x-x_0}<at \textup{ if } t\geq 1},
 \end{equation}
and for any measurable $u:\Hn\to\R$ the corresponding maximal function is given by
 \begin{equation}\label{general tangential maximal function on the tangential approach regions defined by Gamma beta}
     N_{*,\beta}^{a}(u)(x_0)=\sup_{(t,x)\in\Gamma^{\beta}_{a}(x_0)}\abs{u(t,x)},\quad  x_0\in\Rn.
 \end{equation} 
Naturally, only the behaviour of $u_f$ close to the boundary $\partial\Hn$ is of relevance to boundary convergence questions, and we point out that $\Gamma^{\beta}_{a}(x_0)$ coincides with the standard (conical) approach region $\Gamma_{a}(x_0)$ in the range where $t\geq 1$.
As may be expected, the aperture parameter $a>0$ is of little significance, and all of the results stated below hold regardless of its value; when $a=1$ we shall denote $\Gamma^{\beta}_{a}(x_0)=\Gamma^{\beta}(x_0)$, $\Gamma_{a}(x_0)=\Gamma(x_0)$ and $N_{*,\beta}^{a}=N_{*,\beta}$. 
In this context it turns out to be natural to impose regularity on the initial datum $f$ in the scale of the classical Bessel potential spaces $\mathscr{L}^{p}_{\alpha}(\Rn)$, which informally, consist of functions $f\in\Ell{p}$ that possess $\alpha$ derivatives in $\Ell{p}$; see Section~\ref{subsection with definitions of spaces} below for clarification. The following result was obtained by Nagel--Rudin--Shapiro \cite[Theorem~5.5]{Nagel_Rudin_Shapiro_1982}; see also Nagel--Stein \cite[Theorem~5]{Nagel_Stein_1984}.
\begin{thm}\label{theorem Nagel--Stein thm 5}
     Let $n\geq 1$, $p\in(1,\infty)$ and $\alpha\geq 0$ be such that $\alpha p <n$. If $\beta\geq 1-\frac{\alpha p}{n}$, $f\in\mathscr{L}^{p}_{\alpha}(\Rn)$ and $u_{f}(t,x)=(P_{t}\ast f)(x)$, then $\norm{N_{*,\beta}(u_{f})}_{\Ell{p}}\lesssim \norm{f}_{\mathscr{L}^{p}_{\alpha}(\Rn)}$. As a consequence,
     \begin{equation*}
         \lim_{\Gamma^{\beta}(x_0)\ni(t,x)\to (0,x_0)} u_{f}(t,x)=f(x_0) \quad \textup{for a.e. }x_{0}\in\Rn.
     \end{equation*}
 \end{thm} We remark that Theorem~\ref{theorem Nagel--Stein thm 5} can be extended to treat the case $\alpha p=n$, with the conclusions holding for all $\beta>0$ (see the proof of Lemma~\ref{lemma dimensional measures and Bessel potential functions} below). If $\alpha>\frac{n}{p}$, then any element of $\mathscr{L}^{p}_{\alpha}(\Rn)$ has a unique (Hölder) continuous representative vanishing at infinity (see, e.g., \cite[Theorem 1.2.4]{Adams_Hedberg_Potential_Theory}), and therefore $u_{f}\in C(\clos{\Hn})$ (see, e.g., \cite[Chapter \RNum{3}, Section 2]{Stein_Singular_Integrals}).

 Fatou's theorem asserts that the divergence set 
 \begin{equation*}
 \set{x_0\in\Rn : \lim_{\Gamma(x_0)\ni(t,x)\to (0,x_0)}u_{f}(t,x)\neq f(x_0)}
 \end{equation*}
 has zero Lebesgue measure. A natural question is whether additional regularity assumptions on $f$ can make this set smaller in some appropriate sense, such as that measured by the Hausdorff dimension $\dim_{H}$. The following result can be obtained by combining \cite[Theorem~3.13, Proposition~5.3]{Nagel_Rudin_Shapiro_1982} with a classical result relating Hausdorff measure and Bessel capacity (see, e.g., \cite[Theorem~2.6.16]{Ziemer_book_1989}). 
\begin{thm}\label{theorem hausdorff dimension of divergence set for classical poisson extension}
    Let $n\geq 1$, $p\in(1,\infty)$, $\alpha\in (0,\frac{n}{p}]$ and $\beta=1-\frac{\alpha p}{n}$. If $f\in\mathscr{L}^{p}_{\alpha}(\Rn)$ and $u_{f}(t,x)=(P_{t}\ast f)(x)$, then for any $\beta'\in(\beta,1]$, it holds that
    \begin{equation}\label{estimate on dimension of divergence set of classical Poisson extension of f}
        \dim_{H}\left(\set{x_0\in\Rn : \lim_{\Gamma^{\beta'}(x_0)\ni (t,x)\to(0,x_0)} u_{f}(t,x)\neq f(x_0)}\right)\leq n-n(\beta' -\beta).
    \end{equation}
\end{thm}
Note that the limiting case $\beta'=\beta$ is covered by Theorem~\ref{theorem Nagel--Stein thm 5} above. The right-hand side of~\eqref{estimate on dimension of divergence set of classical Poisson extension of f} is smallest when $\beta'=1$ and equals $n-\alpha p$. 
It is necessary to choose a preferred representative of the equivalence class $f\in\mathscr{L}^{p}_{\alpha}(\Rn)$ for the statement \eqref{estimate on dimension of divergence set of classical Poisson extension of f} to make sense, as different representatives can differ on sets of any Hausdorff dimension. We let this representative be the function $\tilde{f}:\Rn\to(-\infty, \infty]$ defined in~\eqref{preferred representative for bessel potential space function} in Section~\ref{subsection with definitions of spaces} below. 

\subsection{The Dirichlet problem for elliptic equations}\label{section on the dirichlet problem for elliptic equations}
The aim of this paper is to generalise Theorems~\ref{theorem Nagel--Stein thm 5} and \ref{theorem hausdorff dimension of divergence set for classical poisson extension} to solutions to the Dirichlet problem for second order divergence form elliptic equations with variable coefficients, under minimal regularity assumptions, in the upper half-space $\Hn$, and in bounded Lipschitz domains in $\R^{1+n}$. 
In general, for an open subset $\Omega\subseteq \R^{1+n}$ we let $A:\Omega\to \R^{(1+n)\times (1+n)}$ be a matrix of measurable, real-valued functions on $\Omega$. We assume that there are constants $0<\lambda\leq \Lambda<\infty$ such that the bound and uniform ellipticity 
\begin{equation}\label{boundedness aanndd ellipticity conditions}
    \abs{\langle A(X)\xi , \eta\rangle}\leq \Lambda \abs{\xi}\abs{\eta} \quad\textup{and}\quad \langle A(X)\xi , \xi\rangle \geq \lambda \abs{\xi}^{2}
\end{equation}
hold for all $\xi,\eta\in\R^{1+n}$ and all $X\in\Omega$.
Consider the elliptic equation 
\begin{equation}\label{formal elliptic equation in divergence form}
Lu=-\dvg(A\nabla u)= -\sum_{i=1}^{n+1}\sum_{j=1}^{n+1}\partial_{i}\left(A_{ij}\partial_{j}u\right)=0
\end{equation}
 where, as usual, a measurable function $u:\Omega\to\R$ is said to be a (weak) \emph{solution} to this equation in $\Omega$ if $u\in\mathrm{W}^{1,2}_{\loc}(\Omega)$ and $\int_{\Omega}\langle A\nabla u , \nabla \varphi\rangle \dd X =0$ for all smooth compactly supported functions $\varphi\in C^{\infty}_{c}(\Omega)$. 

It follows from the classical De Giorgi--Nash--Moser estimates that solutions of the equation $Lu=0$ are locally bounded and Hölder continuous in the sense that
\begin{equation}\label{local boundedness solutions DGNM}
    \norm{u}_{\El{\infty}(B)}\lesssim \left(\dashint_{2B}\abs{u(X)}^{2}\dd X\right)^{1/2},
\end{equation}
and there exists $\alpha_{L}>0$ such that 
\begin{equation}\label{holder continuity solutions DGNM}
    \abs{u(X)-u(Y)}\lesssim \left(\frac{\abs{X-Y}}{r}\right)^{\alpha_{L}}\left(\dashint_{2B}\abs{u(X)}^{2}\dd X\right)^{1/2}\quad \textup{ for all }X,Y\in B,
\end{equation}
for all balls $B$ of radius $r>0$ such that the dilate $2B\subseteq \Omega$, where $\alpha_{L}$ and the implicit constants only depend on $n$, $\lambda $ and $\Lambda$. 

Let us now consider the special case when $\Omega=\Hn=\set{(t,x) : t>0,  x\in\Rn}$, whose boundary is identified with $\Rn$. The nontangential maximal function $N_{*}u$ of a measurable function $u:\Hn\to\R$ is given by
\begin{equation}\label{definition of nontangential maximal function}
    (N_{*}u)(x_{0})=\sup_{(t,x)\in\Gamma(x_0)} \abs{u(t,x)},\quad  x_0\in\Rn.
\end{equation}
Given $p\in(1,\infty)$ and $f\in\Ell{p}$, the Dirichlet problem $\left(\mathcal{D}\right)^{L}_{p}$ is to find a function $u\in\mathrm{W}^{1,2}_{\loc}(\Hn)$ such that
\begin{equation*}
    \left(\mathcal{D}\right)^{L}_{p}\hspace{0.5cm}\left\{\begin{array}{l}
        Lu=0  \text{ in } \Hn;\\
        N_{*}(u)\in\Ell{p};\\
        \lim_{\Gamma(x_0)\ni(t,x)\to (0,x_{0})} u(t,x)=f(x_0) \text{ for a.e. } x_{0}\in\Rn.
    \end{array}\right.
\end{equation*}

The Dirichlet problem $\left(\mathcal{D}\right)^{L}_{p}$ is said to be \emph{well-posed} if for every boundary datum $f\in\Ell{p}$ there exists a unique solution $u$ with the three properties listed above. 

There exists a family of Borel probability measures $\set{\omega^{(t,x)}_{L}}_{(t,x)\in\Hn}$ on $\Rn$, known as the \emph{$L$-harmonic measures}, such that for any compactly supported continuous function $f\in C_{c}(\Rn)$, the function 
\begin{equation}\label{representation of classical solutions by harmonic measure 1st}
u_{f}(t,x)=\int_{\Rn}f(y)\dd \omega^{(t,x)}_{L}(y), \quad (t,x)\in\Hn
\end{equation}
is a solution (the unique bounded solution) to the classical Dirichlet problem with boundary datum $f$, in the sense that it is a weak solution to $Lu_{f}=0$ in $\Hn$ and that $u_{f}\in C(\clos{\Hn})$ with $u_{f}|_{\partial \Hn}= f$ (see, e.g., \cite[Section 5B]{HLM_Degenerate_Dirichlet_2019} in the more general context of degenerate elliptic equations).

We say that $(\mathcal{D})_{p}^{L}$ \emph{holds} if for all $f\in C_{c}(\Rn)$, the classical solution $u_f$ in \eqref{representation of classical solutions by harmonic measure 1st} satisfies 
    \begin{equation}\label{condition Dirichlet Lp holds}
        \norm{N_{*}u_{f}}_{\Ell{p}}\lesssim \norm{f}_{\Ell{p}}.
    \end{equation}
It is well-known that this implies the well-posedness of $(\mathcal{D})^{L}_{p}$ (see Lemma~\ref{lemma Dirichlet holds implies well posedness} below). 

Our first main result is the following generalisation of Theorems~\ref{theorem Nagel--Stein thm 5} and \ref{theorem hausdorff dimension of divergence set for classical poisson extension}.
\begin{thm}\label{main theorem upper half space}
    Let $n\geq 1$ and $p\in (1,\infty)$ be such that $\left(\mathcal{D}\right)^{L}_{p}$ holds. Let $\alpha > 0$ be such that $\alpha p\leq n$, and $\beta=1-\frac{\alpha p}{n}$. Let $f\in\mathscr{L}^{p}_{\alpha}(\Rn)$ and let $u_{f}$ be the unique solution to $\left(\mathcal{D}\right)^{L}_{p}$ with boundary datum $f$. For any $\beta'\in(\beta,1]$ it holds that
    \begin{equation}\label{upper bound hausdorff dimension of divergence set}
        \dim_{H}\left(\set{x_0\in\Rn : \lim_{\Gamma^{\beta'}(x_0)\ni (t,x)\to(0,x_0)} u_{f}(t,x)\neq f(x_0)}\right)\leq n-n(\beta' -\beta).
    \end{equation}
    In addition, if $\alpha p<n$, then 
    \begin{equation}\label{tangential maximal function estimate in statement of main theorem}
    \norm{N_{*,\beta}(u_{f})}_{\Ell{p}}\lesssim \norm{f}_{\mathscr{L}^{p}_{\alpha}(\Rn)}
    \end{equation}
    and 
    \begin{equation}\label{ae tangential convergence for sol of div form with bessel bdry datum}
        \lim_{\Gamma^{\beta}(x_0)\ni (t,x)\to(0,x_0)} u_{f}(t,x)=f(x_0) \quad \textup{ for a.e. }x_{0}\in\Rn.
    \end{equation}
    If $\alpha p=n$, then \eqref{tangential maximal function estimate in statement of main theorem} and \eqref{ae tangential convergence for sol of div form with bessel bdry datum} hold for all $\beta>0$.
\end{thm}
Just like Theorem~\ref{theorem Nagel--Stein thm 5}, this result holds more generally for the tangential approach regions $\Gamma^{\beta}_{a}(x_0)$ for arbitrary $a>0$.
Again, it is necessary to choose a preferred representative of the equivalence class $f\in\mathscr{L}^{p}_{\alpha}(\Rn)$ for the statement \eqref{upper bound hausdorff dimension of divergence set} to make sense. We choose this representative to be the function $\tilde{f}:\Rn\to(-\infty, \infty]$ defined in~\eqref{preferred representative for bessel potential space function} in Section~\ref{subsection with definitions of spaces} below. 

In order to extend Theorem~\ref{main theorem upper half space} to general bounded Lipschitz domains (see Theorem~\ref{second main theorem bounded Lipschitz domain} below), we shall require an analogue of this result for boundary data in the 
Sobolev--Slobodeckij spaces $\mathrm{W}^{s,p}(\Rn)$ (see Section~\ref{subsection with definitions of spaces} for the definition of these spaces). This is achieved by building on the results of Dorronsoro \cite{Dorronsoro_1986}, who extended the results of Nagel--Rudin--Shapiro \cite{Nagel_Rudin_Shapiro_1982} and Nagel--Stein \cite{Nagel_Stein_1984} (specifically Theorems~\ref{theorem Nagel--Stein thm 5} and \ref{theorem hausdorff dimension of divergence set for classical poisson extension} above) beyond the class $\mathscr{L}^{p}_{\alpha}(\Rn)$. More specifically, Dorronsoro considered certain fractional mean oscillation spaces $\mathrm{C}^{p}_{\alpha}(\Rn)\subseteq \Ell{p}$ (defined in Section~\ref{subsection with definitions of spaces}; see also \cite{Dorronsoro_1985}), and proved the following analogue of Theorem~\ref{theorem Nagel--Stein thm 5} (see \cite[Theorem 3]{Dorronsoro_1986}). 
\begin{thm}\label{Dorronsoros theorem}
    Let $n\geq 1$, $p\in(1,\infty)$ and $\alpha>0$ be such that $\alpha p<n$. If  $f\in\mathrm{C}^{p}_{\alpha}(\Rn)$ and $u_{f}(t,x)=(P_{t}\ast f)(x)$, then $\norm{N_{*,\beta}u_{f}}_{\Ell{p}}\lesssim \norm{f}_{\mathrm{C}^{p}_{\alpha}(\Rn)}$ for all $\beta\geq 1-\frac{\alpha p}{n}$. If $\alpha p=n$, then this holds for all $\beta>0$.
\end{thm}
The spaces $\mathrm{C}^{p}_{\alpha}(\Rn)$ were originally introduced by Calder\'{o}n and Scott in \cite{Calderon_Scott_1978} and were thoroughly studied by DeVore and Sharpley in \cite{Devore-Sharpley_Maximal_functions_smoothness_1984}. If $p\in(1,\infty)$ and $\alpha>0$, the Bessel potential space $\mathscr{L}^{p}_{\alpha}(\Rn)$ is continuously embedded in $\mathrm{C}^{p}_{\alpha}(\Rn)$ (see \cite[Theorem 7.5]{Devore-Sharpley_Maximal_functions_smoothness_1984}). More generally \cite[Proposition 3]{Dorronsoro_1986} shows that the Triebel--Lizorkin spaces $\mathrm{F}^{\alpha}_{p,q}(\Rn)$ are continuously embedded in $\mathrm{C}_{\alpha}^{p}(\Rn)$ for all $1\leq p,q<\infty $ and $\alpha>0$.
The spaces $\mathrm{C}^{p}_{\alpha}(\Rn)$ provide a natural setting to extend the conclusions of Theorem~\ref{main theorem upper half space}. In fact, we obtain the analogous results for boundary data $f\in\mathrm{C}^{p}_{\alpha}(\Rn)$ instead of $\mathscr{L}^{p}_{\alpha}(\Rn)$ in Theorem~\ref{main theorem upper half space for Cpalpha} in Section~\ref{section proof of main theorem for Wsp in upper half space}. Theorem~\ref{main theorem upper half space} can therefore be obtained as a corollary of this more general result. A more self-contained proof of Theorem~\ref{main theorem upper half space} that does not rely on the embeddings described above, however, is presented earlier in Section~\ref{Sect:tang}.

We now consider the case of a bounded Lipschitz domain $\Omega\subset \R^{1+n}$ with surface measure $\sigma$ on the boundary $\partial\Omega$.
For any $X_0\in\partial\Omega$ and $a>0$, we consider the standard nontangential approach region
\begin{equation*}
    \Gamma_{\Omega,a}(X_0)=\set{X\in\Omega : \abs{X-X_0}<(1+a)\dist(X,\partial\Omega)}.
\end{equation*}
For $u:\Omega\to \R$, the corresponding nontangential maximal function is given by 
\begin{equation*}
(N_{*}^{a}u)(X_0)=\sup_{X\in\Gamma_{\Omega,a}(X_0)}\abs{u(X)}, \quad X_0\in\partial\Omega.
\end{equation*}
Let us denote $\Gamma_{\Omega}(X_0)=\Gamma_{\Omega,1}(X_0)$ and $N_{*}u=N_{*}^{1}u$. 

As above, given $p\in(1,\infty)$ and $f\in\El{p}(\partial\Omega)$, the Dirichlet problem $\left(\mathcal{D}\right)^{L}_{p,\Omega}$ is to find a function ${u\in\mathrm{W}^{1,2}_{\loc}(\Omega)}$ such that
\begin{equation*}
    \left(\mathcal{D}\right)^{L}_{p,\Omega}\hspace{0.5cm}\left\{\begin{array}{l}
        Lu=0  \text{ in } \Omega;\\
        N_{*}u\in\El{p}(\partial\Omega);\\
        \lim_{\Gamma_{\Omega}(P)\ni X\to P} u(X)=f(P) \text{ for } \sigma\text{-a.e. } P\in\partial\Omega.
    \end{array}\right.
\end{equation*}
The Dirichlet problem $\left(\mathcal{D}\right)^{L}_{p,\Omega}$ is said to be \emph{well-posed} if for every boundary datum $f\in\El{p}(\partial\Omega)$ there exists a unique solution $u$ with the three properties listed above. 

There exists a family of Borel probability measures $\set{\omega^{X}_{L}}_{X\in\Omega}$ on $\partial\Omega$, known as the $L$-\emph{harmonic measures}, such that for any continuous function $f\in C(\partial\Omega)$, the function 
\begin{equation}\label{representation of classical solutions by harmonic measure bounded Lipschitz dom}
u_{f}(X)=\int_{\partial\Omega}f\dd \omega^{X}_{L}, \quad X\in\Omega,
\end{equation}
is a solution to the classical Dirichlet problem with boundary datum $f$, in the sense that it is a weak solution to $Lu_{f}=0$ in $\Omega$ and that $u_{f}\in C(\clos{\Omega})$ with $u_{f}|_{\partial\Omega}= f$ (see, e.g. \cite[Chapter 1, Section 2]{Kenig_cbms_1994}).

In this context, we say that $(\mathcal{D})_{p,\Omega}^{L}$ \emph{holds} if for all $f\in C(\partial\Omega)$, the classical solution $u_f$ in \eqref{representation of classical solutions by harmonic measure bounded Lipschitz dom} satisfies 
    \begin{equation*}
        \norm{N_{*}u_{f}}_{\El{p}(\partial\Omega)}\lesssim \norm{f}_{\El{p}(\partial\Omega)}.
    \end{equation*}
Again, it is well-known that this implies the well-posedness of $(\mathcal{D})^{L}_{p,\Omega}$ (see, e.g., \cite[Theorem 1.7.7]{Kenig_cbms_1994}).

The analogue of the tangential approach regions $\Gamma^{\beta}(x_0)$ in $\Omega$ are now defined as follows. For $X_{0}\in\partial \Omega$, $\beta\in(0,1]$ and $c>0$, we let
\begin{equation*}
    \Gamma^{\beta,c}_{\Omega}(X_0)=\set{X\in \Omega : \abs{X-X_{0}}< \left\{ 
    \begin{array}{ll}
         (1+c)\dist(X,\partial \Omega)^{\beta}, & \textup{if }\dist(X,\partial \Omega)\leq 1;  \\
         (1+c)\dist(X,\partial \Omega), & \textup{if }\dist(X,\partial \Omega)\geq 1  
    \end{array}\right.\;\;\;},
\end{equation*}
and for a measurable $u:\Omega\to \R$, the corresponding maximal function is defined by
\begin{equation*}
    (N_{*,\beta,c}u)(X_0)=\sup_{X\in\Gamma^{\beta,c}_{\Omega}(X_0)}\abs{u(X)},
    \quad X_{0}\in\partial \Omega.
\end{equation*}
Some of these approach regions are sketched in Figure~\ref{tangential_approach_regions_label} below, where only the local part (close to the boundary) is included for clarity. 

\begin{figure}[H]
    \centering
    \includegraphics[scale=2.5]{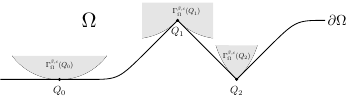}
    \caption{A part of the boundary $\partial\Omega$ of a Lipschitz domain $\Omega$ and some tangential approach regions ($n=1$, $\beta=\tfrac{1}{2}$, $c=1$).}
    \label{tangential_approach_regions_label}
\end{figure}

\begin{rem}
    The definition of the approach regions $\Gamma^{\beta,c}_{\Omega}(X_0)$ is, of course, not limited to the case of bounded Lipschitz domains. If $\Omega=\Hn$, then a comparison with the regions defined in \eqref{first mention of the tangential approach regions in upper half space} shows that $\Gamma^{\beta}_{c}(x_0)\subseteq \Gamma_{\Omega}^{\beta,c}(X_0)\subseteq \Gamma^{\beta}_{1+c}(x_0)$ when $X_0=(0,x_0)\in\partial\Hn$ and $x_0\in\Rn$. Besides, observe that a certain order of tangential approach is permitted at a cusp singularity, whereas nontangential approach is prohibited.
\end{rem}

In the context of the bounded Lipschitz domains $\Omega$ that we consider here, we shall quantify the regularity of the boundary data in terms of the Sobolev spaces $\mathrm{W}^{s,p}(\partial\Omega)$. These may be defined for smoothness exponents $0\leq s\leq k+1$ when the domain $\Omega$ is of class $C^{k,1}$ for an integer $k\geq 0$ (the case $k=0$ corresponding to Lipschitz domains). This is discussed in Section~\ref{subsection on sobolev spaces on boundary of Lipschitz domains}.

Our second main result is the following extension of Theorem~\ref{main theorem upper half space}.
\begin{thm}\label{second main theorem bounded Lipschitz domain}
    Let $n\geq 1$ and let $\Omega\subset\R^{1+n}$ be a bounded domain of class $C^{k,1}$ for some integer $k\geq 0$. Let $p\in(1,\infty)$ be such that $\left(\mathcal{D}\right)^{L}_{p,\Omega}$ holds. Let $0\leq s\leq k+1$ be such that $sp\leq n$, let $\beta=1-\frac{sp}{n}$ and $c>0$. If $f\in\mathrm{W}^{s,p}(\partial\Omega)$ and $u_{f}$ is the unique solution to $\left(\mathcal{D}\right)^{L}_{p,\Omega}$ with boundary datum $f$, then for any $\beta'\in(\beta,1]$ it holds that
    \begin{equation}\label{upper bound hausdorff dimension of divergence set boundary lipschitz domain}
        \dim_{H}\left(\set{Q_{0}\in\partial\Omega : \lim_{\Gamma^{\beta',c}_{\Omega}(Q_0)\ni X\to Q_0} u_{f}(X)\neq f(Q_0)}\right)\leq n-n(\beta'-\beta).
    \end{equation}
    In addition, if $sp<n$, then
    \begin{equation}\label{strong Lp bound tangential max at the boundary of Lipschitz domain}
        \norm{N_{*,\beta,c}(u_{f})}_{\El{p}(\partial\Omega)}\lesssim \norm{f}_{\mathrm{W}^{s,p}(\partial\Omega)}
    \end{equation}
and
\begin{equation}\label{convergence sigma a.e. at the boundary of Lipschitz domain}
         \lim_{\Gamma^{\beta,c}_{\Omega}(Q_0)\ni X\to Q_0} u_{f}(X)=f(Q_0) \quad \textup{ for }\sigma\textup{-a.e. }Q_0\in\partial\Omega.
     \end{equation}
If $sp=n$, then \eqref{strong Lp bound tangential max at the boundary of Lipschitz domain} and \eqref{convergence sigma a.e. at the boundary of Lipschitz domain} hold for all $\beta>0$. 
\end{thm}
The estimate~\eqref{upper bound hausdorff dimension of divergence set boundary lipschitz domain} will be shown to hold for a specific representative $\tilde{f}:\partial\Omega\to(-\infty,\infty]$ of the function $f\in\mathrm{W}^{s,p}(\partial\Omega)$ (see \eqref{preferred representative for sobolev on boundary lipschitz} at the start of Section~\ref{section second part of the proof of second main thm}).

It is well-known that the bound and ellipticity of the coefficients $A$ in \eqref{boundedness aanndd ellipticity conditions} are not sufficient for solvability of $(\mathcal{D})_{p,\Omega}^{L}$, even on smooth domains (see \cite{CFK_Singular_harmonic_measure_1981}).
We conclude this introduction by mentioning some conditions which ensure the well-posedness of the Dirichlet problem. We start with the upper half-space $\Hn$. 
If the coefficients $A:\Hn\to \R^{(1+n)\times (1+n)}$ are $t$-independent and satisfy~\eqref{boundedness aanndd ellipticity conditions}, it has been shown by Hofmann--Kenig--Mayboroda--Pipher that there exists $p_{0}\in(1,\infty)$ such that $(\mathcal{D})^{L}_{p}$ holds for all $p\geq p_0$ (see \cite[Theorem 1.23]{HKMP_2015}). If $A_{0},A_{1}:\Hn\to\R^{(1+n)\times (1+n)}$ are coefficients satisfying \eqref{boundedness aanndd ellipticity conditions}, let their difference function be given by
\begin{equation}
    a(X)=\sup_{\abs{X-Y}<\delta(X)/2}\abs{A_{0}(Y)-A_{1}(Y)}
\end{equation}
for all $X\in\Hn$, where $\delta(X)=\dist(X,\partial\Hn)$. Let $L_{0}=-\dvg(A_{0}\nabla)$ and $L_{1}=-\dvg(A_{1}\nabla)$ denote the corresponding elliptic operators. 
If there is $C>0$ such that 
\begin{equation}\label{small perturbation condition on the coefficients}
    \frac{1}{\meas{Q}}\iint_{T(Q)}\frac{a(X)^{2}}{\delta(X)}\dd X \leq C
\end{equation}
for all surface cubes $Q\subset \Rn$ (where $T(Q)=(0,\ell(Q))\times Q$ is the corresponding Carleson box), then R. Fefferman, Kenig and Pipher have shown that there is $p_0\in(1,\infty)$ such that $(\mathcal{D})^{L_{0}}_{p}$ holds for all $p\geq p_{0}$ if and only if there is $p_1\in(1,\infty)$ such that $(\mathcal{D})^{L_{1}}_{p}$ holds for all $p\geq p_1$ (see \cite[Theorem 2.3]{FKP_1991} and note, as is well-known, that the symmetry assumption on the coefficients made therein is not required).

For the Laplacian $\Delta$ in a bounded Lipschitz domain $\Omega\subset\R^{1+n}$, we mention the classical theorem of Dahlberg \cite{Dahlberg_harmonic_measure_1977, Dahlberg_Poisson_integral_1979} stating that there exists $\eps>0$ such that the Dirichlet problem $(\mathcal{D})^{\Delta}_{p}$ holds for all $2-\eps <p\leq\infty$. As above, the well-posedness of the $\El{p}$ Dirichlet problem is preserved under small perturbations of the coefficients, in the way quantified by \eqref{small perturbation condition on the coefficients}, see \cite[Theorem~2.3]{FKP_1991} and references therein.

\subsubsection*{Structure of the paper} 
The remaining sections are organised as follows.
\begin{itemize}
\item
Section 2 consists of some preliminaries. We introduce the notation that will be used throughout in Section~\ref{section on notations}, the relevant function spaces in Section~\ref{subsection with definitions of spaces}, and describe the pertinent properties of fractional dimensional measures in Section~\ref{subsubsection on fractiona dimensional measures}. 
\item Section 3 is devoted to the upper half space and the proof of our first main theorem -- Theorem ~\ref{main theorem upper half space}. Several well-known properties of the $\El{p}$-Dirichlet problem in $\Hn$ are collected in Section~\ref{section on nontangential approach dirichlet upper half space + estimate div set}, which also includes a discussion on the dimension of the divergence set for the classical nontangential approach. Some further pertinent properties of the Bessel potential spaces $\mathscr{L}^p_{\alpha}(\Rn)$ are then collected in Section~\ref{section further properties bessel}, before introducing some relevant auxiliary maximal functions and proving Theorem~\ref{main theorem upper half space} in Section~\ref{Sect:tang}. 
Section~\ref{section proof of main theorem for Wsp in upper half space} extends Theorem~\ref{main theorem upper half space} to boundary data in the fractional mean oscillation spaces $\mathrm{C}^{p}_{\alpha}(\Rn)$.
\item
Section~\ref{subsection proof second main theorem} is concerned with the more general (Lipschitz) domains $\Omega$ and the proof of our second main theorem -- Theorem~\ref{second main theorem bounded Lipschitz domain}. We recall the relevant elementary properties of Lipschitz domains in Section~\ref{subsection with definition of Lipschitz domains}, followed by some well-known properties of the $\El{p}$-Dirichlet problem in bounded Lipschitz domains in Section~\ref{subsection on L harmonic measure in Lipschitz domains}, before introducing the Sobolev spaces $\textup{W}^{s,p}(\partial\Omega)$ on the boundary in Section~\ref{subsection on sobolev spaces on boundary of Lipschitz domains}. Some elementary but important lemmas are collected in Sections~\ref{section on corkscrew points}, \ref{section on pointwise estimates on solutions} and \ref{section on tangential approach regions in lipschitz domains}, before turning to the proof of Theorem~\ref{second main theorem bounded Lipschitz domain}, which is split into two parts. The tangential maximal function bound \eqref{strong Lp bound tangential max at the boundary of Lipschitz domain} and the corresponding almost everywhere convergence~\eqref{convergence sigma a.e. at the boundary of Lipschitz domain} are proved in Section~\ref{section first part of the proof of second main thm}, and the estimate~\eqref{upper bound hausdorff dimension of divergence set boundary lipschitz domain} on the dimension of the divergence set is proved in Section~\ref{section second part of the proof of second main thm}. \item Section~\ref{Sect:context} provides some contextual remarks on tangential convergence problems in the alternative settings of parabolic, dispersive and wave equations.
\end{itemize}

\section{Preliminaries}
\subsection{Notation}\label{section on notations}
Throughout $n\geq 1$ is an integer, and we use the convention that $\N=\set{0,1,\ldots}$.
We let $\abs{x}$ denote the Euclidean norm of an element $x\in\Rn$, 
$\Delta(x,r)$ denote the Euclidean ball with centre $x$ and radius $r$,
and $\langle x,y\rangle$ denote the standard scalar product of vectors $x,y\in\Rn$. We use the notation $x\vee y=\max\set{x,y}$ and $x\wedge y=\min\set{x,y}$.
As usual, the distance between two subsets $E,F\subseteq \Rn$ is $\dist(E,F)=\inf\set{\abs{x-y} : x\in E, y\in F}$, and for $x\in\mathbb{R}$, we write $\dist(x,F)=\dist(\{x\},F)$. The closure of a set $A\subseteq \Rn$ is denoted by $\clos{A}$, and the indicator (or characteristic) function of such a set is denoted by $\ind{A}$. 
The upper half-space $\Hn$ is defined as 
$\Hn=\set{(t,x)\in\R\times \R^{n}: t>0}$.
The $\ell^{\infty}$ norm of a vector $x=(x_1,\ldots, x_n)\in\Rn$ is denoted by $\abs{x}_{\infty}=\sup_{1\leq i\leq n}\abs{x_i}$, and the cube $Q$ of centre $x_Q\in\Rn$ and sidelength $\ell(Q)=\ell>0$ is the set $Q=\set{x\in\Rn : \abs{x-x_{Q}}_{\infty}<\ell/2}$. The point $X_{Q}=(\ell(Q),x_{Q})\in\Hn$ denotes the corresponding corkscrew point in $\Hn$.

If the Lebesgue measure of $A\subset\Rn$, denoted $\meas{A}$, is finite, we let $$\left(f\right)_{A}=\dashint_{A}f=\frac{1}{\meas{A}}\int_{A}f(x)\dd x$$ denote the average of an integrable function $f: A\to \R$ over the set $A$. For a locally integrable function $f\in\El{1}_{\loc}(\Rn)$, its centred Hardy--Littlewood maximal function is defined for all $x\in\Rn$ by
\begin{equation*}
    Mf(x)=\sup_{r>0}\dashint_{\Delta(x,r)}\abs{f(y)}\dd y.
\end{equation*}
For $p\in[1,\infty]$, we let $\El{p}_{+}(\Rn)$ denote the space of functions $f\in\Ell{p}$ such that $f(x)\geq 0$ for almost every $x\in\Rn$.
If $O\subseteq \Rn$ is open, $f\in\El{1}_{\loc}(O)$ and $\alpha\in\N^{n}$ is a multi-index, then the $\alpha$th partial (distributional) derivative of $f$ is denoted $\partial^{\alpha}f$. In the case when $\alpha=(0,\ldots,1,\ldots,0)$ (where $1$ is in the $k$th component), we let $\partial_{k}f=\partial^{\alpha}f$.

For a normed space $X$, we let $\mathcal{L}(X)$ denote the space of all bounded linear operators ${T:X\to X}$ equipped with the usual operator norm $\norm{T}_{\mathcal{L}(X)}$. The class of Schwartz functions is denoted by $\mathcal{S}(\Rn)$, and the space of tempered distributions is denoted by $\mathcal{S}'(\Rn)$.

If $x,y\geq 0$, we say that $x\lesssim y$ if there exists $C>0$, independent of both $x$ and $y$, such that $x\leq Cy$. We define $x\gtrsim y$ in a similar way, and use the notation $x\eqsim y $ to express that both $x\lesssim y$ and $x\gtrsim y$ hold. We shall occasionally use the symbol $\lesssim_{a}$ to indicate that the implicit constant depends on some parameter $a$, or directly denote the constant as $C_{a}$.

\subsection{Sobolev spaces}\label{subsection with definitions of spaces} In this section we introduce the various Sobolev (or smoothness) spaces referred to in the introduction in the order that they appear -- the Bessel potential spaces $\mathscr{L}^{p}_{\alpha}(\Rn)$, the fractional mean oscillation spaces $\mathrm{C}^{p}_{\alpha}(\Rn)$, and the Sobolev--Slobodeckij spaces $\mathrm{W}^{s,p}$. For each we collect some basic facts that we shall appeal to in subsequent sections.
\subsubsection*{The Bessel potential spaces} A natural criterion that captures the smoothness of a function in $\Ell{p}$ is expressed in terms of the smoothing operators $\mathscr{J}_{\alpha}=(I-\Delta)^{-\alpha/2}$ for $\alpha\geq 0$: a function $f$ belongs to the Bessel potential space $\mathscr{L}^{p}_{\alpha}(\Rn)$ if there exists $g\in\Ell{p}$ such that $f=\mathscr{J}_{\alpha}g$. The operator $\mathscr{J}_{\alpha}$ is given by convolution with the \textit{Bessel kernel} $G_{\alpha}:\Rn\to[0,\infty]$, defined via the Fourier transform by $\widehat{G}_{\alpha}(\xi)=(1+4\pi^{2}\abs{\xi}^{2})^{-\alpha/2}$, or spatially by
\begin{equation}\label{bessel kernel}
    G_{\alpha}(x)=c_{\alpha}\int_{0}^{\infty}\exp\left({-\pi\frac{\abs{x}^{2}}{t}}{-\frac{t}{4\pi}}\right)t^{-\frac{(n-\alpha)}{2}}\frac{\dd t}{t};
\end{equation}
here $c_{\alpha}>0$ ensures that $\norm{G_{\alpha}}_{\El{1}}=1$. 
It follows that $G_{\alpha}\in C^{\infty}(\Rn\setminus\set{0})\cap \Ell{1}$ and is nonnegative for each $\alpha>0$, and 
 $\norm{\mathscr{J}_{\alpha}(g)}_{\Ell{p}}\leq \norm{g}_{\Ell{p}}$ for all $p\in [1,\infty]$ by Minkowski's inequality. It also follows from the representation \eqref{bessel kernel} that the Bessel kernel exhibits exponential decay at infinity, and in particular
\begin{equation}\label{estimate on (gradient) of bessel potential}
    \abs{G_{\alpha}(x)}\lesssim \abs{x}^{-(n-\alpha)}e^{-\abs{x}/2}, \quad \abs{\nabla G_{\alpha}(x)}\lesssim \abs{x}^{-(n-\alpha+1)}e^{-\abs{x}/2},\quad x\in\Rn\setminus \set{0}.
\end{equation}
Of course $\mathscr{L}^{p}_{\alpha}(\Rn)$, being a subspace of $\Ell{p}$, consists of equivalence classes of functions $f\in\Ell{p}$, and the $\mathscr{L}^{p}_{\alpha}$-norm of $f\in\mathscr{L}^{p}_{\alpha}(\Rn)$ is naturally defined by
$\norm{f}_{\mathscr{L}^{p}_{\alpha}(\Rn)}=\norm{g}_{\Ell{p}}$ if $f=\mathscr{J}_{\alpha}(g)$.

As outlined in the introduction, it is necessary to choose a preferred representative of the equivalence class $f\in\mathscr{L}^{p}_{\alpha}(\Rn)$ for the statements of Theorems~\ref{theorem hausdorff dimension of divergence set for classical poisson extension} and \ref{main theorem upper half space} to make sense (specifically for \eqref{estimate on dimension of divergence set of classical Poisson extension of f} and \eqref{upper bound hausdorff dimension of divergence set}). If $\alpha>0$ and $p\in(1,\infty)$, a reasonable representative of $f=G_{\alpha}\ast g\in\mathscr{L}^{p}_{\alpha}(\Rn)$ is the function $\tilde{f} :\Rn \to (-\infty,\infty]$ defined by 
\begin{equation}\label{preferred representative for bessel potential space function}
    \tilde{f}(x)=\left\{\begin{array}{ll}
    \lim_{r\to 0^{+}}\dashint_{\Delta(x,r)} f(y)\dd y, & \textup{if the limit exists in }\R;\\
    \infty, & \textup{otherwise}.
    \end{array}\right.
\end{equation}
Indeed, \cite[Proposition 6.1.3]{Adams_Hedberg_Potential_Theory} shows that $\lim_{r\to 0^{+}}\dashint_{\Delta(x,r)} f(y)\dd y = (G_{\alpha}\ast g)(x)$ for all $x\in\Rn$ such that $(G_{\alpha}\ast \abs{g})(x)<\infty$, that is, away from a set of Hausdorff dimension no greater than $n-\alpha p$; see~\eqref{hausdorff dimension singularity set of non negative bessel potential function} in Section~\ref{subsubsection on fractiona dimensional measures} below. We note that $\smallabs{\tilde{f}(x)}\leq (G_{\alpha}\ast \abs{g})(x)$ for all $x\in\Rn$.

In this setting we also define the \textit{Riesz kernel}, or fractional integral kernel, $I_{\alpha}:\Rn\to[0,\infty]$ by $\widehat{I}_{\alpha}(\xi)=(2\pi\abs{\xi})^{-\alpha}$, or equivalently by
\begin{equation}\label{riesz kernel}
    I_{\alpha}(x)=c_{\alpha}\int_{0}^{\infty}e^{{-\pi\frac{\abs{x}^{2}}{t}}}t^{-\frac{(n-\alpha)}{2}}\frac{\dd t}{t}=\gamma_{\alpha,n}\abs{x}^{-(n-\alpha)}
\end{equation}
for a suitable constant $\gamma_{\alpha,n}>0$. It follows from \eqref{bessel kernel} and \eqref{riesz kernel} that $0<G_{\alpha}(x)\leq I_{\alpha}(x)$ for all $x\in\Rn$, and that $\lim_{x\to 0}\frac{G_{\alpha}(x)}{I_{\alpha}(x)}=1$ whenever $0<\alpha<n$. The reader may refer to \cite[Chapter \RNum{5}, Section 3]{Stein_Singular_Integrals} for more on Riesz and Bessel potentials.

\subsubsection*{The fractional mean oscillation spaces $\mathrm{C}^{p}_{\alpha}(\Rn)$} Another, somewhat less common, criterion to measure the smoothness of functions in $\Ell{p}$ is by quantifying their local approximability by polynomials. 
For integers $k\geq 0$ we let $\bb{P}_{k}$ denote the space of real polynomials on $\Rn$ with degree at most $k$. For $\alpha>0$ and $f\in\El{1}_{\loc}(\Rn)$ consider the maximal function
\begin{equation*}
    S_{\alpha}f(x)=\sup_{\Delta\ni x} \inf_{P\in\bb{P}_{k}}\frac{1}{\meas{\Delta}^{\alpha/n}}\dashint_{\Delta}\abs{f(y)-P(y)}\dd y, 
\end{equation*}
where the supremum is taken over all balls $\Delta\subset\Rn$ containing $x\in\Rn$, and where here and in what follows $k=\lfloor \alpha\rfloor$ is the greatest integer less than or equal to $\alpha$. 
The maximal functions $S_{\alpha}f$ measure the smoothness of $f$. 
For example, when $\alpha\in(0,1)$, we have the elementary inequality 
\begin{equation*}
    \abs{f(x)-f(y)}\lesssim \abs{x-y}^{\alpha}(S_{\alpha}f(x) + S_{\alpha}f(y))
\end{equation*}
for almost every $x,y\in\Rn$ (see, e.g., \cite[Theorem 2.5]{Devore-Sharpley_Maximal_functions_smoothness_1984}).

The maximal function $S_{\alpha}f$ can be defined using a minimising polynomial. 
The space $\bb{P}_{k}$ is a Hilbert space with the $\El{2}$ inner product $\langle f ,g\rangle=\int_{\Delta(0,1)}f(x)g(x)\dd x$. Starting from the set $\set{x^{\gamma}}_{\abs{\gamma}\leq k}$ of all monomials in $\bb{P}_{k}$, arranged in lexicographic order, the Gram--Schmidt process produces an orthonormal basis $\set{\varphi_{\gamma}}_{\abs{\gamma}\leq k}$ of $\bb{P}_{k}$. The operator $P^{k}_{\Delta(0,1)}:\El{1}_{\loc}(\Rn)\to \bb{P}_{k}$ defined by 
\begin{equation*}
    P_{\Delta(0,1)}^{k}f(x)=\sum_{\abs{\gamma}\leq k}\langle f, \varphi_{\gamma}\rangle\varphi_{\gamma}(x), \quad x\in\Rn,
\end{equation*}
is a projection (by construction) and it clearly satisfies $\smallnorm{P_{\Delta(0,1)}^{k}f}_{\El{\infty}(\Delta(0,1))}\lesssim \int_{\Delta(0,1)}\abs{f}$. For a general ball $\Delta=\Delta(x_0,r)\subset\Rn$, we define the corresponding projection $P_{\Delta}^{k}:\El{1}_{\loc}(\Rn)\to \bb{P}_{k}$ by 
\begin{equation*}
P_{\Delta}^{k}f=P_{\Delta(0,1)}(f\circ A)\circ A^{-1},
\end{equation*}
where $A:\Rn\to\Rn:x\to x_0+rx$. It follows by a change of variables that 
\begin{equation}\label{local L infty estimate on minimising polynomials}
    \smallnorm{P_{\Delta}^{k}f}_{\El{\infty}(\Delta)}\lesssim \dashint_{\Delta}\abs{f}.
\end{equation}
The definition shows that $P_{\Delta}^{k}f$ is the unique polynomial in $\bb{P}_{k}$ such that 
\begin{equation*}
\dashint_{\Delta}(f(y)-P_{\Delta}^{k}f(y))y^{\gamma}\dd y=0
\end{equation*}
for all multi-indices $\gamma\in\N^{n}$ with $\abs{\gamma}\leq k$. In particular, $P^{0}_{\Delta}f=\dashint_{\Delta}f$.

It can be seen using \eqref{local L infty estimate on minimising polynomials} that the maximal function
\begin{equation*}
    f_{\alpha}^{\sharp}(x)=\sup_{\Delta\ni x}\frac{1}{\meas{\Delta}^{\alpha/n}}\dashint_{\Delta}\smallabs{f(y)-P_{\Delta}^{k}f(y)}\dd y
\end{equation*}
is such that $S_{\alpha}f(x)\eqsim f_{\alpha}^{\sharp}(x)$, see \cite[Lemma 2.1]{Devore-Sharpley_Maximal_functions_smoothness_1984}. 
For $p\in[1,\infty]$ and $\alpha>0$, the space $\mathrm{C}^{p}_{\alpha}(\Rn)$ is defined as the subspace of $\Ell{p}$ consisting of the (equivalence classes of) functions $f\in\Ell{p}$ such that $f_{\alpha}^{\sharp}\in\Ell{p}$. It is equipped with the norm 
\begin{equation*}
    \norm{f}_{\mathrm{C}^{p}_{\alpha}(\Rn)}=\smallnorm{f}_{\Ell{p}} + \smallnorm{f^{\sharp}_{\alpha}}_{\Ell{p}},
\end{equation*}
which makes it into a Banach space (see \cite[Lemma 6.1]{Devore-Sharpley_Maximal_functions_smoothness_1984}); we clarify that $f_0^\sharp$ is the classical sharp maximal function of C. Fefferman and Stein, and $\norm{f}_{\mathrm{C}^{p}_{0}(\Rn)}\eqsim\smallnorm{f}_{\Ell{p}}$.  Just as with the Bessel potential spaces $\mathscr{L}^{p}_{\alpha}(\Rn)$, if $\alpha>\frac{n}{p}$ and $f\in\mathrm{C}^{p}_{\alpha}(\Rn)$, then $f$ has a continuous representative (see \cite[Theorem 9.1]{Devore-Sharpley_Maximal_functions_smoothness_1984}).

We also consider the following closed subspace of $\mathrm{C}^{p}_{\alpha}(\Rn)$: 
\begin{equation*}
\mathrm{F}^{p}_{\alpha}(\Rn)=\set{f\in\mathrm{C}^{p}_{\alpha}(\Rn) : \lim_{t\to 0^{+}}\smallnorm{f_{\alpha}^{\sharp}(t,\cdot)}_{\Ell{p}}=0},
\end{equation*}
where
\begin{equation*}
    f_{\alpha}^{\sharp}(t,x)=\sup_{\substack{\Delta\ni x\\ \diam{\Delta}\leq t}}\frac{1}{\meas{\Delta}^{\alpha/n}}\dashint_{\Delta}\smallabs{f(y)-P_{\Delta}^{k}f(y)}\dd y\quad \textup{ for all }(t,x)\in\Hn.
\end{equation*}
The space $\mathrm{F}^{p}_{\alpha}(\Rn)$ coincides with the closure of the smooth compactly supported functions $C^{\infty}_{c}(\Rn)$ in $\mathrm{C}^{p}_{\alpha}(\Rn)$ (see \cite[Proposition 1]{Dorronsoro_1986}).
The reader can refer to the monograph of DeVore and Sharpley \cite{Devore-Sharpley_Maximal_functions_smoothness_1984} for an extensive study of the spaces $\mathrm{C}^{p}_{\alpha}(\Rn)$. We note that in \cite{Devore-Sharpley_Maximal_functions_smoothness_1984, Dorronsoro_1986} the maximal functions $S_{\alpha}f$ and $f^{\sharp}_{\alpha}$ are defined as suprema over cubes $Q\ni x $ rather than Euclidean balls $\Delta\ni x$. This obviously leads to pointwise equivalent maximal functions, and so the resulting spaces $\mathrm{F}^{p}_{\alpha}(\Rn)$ and $\mathrm{C}^{p}_{\alpha}(\Rn)$ are identical as sets with equivalent norms.

\subsubsection*{The Sobolev--Slobodeckij spaces $\mathrm{W}^{s,p}$} As we shall see, the spaces $\mathscr{L}^{p}_{\alpha}(\Rn)$ and $\mathrm{C}^{p}_{\alpha}(\Rn)$ are very natural from the point of view of our pointwise convergence problems in the setting of the upper half space. However, their definitions do not appear to extend in a convenient way to the setting of general open subsets $U\subseteq\Rn$. A more suitable approach (for our purposes at least) is to define fractional Sobolev spaces in domains as follows.  If $p\in[1,\infty)$ and $m\geq 0$ is an integer, the Sobolev space $\mathrm{W}^{m,p}(U)$ is defined as usual by
\begin{equation*}
    \mathrm{W}^{m,p}(U)=\set{f\in \El{p}(U) : \partial^{\alpha}f \in\El{p}(U) \textup{ for all multi-indices}\abs{\alpha}\leq m},
\end{equation*}
equipped with the norm $\norm{f}_{\mathrm{W}^{m,p}(U)}=\left(\sum_{\abs{\alpha}\leq m}\norm{\partial_{\alpha}f}_{\El{p}(U)}^{p}\right)^{1/p}$. The Sobolev spaces of fractional order are defined using the Slobodeckij seminorms given by
\begin{equation*}
    \left[f\right]_{\sigma,p,U}=\left(\int_{U}\int_{U}\frac{\abs{f(x)-f(y)}^{p}}{\abs{x-y}^{n+\sigma p}}\dd x\dd y\right)^{1/p}\quad \textup{ for }\sigma\in(0,1).
\end{equation*}
If $s=m+\sigma$, where $m\geq 0$ is an integer and $\sigma\in (0,1)$, then we define 
\begin{equation*}
    \mathrm{W}^{s,p}(U)=\set{f \in\mathrm{W}^{m,p}(U) : \left[ \partial^{\alpha}f\right]_{\sigma, p ,U}<\infty \textup{ whenever }\abs{\alpha}=m},
\end{equation*}
and equip this space with the norm $\norm{f}_{\mathrm{W}^{s,p}(U)}=\left(\norm{f}_{\mathrm{W}^{m,p}(U)}^{p} + \sum_{\abs{\alpha}=m}\left[ \partial_{\alpha}f\right]_{\sigma, p,U}^{p}\right)^{1/p}$.

If $k\geq 1$ is a positive integer and $p\in(1,\infty)$, then $\mathscr{L}^{p}_{k}(\Rn)=\mathrm{W}^{k,p}(\Rn)$ with equivalent norms (see, e.g., \cite[Chapter \RNum{5}, Theorem 3]{Stein_Singular_Integrals}).
Furthermore, the continuous inclusion $\mathrm{W}^{s,p}(\Rn)\subseteq \mathscr{L}^{p}_{s}(\Rn)$ holds for all $s > 0$ if $p\in(1,2]$, and the converse inclusion holds in the range $p\in[2,\infty)$; see, for example, \cite[Chapter \RNum{5}, Theorem~5]{Stein_Singular_Integrals}.

\subsection{Fractional dimensional measures and Frostman's lemma}\label{subsubsection on fractiona dimensional measures}
For $s\in [0,\infty)$, a positive Borel measure $\mu$ on $\Rn$ is called \textit{$s$-dimensional} if 
\begin{equation*}
    c_{s}(\mu)=\sup_{x\in\Rn, r>0} \frac{\mu(\Delta(x,r))}{r^{s}}<\infty.
\end{equation*}
We let $\mathcal{M}^{s}(\Rn)$ denote the collection of all $s$-dimensional probability measures on $\Rn$.

The following lemma shows the relevance of $\mathcal{M}^{s}(\Rn)$ in the study of the Bessel potential spaces $\mathscr{L}^{p}_{\alpha}(\Rn)$.  This result follows from \cite[Theorem 7.2.2]{Adams_Hedberg_Potential_Theory} and a classical estimate on the Bessel capacity of balls. For completeness we shall prove this in the context of general fractional dimensional measures. See also \cite[Theorem~4.7.2]{Ziemer_book_1989} for some related estimates. 
\begin{lem}\label{lemma dimensional measures and Bessel potential functions}
    Let $p\in(1,\infty)$, $\alpha\in (0,\frac{n}{p}]$ and $s>n-\alpha p$. If $g\in\El{p}_{+}(\Rn)$ and $\mu\in\mathcal{M}^{s}(\Rn)$, then $\norm{G_{\alpha}\ast g}_{\El{1}(\mu)}\lesssim (c_{s}(\mu)^{1/p}\vee 1)\norm{g}_{\El{p}}$.
\end{lem}
\begin{proof}
    We first treat the case when $\alpha\in(0,\frac{n}{p})$. Let $g\in\El{p}_{+}(\Rn)$ and $\mu\in\mathcal{M}^{s}(\Rn)$. We shall prove the stronger estimate $\norm{I_{\alpha}\ast g}_{\El{1}(\mu)}\lesssim (c_{s}(\mu)^{1/p}\vee 1)\norm{g}_{\El{p}}$. Note that Fubini's theorem followed by Hölder's inequality show that $\norm{I_{\alpha}\ast g}_{\El{1}(\mu)}\leq \norm{I_{\alpha}\ast \mu}_{\El{p'}(\Rn)}\norm{g}_{\El{p}(\Rn)}$, hence it suffices to prove that $\norm{I_{\alpha}\ast \mu}_{\El{p'}(\Rn)}\lesssim (c_{s}(\mu)^{1/p}\vee 1)$. Using Fubini's theorem we can write for all $x\in\Rn$ that
    \begin{align}
    \begin{split}\label{identity riesz potential convolution measure using fubini}
        (I_{\alpha}\ast \mu)(x)&=\gamma_{\alpha,n}\int_{\Rn} \abs{x-y}^{-(n-\alpha)}\dd \mu (y)\eqsim \int_{\Rn}\left(\int_{\abs{x-y}}^{\infty} r^{-(n-\alpha)}\frac{\dd r}{r}\right)\dd \mu (y)\\
        &=\int_{\Rn}\left(\int_{0}^{\infty} r^{-(n-\alpha)} \ind{\Delta(x,r)}(y)\frac{\dd r}{r}\right)\dd \mu (y)\\
        &=\int_{0}^{\infty}\mu(\Delta(x,r))r^{-(n-\alpha)}\frac{\dd r}{r}.
    \end{split}
    \end{align}
    For $r>0$, we consider the function $F_{r}:\Rn\to [0,1]$ defined by $F_{r}(x)= \mu(\Delta(x,r))$ for all $x\in \Rn$. Since $\mu$ is a probability measure, we obtain that 
    \begin{align*}
        \norm{F_{r}}_{\Ell{p'}}^{p'} &= \int_{\Rn} \mu(\Delta(x,r))^{p'}\dd x \\
        &\leq \int_{\Rn}\mu(\Delta(x,r))\dd x
        =\int_{\Rn}\left(\int_{\Rn} \ind{\Delta(x,r)}(y)\dd \mu(y)\right)\dd x\\
        &=\int_{\Rn}\meas{\Delta(y,r)}\dd \mu(y)\lesssim r^{n},
    \end{align*}
    where Fubini's theorem was used to obtain the last equality. In addition, since $\mu$ is $s$-dimensional, we obtain that
    \begin{align*}
        \norm{F_{r}}_{\Ell{p'}}^{p'} &= \int_{\Rn} \mu(\Delta(x,r))^{p'}\dd x = \int_{\Rn}\mu(\Delta(x,r))^{p'-1}\mu(\Delta(x,r))\dd  x\\
        &\leq (c_{s}(\mu)r^{s})^{p'-1}\int_{\Rn}\mu(\Delta(x,r))\dd  x\lesssim (c_{s}(\mu)r^{s})^{p'-1}r^{n}.
    \end{align*}
    Let us fix $R>0$ to be determined later. Minkowski's inequality and \eqref{identity riesz potential convolution measure using fubini} imply that
    \begin{align*}
        \norm{I_{\alpha}\ast \mu}_{\Ell{p'}}&\lesssim \int_{0}^{\infty}\norm{F_{r}}_{\Ell{p'}}r^{-(n-\alpha)}\frac{\dd r}{r}\\
        &\lesssim c_{s}(\mu)^{1/p}\int_{0}^{R}r^{\frac{s}{p}+\frac{n}{p'}}r^{-(n-\alpha)}\frac{\dd r}{r} + \int_{R}^{\infty}r^{\frac{n}{p'}}r^{-(n-\alpha)}\frac{\dd r}{r}\\
        &\eqsim c_{s}(\mu)^{1/p}R^{\alpha-\frac{n}{p}+\frac{s}{p}} + R^{\alpha-\frac{n}{p}},
    \end{align*}
    where we have used that $s>n-\alpha p$ and that $\alpha<\frac{n}{p}$ to show that both integrals are finite.  Since $\frac{n-\alpha p}{s}\in(0,1)$, choosing $R=c_{s}(\mu)^{-1/s}$ yields
    \begin{equation*}
    \norm{I_{\alpha}\ast \mu}_{\Ell{p'}}\lesssim c_{s}(\mu)^{\frac{1}{p}\left(\frac{n-\alpha p}{s}\right)}\leq c_{s}(\mu)^{\frac{1}{p}}\vee 1.
    \end{equation*}

    Finally, if $\alpha=\frac{n}{p}$, then by assumption $s>0$ and we may choose $\gamma\in(0,\frac{n}{p})$ and $\eps>0$ such that $\alpha=\gamma +\eps$ and $s>n-\gamma p$. Using the estimate proved in the first part of the proof we obtain that 
    \begin{align*}
        \norm{G_{\alpha}\ast g}_{\El{1}(\mu)}&=\norm{G_{\gamma}\ast G_{\eps}\ast g}_{\El{1}(\mu)}\leq \norm{I_{\gamma}\ast G_{\eps}\ast g}_{\El{1}(\mu)}\\
        &\lesssim (c_{s}(\mu)^{\frac{1}{p}}\vee 1)\norm{G_{\eps}\ast g}_{\Ell{p}}\\
        &\leq (c_{s}(\mu)^{\frac{1}{p}}\vee 1)\norm{g}_{\Ell{p}}.\qedhere
    \end{align*}
\end{proof}

For any Borel subset $E\subseteq\Rn$, we let $\mathcal{H}^{s}(E)\in[0,\infty]$ denote the $s$-dimensional Hausdorff measure of $E$. The Hausdorff dimension of $E$ is denoted by $\dim_{H}(E)$. The relevant definitions can be found in \cite[Chapter 4]{Mattila_1995}. To obtain the estimates on the Hausdorff dimension of the divergence set in Theorem~\ref{main theorem upper half space}, we will rely on both Lemma \ref{lemma dimensional measures and Bessel potential functions} and the following result, a proof of which can be found in \cite[Theorem~8.8]{Mattila_1995}.
\begin{lem}[Frostman's lemma]\label{frostman's lemma}
    Let $E$ be a Borel subset of $\Rn$, and let $s>0$. The following are equivalent:
    \begin{enumerate}[label=\emph{(\roman*)}]
        \item $\mathcal{H}^{s}(E)>0$;
        \item There exists a finite $s$-dimensional Radon measure $\mu$ on $\Rn$ such that $\supp{\mu}\subseteq E$ and $\mu(E)>0$.
    \end{enumerate}
\end{lem}

In order to briefly describe the general principle that we shall use to obtain the Hausdorff dimension estimates in Theorem \ref{main theorem upper half space}, let us see how Lemma \ref{lemma dimensional measures and Bessel potential functions} and Frostman's lemma can be combined to study the Bessel potential spaces $\mathscr{L}^{p}_{\alpha}(\Rn)$. If $\alpha>\frac{n}{p}$, then it is well-known that all elements of $\mathscr{L}^{p}_{\alpha}(\Rn)$ can be represented by (Hölder) continuous functions vanishing at infinity (see, e.g., \cite[Theorem 1.2.4]{Adams_Hedberg_Potential_Theory}). 
If $\alpha\in (0,\frac{n}{p}]$, $g\in\El{p}_{+}(\Rn)$ and $f=G_{\alpha}\ast g$, then this convolution is defined everywhere on $\Rn$ (though it may take the value $\infty$ on a set of Lebesgue measure zero). For any $s>n-\alpha p$ and any measure $\mu\in\mathcal{M}^{s}(\Rn)$, it follows from Lemma \ref{lemma dimensional measures and Bessel potential functions} that 
\begin{equation*}
    \mu\left(\set{x\in\Rn : f(x)=\infty}\right)=0.
\end{equation*}
It follows from Frostman's lemma that $\mathcal{H}^{s}\left(\set{x\in\Rn : f(x)=\infty}\right)=0$. Since $s>n-\alpha p$ is arbitrary, this implies that 
\begin{equation}\label{hausdorff dimension singularity set of non negative bessel potential function}
\dim_{H}(\set{x\in\bb{R}^{n} : f(x)=\infty})\leq n-\alpha p. 
\end{equation}
This is optimal, as \v{Z}ubrini\'{c} \cite[Theorem~2]{Zubrinic_2002} has shown that for any $0<\lambda<n-\alpha p$, there exists a compact set $A\subset \Rn$ and a function $g\in\El{p}_{+}(\Rn)$ such that $\dim_{H}(A)=\lambda$ and $A=\set{x\in\Rn : (G_{\alpha}\ast g)(x)=\infty}$. 
    
\section{The Dirichlet problem in \texorpdfstring{$\Hn$}{Hn} and the proof of Theorem \ref{main theorem upper half space}}\label{subsection proof of main thm upper half space} 
The main purpose of this section is to prove Theorem ~\ref{main theorem upper half space}, the first of our two main theorems. We begin
Section~\ref{section on nontangential approach dirichlet upper half space + estimate div set} with some additional details on the elliptic operators and the corresponding Dirichlet problems described in the introduction. We then show how Frostman's lemma (Lemma~\ref{frostman's lemma}) can be used to estimate the divergence set corresponding to the classical nontangential approach (see Proposition~\ref{proposition Hausdorff dimension div set for cones and Bessel}). Additional properties of the Bessel potentials required for the proof of Theorem~\ref{main theorem upper half space} are then presented in Section~\ref{section further properties bessel}. Some relevant auxiliary maximal functions are introduced at the start of Section~\ref{Sect:tang}, before we turn to the proof of Theorem~\ref{main theorem upper half space}.  
Finally, Section~\ref{section proof of main theorem for Wsp in upper half space} is dedicated to the proof of an extension of Theorem~\ref{main theorem upper half space} to boundary data $f\in\mathrm{C}^{p}_{\alpha}(\Rn)$ (see Theorem~\ref{main theorem upper half space for Cpalpha}).
Some estimates obtained in the course of the proof of Theorem~\ref{main theorem upper half space for Cpalpha} (specifically those highlighted in Remark~\ref{remark main implication of main thm within proof}) will be used in the proof of Theorem~\ref{second main theorem bounded Lipschitz domain} in Section~\ref{subsection proof second main theorem}.
\subsection{Nontangential approach in the \texorpdfstring{$\El{p}$}{Lp}-Dirichlet problem in \texorpdfstring{$\Hn$}{Hn}}\label{section on nontangential approach dirichlet upper half space + estimate div set}
We let $A:\Hn\to \R^{(1+n)\times (1+n)}$ be a matrix of measurable, real-valued functions on $\Hn$ satisfying the boundedness and ellipticity assumptions \eqref{boundedness aanndd ellipticity conditions}. We consider the corresponding elliptic operator $L=-\dvg{(A\nabla)}$ defined in \eqref{formal elliptic equation in divergence form}. 
As described in the introduction, there exists a family of Borel probability measures $\set{\omega^{(t,x)}_{L}}_{(t,x)\in\Hn}$ on $\Rn$, known as the $L$-harmonic measures, such that for any compactly supported continuous function $f\in C_{c}(\Rn)$, the function 
\begin{equation}\label{representation of classical solutions by harmonic measure}
u_{f}(t,x)=\int_{\Rn}f(y)\dd \omega^{(t,x)}_{L}(y), \quad (t,x)\in\Hn,
\end{equation}
is a solution (the unique bounded solution) to the classical Dirichlet problem with boundary datum $f$, in the sense that it is a weak solution to $L u_{f}=-\dvg(A\nabla u_{f})=0$ in $\Hn$ and that $u_{f}\in C(\clos{\Hn})$ with $u_{f}|_{\partial \Hn}= f$ (see, e.g., \cite[Section 5B]{HLM_Degenerate_Dirichlet_2019}).

As stated in the introduction, for $p\in(1,\infty)$, we say that $(\mathcal{D})^{L}_{p}$ holds if for all $f\in C_{c}(\Rn)$, the classical solution $u_f$ in \eqref{representation of classical solutions by harmonic measure} satisfies 
    \begin{equation*}
        \norm{N_{*}u_{f}}_{\Ell{p}}\lesssim \norm{f}_{\Ell{p}}.
    \end{equation*}

The $L$-harmonic measures $\set{\omega^{X}_{L}}_{X\in\Hn}$ are all mutually absolutely continuous by the Harnack inequality (see, e.g., \cite[Lemma 5.6]{HLM_Degenerate_Dirichlet_2019}). If $\omega_{L}^{X}$ is absolutely continuous with respect to the Lebesgue measure on $\Rn$ (i.e. the surface measure on $\partial\Hn\equiv \Rn)$, then we let ${k(X,\cdot )=\frac{\dd \omega^{X}_{L}}{\dd x}:\Rn\to [0,\infty]}$ denote the corresponding Radon--Nikodym derivative. This function is referred to as the Poisson kernel for $L$ in $\Hn$.

The following well-known result describes how the well-posedness of $\left(\mathcal{D}\right)^{L}_{p}$ relates to the quantitative absolute continuity of the $L$-harmonic measure with respect to the Lebesgue measure on $\Rn$. See, for example, \cite[Theorem 3.1]{Hoffman_Martell_Ainfty_estimates_extrapolation_2012}, and note that $p$ should be replaced by $q$ in items (b) and (c) of this reference. Recall that $X_{Q}$ is the corkscrew point corresponding to the cube $Q\subset \Rn$, see Section~\ref{section on notations}.
\begin{thm}\label{theorem equivalent properties to well posedness of Lp Dirichlet}
The following properties are equivalent:
    \begin{enumerate}[label=\emph{(\roman*)}]
    \item $\omega_{L}^{X_{Q_0}}\in A_{\infty}(Q_0)$, uniformly for all cubes $Q_0\subset\Rn$. In other words, for every cube $Q_0\subset\Rn$ and every $\eps>0$, there exists $\delta>0$ depending only on $n$, $\lambda$, $\Lambda$ and $\eps$, such that if $Q\subset Q_0$ is a cube and $E\subset Q$ is such that $\omega_{L}^{X_{Q_0}}(E)\leq \delta \omega_{L}^{X_{Q_0}}(Q)$, then $\meas{E}\leq \eps \meas{Q}$.
    \item There exists $p\in(1,\infty)$ such that $(\mathcal{D})_{p}^{L}$ holds.
    \item The measures $\set{\omega^{X}_{L}}_{X\in\Hn}$ are all absolutely continuous with respect to the Lebesgue measure, and there exists $q\in(1,\infty)$ and $C_0>0$ such that for all cubes $Q_{0}\subset\Rn$, the Poisson kernel $k(X_{Q_0},\cdot )$ satisfies 
\begin{equation}\label{reverse Holder property for rn derivative in local cube}
    \left(\dashint_{Q}k(X_{Q_0},y)^{q}\dd y\right)^{1/q}\leq C_{0} \dashint_{Q}k(X_{Q_0},y)\dd y\leq C_{0}\meas{Q}^{-1}
\end{equation}
for all cubes $Q\subseteq Q_{0}$.
    \end{enumerate}
Moreover, in the equivalence between \emph{(ii)} and \emph{(iii)}, it holds that $\frac{1}{p}+\frac{1}{q}=1$.
\end{thm}
As stated in the introduction, the well-posedness of $(\mathcal{D})^{L}_{p}$ is guaranteed by the condition that $(\mathcal{D})^{L}_{p}$ holds. The result is classical for bounded domains (see \cite[Theorem 1.7.7]{Kenig_cbms_1994}) but a different argument is needed to treat uniqueness in unbounded domains. A simple proof is given in \cite[Section 6]{HLMP_Dirichlet_BMO_antisym_2022}. The following lemma and its proof are stated for completeness.  
\begin{lem}\label{lemma Dirichlet holds implies well posedness}
    If any of the equivalent properties of Theorem~\ref{theorem equivalent properties to well posedness of Lp Dirichlet} hold, then there is $p\in(1,\infty)$ (that given by item \emph{(ii)} of Theorem~\ref{theorem equivalent properties to well posedness of Lp Dirichlet}) such that $(\mathcal{D})_{p}^{L}$ is well-posed, and for any $f\in\Ell{p}$ the unique solution $u_{f}$ with boundary data $f$ is given by 
    \begin{equation}\label{definition of u_f for f in Lp}
u_{f}(t,x)=\int_{\Rn}f(y)\dd \omega^{(t,x)}_{L}(y)=\int_{\Rn} k((t,x),  y)f(y)dy
\end{equation}
for all $(t,x)\in \Hn$.
\end{lem}
\begin{proof}
We note that condition (ii) in Theorem \ref{theorem equivalent properties to well posedness of Lp Dirichlet} implies that the Poisson kernel $k((t,x),\cdot )$ belongs to $\Ell{q}$ for all $(t,x)\in\Hn$, with the estimate 
\begin{equation}\label{L^q estimate on Radon Nikodym derivative k}
    \int_{\Rn}k((t,x),y)^{q}\dd y \lesssim \meas{\Delta(x,t)}^{1-q},
\end{equation}
where the implicit constant only depends on $n,\lambda$ and $\Lambda$ (see, e.g.,  \cite[Lemma 5.32]{HLM_Degenerate_Dirichlet_2019} and Remark~\ref{remark about proof of lem 5.32 in HLM} below). This implies that for $f\in\Ell{p}$, the integral in \eqref{definition of u_f for f in Lp} converges by Hölder's inequality. It is standard that condition (ii) of Theorem \ref{theorem equivalent properties to well posedness of Lp Dirichlet} and \eqref{L^q estimate on Radon Nikodym derivative k} imply that $u_{f}$ is a weak solution satisfying the nontangential maximal function estimate $\norm{N_{*}u_{f}}_{\Ell{p}}\lesssim \norm{f}_{\Ell{p}}$ and the almost everywhere nontangential convergence to $f$ at the boundary (see, e.g., \cite[Theorem 5.34]{HLM_Degenerate_Dirichlet_2019}). The fact that $u_f$ is the unique solution to $\left(\mathcal{D}\right)^{L}_{p}$ (under the assumption that (ii) of Theorem \ref{theorem equivalent properties to well posedness of Lp Dirichlet} holds) is proved in \cite[Section~6]{HLMP_Dirichlet_BMO_antisym_2022}. 
\end{proof}

\begin{rem}\label{remark about proof of lem 5.32 in HLM}
    We take this opportunity to rectify a slight imprecision that appears in the proof of the nontangential maximal function estimate in the proof of \cite[Lemma 5.32]{HLM_Degenerate_Dirichlet_2019} (which concerns weights $\mu$ that here we only consider when $\mu\equiv 1$). This is pertinent to the current paper as the estimate obtained there is essential to the proof of Theorem~\ref{main theorem upper half space}. The reader may wish to return to this as the need arises later in that proof.
    Given a non-negative $f\in\El{p}(\Rn)$, one should not only fix $x_0\in\Rn$ but also some $t_0>0$ before decomposing $f$ into
    \begin{equation}\label{decomposition of f into dyadic annuli in the proof of lem 5.32 HLM}
        f=f\ind{\Delta(x_0,2t_0)} + \sum_{j =1}^{\infty}f\ind{\Delta(x_0, 2^{j+1}t_0)\setminus \Delta(x_0,2^{j}t_0)}=:\sum_{j=0}^{\infty}f_{j},
    \end{equation}
    and decomposing $u(X)=\int_{\Rn}f\dd \omega^{X}$ correspondingly into 
    \begin{equation*}
        u(X)=\sum_{j=0}^{\infty}u_{j}(X),
    \end{equation*}
    with $u_{j}(X)=\int_{\Rn}f_{j}\dd\omega^{X}$ for all $X=(t,x)\in\Hn$. This ensures that the functions $u_j$ are non-negative solutions to the equation $-\dvg(A\nabla u_j)=0$ (which vanish on various parts of the boundary $\partial\Hn\equiv\Rn$) and we are therefore justified in applying the Harnack inequality \cite[equation (2.18)]{HLM_Degenerate_Dirichlet_2019} and the boundary Hölder continuity estimate \cite[equations (5.1) and (5.2)]{HLM_Degenerate_Dirichlet_2019} (in place of \cite[equation (5.7)]{HLM_Degenerate_Dirichlet_2019}) to obtain
    \[
    u_0(t_0,x_0)\lesssim u_0(4t_0,x_0) \quad \text{and} \quad u_j(t,x)\lesssim \left(\frac{t}{2^jt_0}\right)^{\alpha_{L}}u_j(2^{j+2}t_0,x_0)\]
    when $t\in(0,2t_0)$, $x\in\Delta(x_0,t_0)$ and $j\geq 1$, where $\alpha_{L}>0$ is from \eqref{holder continuity solutions DGNM}, so in particular
    \begin{equation}\label{eq:u_jdecompHLM}
    u(t_0,x_0)\lesssim \sum_{j=0}^\infty 2^{-\alpha_{L}j} u_{j}(2^{j+2}t_0,x_0).
   \end{equation}
   In the proof of \cite[Lemma 5.32]{HLM_Degenerate_Dirichlet_2019}, the function $f$ is decomposed as in \eqref{decomposition of f into dyadic annuli in the proof of lem 5.32 HLM} but with $t_0=t$, where $X=(t,x)\in\Hn$. Consequently, the functions $f_j$ also depend on $X$, which does not guarantee that the functions $u_j$ are solutions. However, starting from \eqref{decomposition of f into dyadic annuli in the proof of lem 5.32 HLM} and using \eqref{eq:u_jdecompHLM} for arbitrary $(t_0,x_0)\in\Hn$, the desired estimates are  obtained from the proof of \cite[Lemma 5.32]{HLM_Degenerate_Dirichlet_2019}.
\end{rem}

For $s\geq 1$, we now introduce the operator $M_{s}$, which is a variant of the standard Hardy--Littlewood maximal operator $M$, defined for all $x\in\Rn$ and $f\in\El{s}_{\loc}(\Rn)$ by 
\begin{equation*}
    M_{s}(f)(x)=M({\abs{f}^{s}})(x)^{1/s}=\sup_{t>0}\left(\dashint_{\Delta(x,t)}\abs{f}^{s}\right)^{1/s}.
\end{equation*}
In order to illustrate how the estimate \eqref{upper bound hausdorff dimension of divergence set} on the Hausdorff dimension of the divergence set in Theorem~\ref{main theorem upper half space} can be obtained using Frostman's lemma (Lemma~\ref{frostman's lemma}), we shall first prove (in Proposition~\ref{proposition Hausdorff dimension div set for cones and Bessel} below) the simpler (nontangential) case corresponding to $\beta'=1$ in Theorem~\ref{main theorem upper half space}. It relies on the following elementary observation about the maximal operators $M_{s}$ and convolutions.
\begin{lem}\label{lemma maximal function commutes with convolution}
    Let $q\in[1,\infty)$. For all non-negative $G\in\El{1}_{\loc}(\Rn)$ and non-negative $g\in\El{q}_{\loc}(\Rn)$ it holds that 
    \begin{equation*}
        M_{q}(G\ast g)(x)\leq (G\ast M_{q}(g))(x)
    \end{equation*}
    for all $x\in\Rn$.
\end{lem}
\begin{proof}
    Let $x\in\Rn$ and $t>0$ be fixed. By Minkowski's inequality and a change of variables we obtain that 
    \begin{align*}
        \left(\dashint_{\Delta(x,t)} (G\ast g)^{q} \right)^{1/q} &= \meas{\Delta(x,t)}^{-1/q}\norm{\int_{\Rn}g(\cdot - z)G(z)\dd z}_{\El{q}(\Delta(x,t))}\\
        &\leq \meas{\Delta(x,t)}^{-1/q}\int_{\Rn}\norm{g(\cdot - z)}_{\El{q}(\Delta(x,t))}G(z)\dd z\\
        &=\meas{\Delta(x-z,t)}^{-1/q}\int_{\Rn}\norm{g}_{\El{q}(\Delta(x-z,t))}G(z)\dd z\\
        &\leq \int_{\Rn}(M_{q}(g))(x-z)G(z)\dd z = (G\ast M_{q}(g))(x).\qedhere
    \end{align*}
\end{proof}

We can now use this lemma to estimate the Hausdorff dimension of the divergence set below. Again, it is implicit in Proposition~\ref{proposition Hausdorff dimension div set for cones and Bessel} below that the function $f$ in \eqref{estimate on hausdorff dimension of divergence set for conical regions and bessel potential functions} is the preferred representative $\tilde{f}:\Rn\to (-\infty,\infty]$ from~\eqref{preferred representative for bessel potential space function}. 
\begin{prop}\label{proposition Hausdorff dimension div set for cones and Bessel}
    Let $p\in(1,\infty)$ be such that $\left(\mathcal{D}\right)^{L}_{p}$ holds. If $\alpha\in (0,\frac{n}{p}]$, $f\in\mathscr{L}^{p}_{\alpha}(\Rn)$ and $u_{f}$ is the unique solution to $\left(\mathcal{D}\right)^{L}_{p}$ with boundary datum $f$, then
\begin{equation}\label{estimate on hausdorff dimension of divergence set for conical regions and bessel potential functions}
    \dim_{H}\left(\set{x_{0}\in\Rn : \lim_{\Gamma(x_0)\ni(t,x)\to (0,x_{0})} u_{f}(t,x)\neq f(x_0)}\right)\leq n-\alpha p.
\end{equation}
\end{prop}
\begin{proof}
Let $g\in\El{p}(\Rn)$ such that $f=G_{\alpha}\ast g\in\mathscr{L}^{p}_{\alpha}(\Rn)$. The unique solution to $\left(\mathcal{D}\right)^{L}_{p}$ with boundary datum $f$ is given by \eqref{definition of u_f for f in Lp}.

By relying on the self-improvement property of the reverse Hölder estimate in item (ii) of Theorem~\ref{theorem equivalent properties to well posedness of Lp Dirichlet}, it is shown in the proof of \cite[Lemma 5.32]{HLM_Degenerate_Dirichlet_2019} (see also Remark~\ref{remark about proof of lem 5.32 in HLM} above) that there exists $1<r'<p$ such that 
\begin{equation}\label{pointwise domination of nontangential max by power of HL max}
    (N_{*}u_{f})(x_0)\lesssim (M_{r'}f)(x_0)
\end{equation}
for all $x_0\in\Rn$. Note that the maximal function $M_{r'}$ is bounded on $\El{p}$ since $p>r'$, and let us remark that this recovers the usual $\El{p}$ nontangential maximal function bounds.
It follows from \eqref{pointwise domination of nontangential max by power of HL max} and Lemma~\ref{lemma maximal function commutes with convolution} that
    \begin{align*}
        (N_{*}u_{f})(x_0)\lesssim (M_{r'}(G_{\alpha}\ast \abs{g}))(x_0)\leq (G_{\alpha}\ast (M_{r'}g))(x_0)
    \end{align*}
    for all $x_0\in\Rn$. Consequently, if $s>n-\alpha p$ and $\mu\in\mathcal{M}^{s}(\Rn)$, then it follows from Lemma~\ref{lemma dimensional measures and Bessel potential functions} that
    \begin{align}\label{maximal function estimate wrt dimensional measure to Bessel}
        \norm{N_{*}u_{f}}_{\El{1}(\mu)}\lesssim \norm{G_{\alpha}\ast (M_{r'}g)}_{\El{1}(\mu)}\lesssim \norm{M_{r'}g}_{\Ell{p}}\lesssim \norm{g}_{\Ell{p}}=\norm{f}_{\mathscr{L}^{p}_{\alpha}(\Rn)},
    \end{align}
    where the implicit constants also depend on $c_{s}(\mu)$, as described by Lemma~\ref{lemma dimensional measures and Bessel potential functions}, but this is unimportant here.
    The bound~\eqref{estimate on hausdorff dimension of divergence set for conical regions and bessel potential functions} now follows by a standard approximation argument and application of Frostman's lemma, which we detail here for completeness.
    
    Let us fix $\eps>0$, and let $\tilde{f}$ be the representative of $f$ given by \eqref{preferred representative for bessel potential space function}. As explained in Section~\ref{subsection with definitions of spaces}, it holds that $\smallabs{\tilde{f}(x)}\lesssim (G_{\alpha}\ast \abs{g})(x)$ for \emph{all} $x\in\Rn$. Consequently, Lemma~\ref{lemma dimensional measures and Bessel potential functions} shows that $\smallnorm{\tilde{f}}_{\El{1}(\mu)}\lesssim \norm{g}_{\Ell{p}}=\norm{f}_{\mathscr{L}^{p}_{\alpha}(\Rn)}$. Since there exists a Schwartz function $h\in \mathcal{S}(\Rn)$ such that $\norm{f-h}_{\mathscr{L}^{p}_{\alpha}(\Rn)}<\eps$, it follows that
    \begin{equation*}
    \smallnorm{\tilde{f}-h}_{\El{1}(\mu)}\lesssim \norm{f-h}_{\mathscr{L}^{p}_{\alpha}(\Rn)}< \eps.
    \end{equation*}
Even if $h\not\in C_{c}(\Rn)$ by lack of a compact support, a simple density argument shows that $u_h\in C^{0}(\clos{\Hn})$ with $u_{h}|_{\partial\Hn}=h$. Observe that for any $x_0\in\Rn$ and $(t,x)\in\Hn$ it holds that
    \begin{equation*}
        \abs{u_{f}(t,x)-\tilde{f}(x_0)}\leq \abs{u_{f-h}(t,x)} + \abs{u_{h}(t,x)-h(x_0)} + \abs{h(x_0)-\tilde{f}(x_0)},
    \end{equation*}
    and since $u_{h}\in C^{0}(\clos{\Hn})$ it follows that
    \begin{equation*}
        \limsup_{\Gamma(x_0)\ni(t,x)\to (0,x_{0})}\abs{u_{f}(t,x)-\tilde{f}(x_0)} \leq (N_{*}u_{f-h})(x_0) + \abs{\tilde{f}(x_0)-h(x_0)}.
    \end{equation*}
    As a consequence, we can use Markov's inequality, \eqref{maximal function estimate wrt dimensional measure to Bessel} and Lemma~\ref{lemma dimensional measures and Bessel potential functions} to obtain for any $\lambda>0$ that
    \begin{align*}
        \mu\left(\!\set{x_{0}\in\Rn : \limsup_{\Gamma(x_0)\ni(t,x)\to (0,x_{0})}\abs{u_{f}(t,x)-\tilde{f}(x_0)} >\lambda }\!\right) &\lesssim \lambda^{-1}\norm{N_{*}u_{f-h}}_{\El{1}(\mu)} \! + \! \lambda^{-1}\smallnorm{\tilde{f}-h}_{\El{1}(\mu)}\\
        &\lesssim \lambda^{-1}\norm{f-h}_{\mathscr{L}^{p}_{\alpha}(\Rn)}<\eps\lambda^{-1}.
    \end{align*}
    Since $\eps$ and $\lambda$ are arbitrary, this implies that 
    \begin{equation*}
        \mu\left(\set{x_{0}\in\Rn : \limsup_{\Gamma(x_0)\ni(t,x)\to (0,x_{0})}\abs{u_{f}(t,x)-\tilde{f}(x_0)} \neq 0}\right)=0. 
    \end{equation*} 
    Since $\mu\in\mathcal{M}^{s}(\Rn)$ was arbitrary, Frostman's lemma (Lemma~\ref{frostman's lemma}) implies that 
    \begin{equation}\label{beta dimensional measure of divergence set on the unit ball is zero}
        \mathcal{H}^{s}\left(\set{x_{0}\in\Rn : \limsup_{\Gamma(x_0)\ni(t,x)\to (0,x_{0})}\abs{u_{f}(t,x)-\tilde{f}(x_0)} \neq 0}\right)=0.
    \end{equation}
    Since $s>n-\alpha p$ was arbitrary, this concludes the proof.
\end{proof}
\subsection{Further properties of Bessel potentials}\label{section further properties bessel}
For $t>0$ and $x\in\Rn$, the classical Poisson kernel $P_{t}(x)$ is defined by
\begin{equation*}
    P_{t}(x)=t^{-n}P_{1}(x/t) = \frac{c_n t}{(t^{2}+\abs{x}^{2})^{\frac{n+1}{2}}}
\end{equation*}
where $c_n>0$ guarantees that $\int_{\Rn}P_{t}(x)\dd x=1$ for all $t>0$.
\begin{lem}\label{lemma domination of multiples of Poisson kernel by convolution with bessel}
    If $\alpha\in[0,n)$ and $T>0$, then $t^{\alpha}P_{t}(x)\lesssim \mathscr{J}_{\alpha}(P_{t})(x)$ for all $x\in\Rn$ and $0<t\leq T$. In particular, for all $p\in[1,\infty]$ and $g\in\El{p}_{+}(\Rn)$ it holds that $t^{\alpha}(P_{t}\ast g)(x)\lesssim (P_{t}\ast \mathscr{J}_{\alpha}(g))(x)$ for all $x\in\Rn$ and $0<t\leq T$.
\end{lem}
\begin{proof}
    This is trivial when $\alpha=0$. Let us assume that $\alpha\in(0,n)$. There is some $\delta>0$ such that $G_{\alpha}(y)\gtrsim \abs{y}^{-(n-\alpha)}$ for all $\abs{y}<\delta$ (see, e.g., (29) in \cite[Chapter \RNum{5}, Section 3]{Stein_Singular_Integrals}). We may assume that $\delta\leq T$. For all $x\in\Rn$ and $0<t\leq T$ it holds that 
    \begin{align*}
        \mathscr{J}_{\alpha}(P_{t})(x)&= \int_{\Rn}P_{t}(x-y)G_{\alpha}(y)\dd y\\
        &\gtrsim t^{-n}\int_{\abs{y}<\delta} P_{1}\left(\frac{x-y}{t}\right)\abs{y}^{-(n-\alpha)}\dd y\\
        &=t^{\alpha-n}\int_{\abs{z}<\delta/t}P_{1}\left(\frac{x}{t}-z\right)\abs{z}^{-(n-\alpha)}\dd z\\
        &\geq t^{\alpha-n}\int_{\abs{z}<\delta T^{-1}}P_{1}\left(\frac{x}{t}-z\right)\dd z,
    \end{align*}
    where we have made the change of variables $z=y/t$ and used that $\abs{z}^{-(n-\alpha)}\geq 1$ for all $\abs{z}<\delta T^{-1}\leq 1$. We conclude by observing that $P_{1}(y-z)\gtrsim P_{1}\left(y\right)$ for all $y\in\Rn$ and $\abs{z}<1$. 
\end{proof}
The following elementary lemma will serve as a substitute for the classical Poincaré inequality in the context of Bessel potential spaces of fractional order.
\begin{lem}[$\mathscr{L}^{p}_{\alpha}$-Poincaré]\label{lemma Poincaré type inequality for fractional order Bessel spaces}
    Let $\alpha\in(0,1)$, $p\in[1,\infty]$ and $g\in\El{p}(\Rn)$. If $f=G_{\alpha}\ast g$, then for almost all $y,z\in\Rn$ it holds that 
    \begin{equation}\label{pointwise estimate on oscillation of bessel potential functions}
        \abs{f(y)-f(z)}\lesssim \abs{y-z}^{\alpha}\left(Mg(y) + Mg(z)\right).
    \end{equation}
    Consequently, for all $q\in[1,\infty)$ and all balls $\Delta\subset\Rn$ of radius $r>0$ it holds that 
    \begin{equation*}
        \left(\dashint_{\Delta}\abs{f-(f)_{\Delta}}^{q}\right)^{1/q}\lesssim r^{\alpha}\left(\dashint_{\Delta}\left(Mg\right)^{q}\right)^{1/q}.
    \end{equation*}
\end{lem}
\begin{proof}
    The proof is based on the fact that
    \begin{equation}\label{main continuity estimate for Bessel potentials}
        \abs{G_{\alpha}(x-y)-G_{\alpha}(x)}\lesssim\abs{y}\abs{x}^{-n+\alpha-1}
    \end{equation}
    for all $x,y\in\Rn$ such that $\abs{x}\geq 2\abs{y}$, which follows from \eqref{estimate on (gradient) of bessel potential} and the mean value theorem. Since $f\in\El{p}(\Rn)$, there is a null set $E\subset\Rn$ such that $\abs{f(x)}<\infty$ for all $x\in\Rn\setminus E$. If $y,z\in\Rn\setminus E$ and $y\neq z$, then 
    \begin{align}
    \begin{split}\label{main decomposition estimate for difference of Bessel potentials}
        \abs{f(y)-f(z)}&\lesssim \int_{\Rn}\abs{G_{\alpha}(y-x)-G_{\alpha}(z-x)}\abs{g(x)}\dd x\\
        &= \int_{\abs{z-x}\geq 2\abs{z-y}}\abs{G_{\alpha}((z-x) -(z-y))-G_{\alpha}(z-x)}\abs{g(x)}\dd x \\
        &+ \int_{\abs{z-x}<2\abs{z-y}}\abs{G_{\alpha}(y-x)-G_{\alpha}(z-x)}\abs{g(x)}\dd x\\
        &\lesssim  \abs{z-y}\int_{\abs{z-x}\geq 2\abs{z-y}} \abs{z-x}^{-n+\alpha -1}\abs{g(x)}\dd x\\
        &+ \int_{\abs{y-x}<3\abs{z-y}} \abs{y-x}^{-n+\alpha}\abs{g(x)}\dd x +\int_{\abs{z-x}<2\abs{z-y}} \abs{z-x}^{-n+\alpha}\abs{g(x)}\dd x,
    \end{split}
    \end{align}
    where we have used \eqref{main continuity estimate for Bessel potentials}, the estimate $G_{\alpha}(x)\leq I_{\alpha}(x)$, and the fact that $\abs{z-x}<2\abs{y-z}$ implies $\abs{y-x}<3\abs{z-y}$. For any $r>0$, observe that 
    \begin{align*}
        \int_{\abs{z-x}\geq r}\abs{z-x}^{-n+\alpha -1}\abs{g(x)}\dd x &= \sum_{k=0}^{\infty} \int_{2^{k}r\leq \abs{z-x}<2^{k+1}r} \abs{z-x}^{-n+\alpha -1}\abs{g(x)}\dd x\\
        &\lesssim \sum_{k=0}^{\infty} (2^{k}r)^{\alpha-1}\dashint_{B(z,2^{k+1}r)}\abs{g(x)}\dd x \lesssim r^{\alpha-1}Mg(z),
    \end{align*}
    using that the series converges because $\alpha<1$. Similarly, 
    \begin{align*}
        \int_{\abs{y-x}<r}\abs{y-x}^{-n+\alpha}\abs{g(x)}\dd x &= \sum_{k=0}^{\infty} \int_{2^{-(k+1)}r \leq \abs{y-x}<2^{-k}r}\abs{y-x}^{-n+\alpha}\abs{g(x)}\dd x\\
        &\lesssim \sum_{k=0}^{\infty} (2^{-k}r)^{\alpha}\dashint_{B(y,2^{-k}r)}\abs{g(x)}\dd x \lesssim r^{\alpha}Mg(y),
    \end{align*}
    using that the series converges because $\alpha>0$. Combining these two estimates with \eqref{main decomposition estimate for difference of Bessel potentials} yields \eqref{pointwise estimate on oscillation of bessel potential functions}.
\end{proof}
We remark that the previous lemma is reminiscent of one implication of the well-known Haj\l asz--Sobolev characterisation of Sobolev spaces, which asserts that if $f\in\textup{W}^{1,p}(\Rn)$ and $1<p<\infty$ then 
\begin{equation*}
    \abs{f(y)-f(z)}\lesssim \abs{y-z}\left(M(\nabla f)(y) +M(\nabla f)(z)\right)
\end{equation*}
for almost every $y,z\in\Rn$ (see, e.g., \cite[Theorems 2.1 and 2.2]{Hajlasz_Sobolev_2003}).

We shall use Lemma~\ref{lemma Poincaré type inequality for fractional order Bessel spaces} in conjunction with the following well-known result due to Calder\'{o}n.
\begin{thm}\label{calderon's theorem identification bessel with sobolev}
Let $p\in(1,\infty)$ and $\alpha\geq 1$. Then $f\in\mathscr{L}^{p}_{\alpha}(\Rn)$ if and only if $f\in\mathscr{L}^{p}_{\alpha-1}(\Rn)$ and $\partial_{k}f\in \mathscr{L}^{p}_{\alpha-1}(\Rn)$ for each $k\in\set{1,\ldots, n}$. Moreover, $\norm{f}_{\mathscr{L}^{p}_{\alpha}}\eqsim \norm{f}_{\mathscr{L}^{p}_{\alpha-1}} + \sum_{k=1}^{n}\norm{\partial_{k}f}_{\mathscr{L}^{p}_{\alpha-1}}$.
\end{thm}
Theorem \ref{calderon's theorem identification bessel with sobolev} relies on a special case of the following lemma, whose proof follows from the arguments in the proof of \cite[Chapter \RNum{5}, Lemma 3]{Stein_Singular_Integrals}.
\begin{lem}\label{lemma representation formula derivatives of Bessel potential functions}
    Let $\alpha>0$, $p\in(1,\infty)$, $g\in\El{p}(\Rn)$ and $f= \mathscr{J}_{\alpha}g \in\mathscr{L}^{p}_{\alpha}(\Rn)$. For all multi-indices $\gamma \in\N^{n}$ with $0<\abs{\gamma}=k\leq \alpha$, it holds that $\partial^{\gamma}f\in\mathscr{L}^{p}_{\alpha- k}(\Rn)$ with the representation
    \begin{equation*}
        \partial^{\gamma}f = \mathscr{J}_{\alpha-k}\left(\mathcal{R}^{\gamma} (\mu_{k}\ast g)\right),
    \end{equation*}
    where $\mathcal{R}^{\gamma}=\mathcal{R}^{\gamma_{1}}_{1}\ldots\mathcal{R}_{n}^{\gamma_{n}}\in\mathcal{L}(\Ell{p})$ is a composition of Riesz transforms, and $\mu_{k}$ is a finite (signed) Borel measure on $\Rn$. In particular, for any multi-index $\gamma\in\N^{n}$ with $0\leq \abs{\gamma}=k\leq \alpha$, there exists $g^{(\gamma)}\in\El{p}(\Rn)$ such that $\partial^{\gamma}f=\mathscr{J}_{\alpha-k}(g^{(\gamma)})$ and $\smallnorm{g^{(\gamma)}}_{\El{p}}\lesssim \norm{g}_{\El{p}}$.
\end{lem}
We now introduce some notation. Given a function $f\in\El{1}_{\loc}(\Rn)$ with weak partial derivatives of order $k\geq 1$ in $\El{1}_{\loc}(\Rn)$, we let 
\begin{equation}\label{notation for higher order gradients of integer order k}
\nabla^{k}f=\left(\partial^{\gamma}f\right)_{\gamma \in\N^{n}, \abs{\gamma}=k},
\end{equation}
where the vector is arranged in lexicographic order (for example).
In general, we denote the norm of a vector $x=(x^{\gamma})_{\gamma \in\N^{n}, 
 \abs{\gamma}=k}$, where $x^{\gamma}\in\R$ for all $\gamma\in\N^{n}$ with $\abs{\gamma}=k$, by $\abs{x}=\sum_{\gamma\in\N^{n}, \abs{\gamma}=k}\abs{x^{\gamma}}$. With that in mind, in the context of Lemma \ref{lemma representation formula derivatives of Bessel potential functions} and in the pointwise almost everywhere sense, it holds that
 \begin{align}\label{ponitwise domination of abs kth gradient by Bessel potential of some gk}
     \smallabs{\nabla^{k}f}&=\sum_{\gamma\in\N^{n}, \abs{\gamma}=k}\abs{\partial^{\gamma}f}\leq \sum_{\abs{\gamma}=k}\mathscr{J}_{\alpha-k}(\smallabs{g^{(\gamma)}}) = \mathscr{J}_{\alpha-k}(\smallabs{g^{k}}),
 \end{align}
 where $g^{k}=(g^{(\gamma)})_{\gamma\in\N^{n}, \abs{\gamma}=k}$, and therefore $\smallnorm{\smallabs{g^{k}}}_{\Ell{p}}\lesssim \norm{g}_{\El{p}}$. We extend this notation to the case $k=0$ by setting $\nabla^{0}f=f$ and $g^{0}=g$, so that \eqref{ponitwise domination of abs kth gradient by Bessel potential of some gk} trivially holds in this case.

\subsection{Tangential maximal functions}\label{Sect:tang}
 For any exponent $\beta\in(0,1]$, and $x_{0}\in\Rn$ we recall the tangential approach region
 \begin{equation*}
     \Gamma^{\beta}(x_0)=\set{(t,x)\in\Hn : \abs{x-x_{0}}<t^{\beta} \textup{ if }0<t\leq 1, \textup{ and } \abs{x-x_0}<t \textup{ if } t\geq 1}.
 \end{equation*}
For $\alpha\geq 0$, $\beta\in (0,1]$, $p\in [1,\infty]$ and $g\in\Ell{p}$, consider the maximal function 
 \begin{equation}\label{Nagel Stein maximal function poisson extension of bessel potential bounded on Lp}
     \mathcal{N}_{\alpha , \beta}(g)(x_0)= \sup_{(t,x)\in\Gamma^{\beta}(x_0)} \abs{\left(P_{t}\ast \mathscr{J}_{\alpha}g\right)(x)}, \quad x_0\in\Rn. 
 \end{equation}
 
 Theorem~\ref{theorem Nagel--Stein thm 5} states that $\mathcal{N}_{\alpha,\beta}$ is bounded on $\Ell{p}$ provided that $p\in(1,\infty)$, $\alpha p<n$ and $\beta\geq 1-\frac{\alpha p}{n}$. As mentioned in the introduction, if $\alpha=\frac{n}{p}$ then $\mathcal{N}_{\alpha,\beta}$ is bounded on $\Ell{p}$ for all $\beta>0$.

For any $p\in(0,\infty)$ and $u:\Hn\to\R$, following \cite[Section 4]{Nagel_Stein_1984} we also consider the local maximal function
\begin{equation*}
    \mathcal{M}_{p,\beta}(u)(x_0)=\sup_{\substack{(t,x)\in\Gamma^{\beta}(x_0) \\ 0<t\leq 1}} t^{\frac{n(1-\beta)}{p}}\abs{u(t,x)}, \quad x_0\in\Rn.
\end{equation*}
This corresponds to the maximal function $M_{p}(u)(x_0)$ for the domain $\Omega=\Gamma^{\beta}(0)$ in the notation of \cite[Section 3]{Nagel_Stein_1984}.
The factor $t^{\frac{n(1-\beta)}{p}}$ mitigates the growth of $\abs{u(t,x)}$ as $(t,x)$ approaches $(0,x_0)$ within $\Gamma^{\beta}(x_0)$, allowing the maximal function $\mathcal{M}_{p,\beta}(u)$ to essentially be dominated by the standard nontangential maximal function $N_{*}(u)$. In fact, as a special case of \cite[Theorem~4$'$]{Nagel_Stein_1984} it holds that 
\begin{equation}\label{Lp domination of mitigating maximal function by non tangential max function nagel stein}
    \norm{\mathcal{M}_{p,\beta}(u)}_{\Ell{p}}\lesssim \norm{N_{*}(u)}_{\Ell{p}}.
\end{equation}
Note that the continuity assumption made on $u$ in the aforementioned reference is unnecessary.

The final lemma that we shall need for the proof of Theorem \ref{main theorem upper half space} is the following analogue of \cite[Lemma 11]{Nagel_Stein_1984}. Here we have to consider some $\El{q}$ averages of the function $f$ instead of working directly with the Poisson extension $u(t,x)=(P_{t}\ast f)(x)$ as in the aforementioned reference. 
 \begin{lem}\label{lemma boundedness mitigating maximal function on j dependent intervals}
     Let $1\leq q< p < \infty$ and $f\in\Ell{p}$. Let $v_f:\Hn\to[0,\infty]$ be defined as 
     \begin{equation*}
         v_{f}(t,x)=\left(\dashint_{\Delta(x,2t)}\abs{f}^{q}\right)^{1/q}. 
     \end{equation*}
     For integers $j\geq 0$ and $\beta\in(0,1)$, consider the maximal function 
     \begin{equation}\label{maximal function dilated in the t direction for j integer}
         \mathcal{M}_{p,\beta,j}(v_{f})(x_0)= 2^{\frac{nj}{p}}\sup_{\substack{\abs{x-x_0}< t^{\beta}\\0< t < 2^{-j/(1-\beta)}}} t^{\frac{n(1-\beta)}{p}}\abs{v_{f}(2^{j}t, x)}, \quad x_0\in\Rn.
     \end{equation}
     Then, $\norm{\mathcal{M}_{p,\beta,j}(v_{f})}_{\El{p}}\lesssim \norm{f}_{\El{p}}$, where the implicit constant does not depend on $j$. 
 \end{lem}
 \begin{proof}
     Observe that $N_{*}(v_f)(x_0)\lesssim (M_{q}f)(x_0)$ for all $x_0\in\Rn$ and therefore
     \begin{equation*}
         \norm{N_{*}(v_f)}_{\El{p}}\lesssim \norm{M_{q}f}_{\El{p}}\lesssim \norm{f}_{\El{p}}
     \end{equation*}
     for all $f\in\Ell{p}$, since $p>q$. Moreover, the dilations $f_{\delta}(x)=f(\delta x)$, $x\in\Rn$, have the property that $v_{f_{\delta}}(t,x)=v_{f}(\delta t,\delta x)$ for all $(t,x)\in\Hn$ and $\delta>0$. The result therefore follows from \eqref{Lp domination of mitigating maximal function by non tangential max function nagel stein} by scaling, just as in \cite[Lemma 11]{Nagel_Stein_1984}.
 \end{proof}
Let us remark that $t^{\frac{n(1-\beta)}{p}}=t^{\alpha}$ when $\beta=1-\frac{\alpha p}{n}$ in \eqref{maximal function dilated in the t direction for j integer}. We also make the following remark for later reference.
\begin{rem}\label{remark 2 - convolution commute  with general maximal functions approach}
    Let $\Gamma\subseteq \Hn$ be any subset (typically an approach region with $(0,0)\in\R^{1+n}$ as an accumulation point). For any $x_0\in\Rn$, let us denote 
    \begin{equation*}
    \Gamma +x_0 = \set{(t,x)+(0,x_0) : (t,x)\in\Gamma}\subseteq \Hn.
    \end{equation*}
    For $u:\Hn\to\R$, consider (with a slight abuse of notation) the corresponding maximal function 
    \begin{equation*}
    N_{*,\Gamma}u(x_0)=\sup_{(t,x)\in\Gamma +x_0}\abs{u(t,x)}, \quad x_0\in\Rn.
    \end{equation*}
    If there are $G:\Rn\to [0,\infty]$ and $v:\Hn\to\R$ such that $\abs{u(t,x)}\leq (G\ast \abs{v(t,\cdot)})(x)$ for all $(t,x)\in\Hn$, then 
    \begin{equation*}
        (N_{*,\Gamma}u)(x_0)\leq (G\ast N_{*,\Gamma} v)(x_0) \quad \textup{ for all }x_0\in\Rn.
    \end{equation*}
    Indeed, if $x_0\in\Rn$ and $(t,x)\in\Gamma +x_0$, then $(t,x-z)\in\Gamma +(x_0-z)$, thus 
    \begin{align*}
        \abs{u(t,x)}&\leq \int_{\Rn}G(z)\abs{v(t,x-z)}\dd z
        \leq \int_{\Rn}G(z) (N_{*,\Gamma}v)(x_0-z)\dd z = (G\ast N_{*,\Gamma}v)(x_0).
    \end{align*}
\end{rem}

We now have all the ingredients needed for the proof of Theorem~\ref{main theorem upper half space}.
 \begin{proof}[Proof of Theorem~\ref{main theorem upper half space}]
 We shall first prove the almost everywhere tangential convergence in \eqref{ae tangential convergence for sol of div form with bessel bdry datum} by establishing the corresponding maximal function estimate \eqref{tangential maximal function estimate in statement of main theorem}. Let $p\in(1,\infty)$ be such that $(\mathcal{D})^{L}_{p}$ holds. Let $0<\alpha \leq  \frac{n}{p}$ and $f\in\mathscr{L}^{p}_{\alpha}(\Rn)$. By linearity of the equation $Lu=0$ and sub-additivity of maximal functions, we may assume that there is $g\in\El{p}_{+}(\Rn)$ such that $f=\mathscr{J}_{\alpha}g$ (in particular $f\geq 0$ a.e. on $\Rn$). We let $\beta= 1-\frac{\alpha p}{n}>0$ if $\alpha <\frac{n}{p}$, or we let $\beta\in(0,1]$ be arbitrary if $\alpha=\frac{n}{p}$. Lemma~\ref{lemma Dirichlet holds implies well posedness} shows that the unique solution $u_{f}$ to $\left(\mathcal{D}\right)^{L}_{p}$ with boundary datum $f$ is given by 
 \begin{equation*}
     u_{f}(t,x)=\int_{\Rn} k((t,x),y)f(y)\dd y\quad \textup{ for all }(t,x)\in\Hn.
 \end{equation*}  
We decompose $u_{f}(t,x)=\sum_{j\geq 0}u_{j}(t,x)$, where 
 \begin{equation*}
 u_{j}(t,x)=\int_{\Delta(x,2^{j+1}t)\setminus \Delta(x,2^{j}t)} k((t,x),y)f(y)\dd y
 \end{equation*}
 for $j\geq 1$, and $u_{0}(t,x)=\int_{\Delta(x,2t)} k((t,x),y)f(y)\dd y$. The proof 
 of \cite[Lemma 5.32]{HLM_Degenerate_Dirichlet_2019} (see also Remark~\ref{remark about proof of lem 5.32 in HLM} above) shows that there are $1<r<p$ and $\alpha_{L}\in(0,1)$ such that 
 \begin{equation}\label{estimate u_j starting point}
     u_{j}(t,x)\lesssim 2^{-\alpha_{L}j}\left(\dashint_{\Delta(x,2^{j+1}t)}f^{r}\right)^{1/r}
 \end{equation}
 for all $(t,x)\in\Hn$ and $j\geq 0$, where $\alpha_{L}>0$ is the De Giorgi--Nash--Moser constant from \eqref{holder continuity solutions DGNM}.  
 
Let $x_0\in\Rn$ and $j\geq 0$ be fixed, and let $(t,x)\in\Gamma^{\beta}(x_0)$. 
 Let us first assume that $t\geq 2^{-j/(1-\beta)}$. If $t\leq 1$, then $t^{\beta-1}\leq 2^{j}$, thus $t^{\beta}\leq 2^{j}t$ and therefore $\Delta(x,2^{j+1}t)\subseteq \Delta(x_0, 2^{j+2}t)$. As a consequence, it follows from \eqref{estimate u_j starting point} that 
 \begin{equation}\label{estimate for u_j for large t}
 {u_{j}(t,x)\lesssim 2^{-\alpha_{L}j}\left(M_{r}f\right)(x_0)}. 
 \end{equation}
 If $t\geq 1$, then this estimate trivially holds.
 
 Let us now assume that $0<t<2^{-j/(1-\beta)}$. Let $d=\lfloor \alpha \rfloor$, so there exists $\tilde{\alpha}\in[0,1)$ such that $\alpha=d+\tilde{\alpha}$. We claim that the following estimate holds:
 \begin{equation}\label{important estimate on term j for small t}
     u_{j}(t,x)\lesssim 2^{-\alpha_{L}j}\left[\sum_{k=0}^{d}(P_{2^{j+1}t}\ast \mathcal{J}_{\alpha}\smallabs{g^{k}})(x) + (2^{j+1}t)^{\alpha}\left(\dashint_{\Delta(x,2^{j+1}t)} (M\smallabs{g^{d}})^{r}\right)^{1/r}\right],
 \end{equation}
 where the functions $g^{k}$, $k\in\set{0,\ldots,d}$, are related to $f$ in the way prescribed by Lemma \ref{lemma representation formula derivatives of Bessel potential functions} and \eqref{ponitwise domination of abs kth gradient by Bessel potential of some gk}.  
 Indeed, first note that \eqref{estimate u_j starting point} and Poincaré's inequality imply that
\begin{align*}
    u_{j}(t,x)&\lesssim 2^{-\alpha_{L}j}\left[ (f)_{\Delta(x,2^{j+1}t)} + \left(\dashint_{\Delta(x,2^{j+1}t)}\smallabs{f-(f)_{\Delta(x,2^{j+1}t)}}^{r}\right)^{1/r}\right]\\
    &\lesssim 2^{-\alpha_{L}j}\left[ (f)_{\Delta(x,2^{j+1}t)} +  2^{j+1}t\left(\dashint_{\Delta(x,2^{j+1}t)}\smallabs{\nabla f}^{r}\right)^{1/r}\right].
\end{align*}
Iterating this procedure, keeping in mind the notation \eqref{notation for higher order gradients of integer order k}, we obtain
 \begin{align}
 \begin{split}
 \label{result of iterated poincaré ineq}
     u_{j}(t,x)\lesssim 2^{-\alpha_{L}j}\left[\sum_{k=0}^{d-1}(2^{j+1}t)^{k}\left(\smallabs{\nabla^{k}f}\right)_{\Delta(x,2^{j+1}t)} + (2^{j+1}t)^{d}\left(\dashint_{\Delta(x,2^{j+1}t)}\smallabs{\nabla^{d}f}^{r}\right)^{1/r}\right],
     \end{split}
 \end{align}
 with the convention that the sum over $k$ is zero if $d=0$.
 
If $d\geq 1$, let $k\in\set{0,\ldots, d-1}$. We note that since $t<2^{-j/(1-\beta)}$, it holds that $2^{j+1}t < 2$, so \eqref{ponitwise domination of abs kth gradient by Bessel potential of some gk} and Lemma \ref{lemma domination of multiples of Poisson kernel by convolution with bessel} imply that
 \begin{align}
 \begin{split}\label{step introducing the poisson semigroup from the averages on balls}
     (2^{j+1}t)^{k}\left(\smallabs{\nabla^{k}f}\right)_{\Delta(x,2^{j+1}t)}&\leq (2^{j+1}t)^{k} \left(\mathscr{J}_{\alpha-k}\smallabs{g^{k}}\right)_{\Delta(x,2^{j+1}t)}\\
     &\lesssim (2^{j+1}t)^{k} \left(P_{2^{j+1}t}\ast \mathscr{J}_{\alpha-k}\smallabs{g^{k}}\right)(x)\\
     &\lesssim \left(P_{2^{j+1}t}\ast \mathscr{J}_{k}\left(\mathscr{J}_{\alpha-k}\smallabs{g^{k}}\right)\right)(x)\\
     &=\left(P_{2^{j+1}t}\ast \mathscr{J}_{\alpha}\smallabs{g^{k}}\right)(x),
     \end{split}
 \end{align}
 where in the second line we use that $(F)_{\Delta(y,r)}\lesssim (P_{r}\ast F)(y)$ for all non-negative functions $F\in\El{1}_{\loc}(\Rn)$, all $y\in\Rn$ and $r>0$, which follows from the fact that $\ind{B(0,1)}(y)\lesssim P_{1}(y)$ for all $y\in\Rn$.
Also note that \eqref{ponitwise domination of abs kth gradient by Bessel potential of some gk} implies that $\smallabs{\nabla^{d}f}\leq \mathscr{J}_{\tilde{\alpha}}\smallabs{g^{d}}$ pointwise almost everywhere. These observations imply that
 \begin{equation*}
     u_{j}(t,x)\lesssim 2^{-\alpha_{L}j}\left[\sum_{k=0}^{d-1}\left(P_{2^{j+1}t}\ast \mathscr{J}_{\alpha}\smallabs{g^{k}}\right)(x) + (2^{j+1}t)^{d}\left(\dashint_{\Delta(x,2^{j+1}t)}\left(\mathscr{J}_{\tilde{\alpha}}\smallabs{g^{d}}\right)^{r}\right)^{1/r}\right].
 \end{equation*}
 
If $\tilde{\alpha}=0$, we simply use the bound $\mathscr{J}_{\tilde{\alpha}}\smallabs{g^{d}}=\smallabs{g^{d}}\leq M\smallabs{g^{d}}$, which holds pointwise almost everywhere on $\Rn$ by Lebesgue's differentiation theorem, to prove \eqref{important estimate on term j for small t}. If $\tilde{\alpha}\in(0,1)$, then we use Lemma \ref{lemma Poincaré type inequality for fractional order Bessel spaces} to obtain
\begin{align*}
\left(\dashint_{\Delta(x,2^{j+1}t)}\left(\mathscr{J}_{\tilde{\alpha}}\smallabs{g^{d}}\right)^{r}\right)^{1/r}&\leq  \left(\dashint_{\Delta(x,2^{j+1}t)}\abs{\mathscr{J}_{\tilde{\alpha}}\smallabs{g^{d}}-\left(\mathscr{J}_{\tilde{\alpha}}\smallabs{g^{d}}\right)_{\Delta(x,2^{j+1}t)}}^{r}\right)^{1/r}\\
    &+ \left(\mathscr{J}_{\tilde{\alpha}}\smallabs{g^{d}}\right)_{\Delta(x,2^{j+1}t)}\\
    &\lesssim (2^{j+1}t)^{\tilde{\alpha}}\left(\dashint_{\Delta(x,2^{j+1}t)} (M\smallabs{g^{d}})^{r}\right)^{1/r} + \left(\mathscr{J}_{\tilde{\alpha}}\smallabs{g^{d}}\right)_{\Delta(x,2^{j+1}t)}.
\end{align*}
The argument of \eqref{step introducing the poisson semigroup from the averages on balls} similarly implies that
\begin{align*}
    (2^{j+1}t)^{d}\left(\mathscr{J}_{\tilde{\alpha}}\smallabs{g^{d}}\right)_{\Delta(x,2^{j+1}t)}&\lesssim (2^{j+1}t)^{d}\left(P_{2^{j+1}t}\ast \mathscr{J}_{\tilde{\alpha}}\smallabs{g^{d}}\right)(x)\\
    &\lesssim \left(P_{2^{j+1}t}\ast \mathscr{J}_{\alpha}\smallabs{g^{d}}\right)(x),
\end{align*}
and the estimate \eqref{important estimate on term j for small t} follows. 

We shall now consider the right-hand side of \eqref{important estimate on term j for small t}. We introduce the function $v:\Hn\to [0,\infty]$ defined by
\begin{equation*}
    v(s,y)=\left(\dashint_{\Delta(y,2s)} (M\smallabs{g^{d}})^{r}\right)^{1/r}\quad \textup{ for all }(s,y)\in\Hn.
\end{equation*}
Since $(t,x)\in\Gamma^{\beta}(x_0)$, it follows that $(2^{j+1}t,x)\in\Gamma^{\beta}(x_0)$. In addition, since we are working in the regime where $0<t<2^{-j/(1-\beta)}$, \eqref{Nagel Stein maximal function poisson extension of bessel potential bounded on Lp} and Lemma \ref{lemma boundedness mitigating maximal function on j dependent intervals} imply that the estimate \eqref{important estimate on term j for small t} takes the form
\begin{align}
\begin{split}\label{use of dyadic maximal function in proof of main thm}
    u_{j}(t,x)&\lesssim 2^{-\alpha_{L}j}\left[\sum_{k=0}^{d}\left(P_{2^{j+1}t}\ast \mathcal{J}_{\alpha}\smallabs{g^{k}}\right)(x) + 2^{j(\alpha -\frac{n}{p}) }t^{\alpha - \frac{n}{p}(1-\beta)}2^{\frac{jn}{p}}t^{\frac{n}{p}(1-\beta)}v(2^{j}t,x)\right]\\
    &\leq 2^{-\alpha_{L}j}\left[ \sum_{k=0}^{d} \mathcal{N}_{\alpha,\beta}(\smallabs{g^{k}})(x_0) + \left(\mathcal{M}_{p,\beta , j}v\right)(x_0)\right],
    \end{split}
\end{align}
where in the second line we use that $2^{j(\alpha -\frac{n}{p}) }\leq 1$ (since $0<\alpha\leq \frac{n}{p}$) and that $t^{\alpha - \frac{n}{p}(1-\beta)}=1$ if $\alpha<\frac{n}{p}$ and $\beta =1-\frac{\alpha p}{n}$, or that $t^{\alpha - \frac{n}{p}(1-\beta)}=t^{\beta \alpha}\leq 1$ when $\alpha =\frac{n}{p}$ and $\beta\in(0,1)$ is arbitrary.
Combining this with the estimate \eqref{estimate for u_j for large t} obtained for $t\geq 2^{-j/(1-\beta)}$, this means that 
\begin{equation*}
    \sup_{(t,x)\in\Gamma^{\beta}(x_0)} u_{j}(t,x)\lesssim 2^{-\alpha_{L}j}\left[ \sum_{k=0}^{d} \mathcal{N}_{\alpha,\beta}(\smallabs{g^{k}})(x_0) +\left(\mathcal{M}_{p,\beta , j}v\right)(x_0) + \left(M_{r}f\right)(x_0)\right].
\end{equation*}
Since $u_{f}=\sum_{j=0}^{\infty}u_{j}$ pointwise on $\Hn$, this implies that
\begin{equation*}
    N_{*,\beta}(u_{f})(x_0)\lesssim \sum_{j\geq 0}2^{-\alpha_{L}j}\left[ \sum_{k=0}^{d} \mathcal{N}_{\alpha,\beta}(\smallabs{g^{k}})(x_0) + \left(\mathcal{M}_{p,\beta , j}v\right)(x_0) + \left(M_{r}f\right)(x_0)\right].
\end{equation*}
Since $p>r>1$, we may integrate the $p$th power of both sides of this inequality with respect to $x_0\in\Rn$ and use Lemma~\ref{lemma boundedness mitigating maximal function on j dependent intervals} and Theorem~\ref{theorem Nagel--Stein thm 5} (along with the remark about the case when $\alpha=\frac{n}{p}$) to obtain
\begin{align}
\begin{split}\label{final estimate in the proof of 1st main theorem}
    \norm{ N_{*,\beta}(u_{f})}_{\Ell{p}}&\lesssim \sum_{j\geq 0}2^{-\alpha_{L}j}\left[\sum_{k=0}^{d} \smallnorm{\smallabs{g^{k}}}_{\El{p}} + \smallnorm{M\smallabs{g^{d}}}_{\El{p}} + \norm{M_{r}f}_{\El{p}}\right]\\
    &\lesssim \sum_{j\geq 0}2^{-\alpha_{L}j}\left(\norm{g}_{\El{p}} + \norm{f}_{\El{p}}\right)\lesssim \norm{g}_{\El{p}}=\norm{f}_{\mathscr{L}^{p}_{\alpha}(\Rn)},
    \end{split}
\end{align}
where we have also used Lemma~\ref{lemma representation formula derivatives of Bessel potential functions} and the maximal theorem to obtain the second inequality. 

We now turn to the proof of the estimate \eqref{upper bound hausdorff dimension of divergence set} on the Hausdorff dimension of the divergence set. If $\beta'\in(\beta,1]$, then there is $\alpha'\in [0,\alpha)$ such that $\beta'=1-\frac{\alpha' p}{n}$. We can write $f=\mathscr{J}_{\alpha}(g)=G_{\alpha-\alpha'}\ast \mathscr{J}_{\alpha'}g$ for some $g\in\El{p}(\Rn)$. Starting from \eqref{estimate u_j starting point}, and using Minkowski's inequality as in the proof of Lemma~\ref{lemma maximal function commutes with convolution}, shows that for all $(t,x)\in\Hn$ it holds that 
\begin{align}
\begin{split}\label{estimate for proof est hausdorff dim div set for bessel}
    \abs{u_{f}(t,x)}&\lesssim \sum_{j=0}^{\infty}2^{-\alpha_{L}j}\left(\dashint_{\Delta(x,2^{j+1}t)}\abs{G_{\alpha-\alpha'}\ast \mathscr{J}_{\alpha'}g}^{r}\right)^{1/r}\\
    &\leq \sum_{j=0}^{\infty}2^{-\alpha_{L}j}\left[\int_{\Rn}G_{\alpha-\alpha'}(x-z)\left(\dashint_{\Delta(z,2^{j+1}t)}\abs{\mathscr{J}_{\alpha'}g}^{r}\right)^{1/r}\dd z\right]\\
    &=\int_{\Rn}G_{\alpha-\alpha'}(x-z)\left[\sum_{j=0}^{\infty} 2^{-\alpha_{L}j} \left(\dashint_{\Delta(z,2^{j+1}t)}\abs{\mathscr{J}_{\alpha'}g}^{r}\right)^{1/r}\right]\dd z\\
    &=(G_{\alpha-\alpha'}\ast w(t,\cdot))(x),
    \end{split}
\end{align}
where we have introduced the function $w:\Hn\to [0,\infty]$ defined by
\begin{equation*}
    w(t,x)=\sum_{j=0}^{\infty} 2^{-\alpha_{L}j} \left(\dashint_{\Delta(x,2^{j+1}t)}\abs{\mathscr{J}_{\alpha'}g}^{r}\right)^{1/r} \quad \textup{ for all } (t,x)\in\Hn.
\end{equation*}
The first part of the proof (see \eqref{final estimate in the proof of 1st main theorem}) shows that $\norm{N_{*,\beta'}w}_{\Ell{p}}\lesssim \norm{g}_{\Ell{p}}$. Also note that \eqref{estimate for proof est hausdorff dim div set for bessel} and Remark~\ref{remark 2 - convolution commute  with general maximal functions approach} show that 
\begin{align*}
    N_{*,\beta'}u_{f}(x_0)\lesssim \left(G_{\alpha-\alpha'}\ast N_{*,\beta'}w\right)(x_0)\quad \textup{ for all }x_{0}\in\Rn.
\end{align*}
Consequently, for any $s> n-(\alpha-\alpha')p$ and $\mu\in\mathcal{M}^{s}(\Rn)$, it follows from Lemma~\ref{lemma dimensional measures and Bessel potential functions} that 
\begin{equation}\label{fractional measure integral estimate on the tangential maximal function in the proof of main thm}
    \norm{N_{*,\beta'}(u_f)}_{\El{1}(\mu)}\lesssim \norm{G_{\alpha-\alpha'}\ast N_{*,\beta'}w}_{\El{1}(\mu)}\lesssim \norm{N_{*,\beta'}w}_{\Ell{p}}\lesssim \norm{g}_{\Ell{p}}=\norm{f}_{\mathscr{L}^{p}_{\alpha}(\Rn)},
\end{equation}
where the implicit constant also depends on $c_{s}(\mu)$.
The argument of the proof of Proposition~\ref{proposition Hausdorff dimension div set for cones and Bessel} can now be used with the estimate~\eqref{fractional measure integral estimate on the tangential maximal function in the proof of main thm} to obtain the estimate~\eqref{upper bound hausdorff dimension of divergence set} on the Hausdorff measure of the divergence set (observe that $n-(\alpha-\alpha')p=n-n(\beta' -\beta)$). 
\end{proof}
\subsection{Further properties of the spaces \texorpdfstring{$\mathrm{C}^{p}_{\alpha}(\Rn)$}{Cpalpha} and an extension of Theorem~\ref{main theorem upper half space}}\label{section proof of main theorem for Wsp in upper half space}
The aim of this section is to prove the following generalisation of Theorem~\ref{main theorem upper half space} that permits boundary data in the fractional mean oscillation spaces $\mathrm{F}^{p}_{\alpha}(\Rn)\subseteq\mathrm{C}^{p}_{\alpha}(\Rn)$ instead of $\mathscr{L}^{p}_{\alpha}(\Rn)$. We recall that the spaces $\mathrm{F}^{p}_{\alpha}(\Rn)$ are defined at the end of Section~\ref{subsection with definitions of spaces}, and coincide with the closure of $C^{\infty}_{c}(\Rn)$ in $\mathrm{C}^{p}_{\alpha}(\Rn)$.
\begin{thm}\label{main theorem upper half space for Cpalpha}
    Let $p\in (1,\infty)$ be such that $\left(\mathcal{D}\right)^{L}_{p}$ holds. Let $\alpha > 0$ be such that $\alpha p\leq n$, and let $\beta=1-\frac{\alpha p}{n}$. Let $f\in\mathrm{F}^{p}_{\alpha}(\Rn)$ and let $u_{f}$ be the unique solution to $\left(\mathcal{D}\right)^{L}_{p}$ with boundary datum $f$.
    For any $\beta'\in(\beta,1]$ it holds that
    \begin{equation}\label{upper bound hausdorff dimension of divergence set for Cpalpha}
        \dim_{H}\left(\set{x_0\in\Rn : \lim_{\Gamma^{\beta'}(x_0)\ni (t,x)\to(0,x_0)} u_{f}(t,x)\neq f(x_0)}\right)\leq n-n(\beta' -\beta).
    \end{equation}
    If $\alpha p<n$, then 
    \begin{equation}\label{tangential maximal function estimate in statement of main theorem for Calphap}
    \norm{N_{*,\beta}(u_{f})}_{\Ell{p}}\lesssim \norm{f}_{\mathrm{C}^{p}_{\alpha}(\Rn)}
    \end{equation}
    and
    \begin{equation}\label{ae tangential convergence for sol of div form with frac sobolev bdry datum}
        \lim_{\Gamma^{\beta}(x_0)\ni (t,x)\to(0,x_0)} u_{f}(t,x)=f(x_0) \quad \textup{ for a.e. }x_{0}\in\Rn.
    \end{equation}
    If $\alpha p=n$, then \eqref{tangential maximal function estimate in statement of main theorem for Calphap} and \eqref{ae tangential convergence for sol of div form with frac sobolev bdry datum} hold for all $\beta>0$.
\end{thm}

The proof of Theorem~\ref{main theorem upper half space for Cpalpha} will show that the estimate \eqref{tangential maximal function estimate in statement of main theorem for Calphap} holds for all $f\in\mathrm{C}^{p}_{\alpha}(\Rn)$.
As in Theorem~\ref{main theorem upper half space}, a preferred representative of $f\in\mathrm{F}^{p}_{\alpha}(\Rn)$ has to be chosen for the estimate \eqref{upper bound hausdorff dimension of divergence set for Cpalpha} to hold. We again choose the representative $\tilde{f}:\Rn\to(-\infty,\infty]$ defined in \eqref{preferred representative for bessel potential space function}, but we need a few lemmas to show that this representative has the properties required for the proof of Proposition~\ref{proposition Hausdorff dimension div set for cones and Bessel} to go through. The most important result in this regard is Lemma~\ref{lemma existence of limit of averages away from a small set for Cpalpha} below (see also Remark~\ref{remark properties particular representative of f in cpalpha}). Once the relevant properties of the representative $\tilde{f}$ have been established, we will state an elementary Poincaré-type inequality for the spaces $\mathrm{C}^{p}_{\alpha}(\Rn)$ (Lemma~\ref{poincaré type lemma for C p a}) which will be needed in the proof of Theorem~\ref{main theorem upper half space for Cpalpha} at the end of the section. 

We start with the following important observation due to Dorronsoro (see \cite[Theorem 4]{Dorronsoro_1986}). 
\begin{lem}\label{lemma lifting property bessel operator}
For $p\in(1,\infty)$, $\alpha>0$ and $\beta\geq 0$, the operator $\mathscr{J}_{\beta} : \Ell{p}\to\Ell{p}$ defined in Section~\emph{\ref{subsection with definitions of spaces}} maps the space $\mathrm{C}^{p}_{\alpha}(\Rn)$ isomorphically onto $\mathrm{C}^{p}_{\alpha +\beta}(\Rn)$. In particular, if $f\in\mathrm{C}^{p}_{\alpha+\beta}(\Rn)$ then there is a unique $g\in\mathrm{C}^{p}_{\alpha}(\Rn)$ such that $f=\mathscr{J}_{\beta}g$ and $\norm{g}_{\mathrm{C}^{p}_{\alpha}(\Rn)}\eqsim \norm{f}_{\mathrm{C}^{p}_{\alpha+\beta}(\Rn)}$.
\end{lem}
The following analogue of Lemma~\ref{lemma dimensional measures and Bessel potential functions} is the main ingredient that will be used in the proof of Lemma~\ref{lemma existence of limit of averages away from a small set for Cpalpha}.  
\begin{lem}\label{lemma estimate Cpalpha functions against fractional measures}
    Let $p\in(1,\infty)$, $\alpha\in(0,\frac{n}{p}]$ and $f\in\mathrm{C}^{p}_{\alpha}(\Rn)$. If $s>n-\alpha p$ and $\mu\in\mathcal{M}^{s}(\Rn)$ then $\norm{Mf}_{\El{1}(\mu)}\lesssim \norm{f}_{\mathrm{C}^{p}_{\alpha}(\Rn)}$.
\end{lem}
\begin{proof}
    It follows from \cite[Theorem 5]{Dorronsoro_1986} that $\abs{f(y)}\lesssim (I_{\alpha}\ast S_{\alpha}f)(y)$ for a.e. $y\in\Rn$. Combining this with \cite[Lemma 3]{Dorronsoro_1986} shows that 
    \begin{equation}\label{domination maximal function by riesz potential of sharph function}
    Mf(x)\lesssim (I_{\alpha}\ast S_{\alpha}f)(x)
    \end{equation}
    for \emph{all} $x\in\Rn$. If $\alpha\in(0,\frac{n}{p})$, then (the first part of the proof of) Lemma~\ref{lemma dimensional measures and Bessel potential functions} shows that 
    \begin{equation*}
        \norm{Mf}_{\El{1}(\mu)}\lesssim \norm{I_{\alpha}\ast S_{\alpha}f}_{\El{1}(\mu)}\lesssim \norm{S_{\alpha}f}_{\Ell{p}}\lesssim \norm{f}_{\mathrm{C}^{p}_{\alpha}(\Rn)}.
    \end{equation*}
    If $\alpha=\frac{n}{p}$, then $s>0$ and we may choose $\gamma\in(0,\frac{n}{p})$ and $\eps>0$ such that $\alpha=\gamma +\eps$ and $s>n-\gamma p$. Lemma~\ref{lemma lifting property bessel operator} shows that there exists $g\in\mathrm{C}^{p}_{\gamma}(\Rn)$ such that $f=G_{\eps}\ast g$ and $\norm{g}_{\mathrm{C}^{p}_{\gamma}(\Rn)}\eqsim \norm{f}_{\mathrm{C}^{p}_{\alpha}(\Rn)}$. It follows from Lemma~\ref{lemma maximal function commutes with convolution} that 
    \begin{equation*}
        Mf(x)\leq (G_{\eps}\ast Mg)(x)\lesssim (I_{\gamma}\ast G_{\eps}\ast S_{\gamma}g)(x) 
    \end{equation*}
    for all $x\in\Rn$, and therefore the same argument as above yields
    \begin{equation}\label{final estimate in proof of lemma domaination max function fractional measure}
        \norm{Mf}_{\El{1}(\mu)}\lesssim \norm{G_{\eps}\ast S_{\gamma}g}_{\Ell{p}}\leq \norm{S_{\gamma}g}_{\Ell{p}}\lesssim \norm{g}_{\mathrm{C}^{p}_{\gamma}(\Rn)}\eqsim \norm{f}_{\mathrm{C}^{p}_{\alpha}(\Rn)}.\qedhere
    \end{equation}
\end{proof}
The previous lemma can be combined with a well-known approximation argument to obtain the following result. We spell out the details in the interests of clarity.
\begin{lem}\label{lemma existence of limit of averages away from a small set for Cpalpha}
    Let $p\in(1,\infty)$, $\alpha\in(0,\frac{n}{p}]$ and $f\in\mathrm{F}^{p}_{\alpha}(\Rn)$. There exists a set $E\subseteq \Rn$ such that $\lim_{r\to 0^{+}}\dashint_{\Delta(x,r)}f(y)\dd y$ exists for all $x\in\Rn\setminus E$, and $\mu(E)=0$ for all $\mu\in\mathcal{M}^{s}(\Rn)$ with $s>n-\alpha p$.
\end{lem}
\begin{proof}
    For $R>0$ and $f\in\El{1}_{\loc}(\Rn)$, consider the maximal function
    \begin{equation*}
        D_{R}f(x)=\sup_{0<r<R}\abs{\dashint_{\Delta(x,R)}f(y)\dd y-\dashint_{\Delta(x,r)}f(y)\dd y}
    \end{equation*}
    for all $x\in\Rn$. The operator $D_{R}$ has the following elementary properties:
    \begin{enumerate}[label=(\roman*)]
        \item $D_{R}(f+g)(x)\leq D_{R}f(x) + D_{R}g(x)$ for all $f,g\in\El{1}_{\loc}(\Rn)$, $R>0$ and $x\in\Rn$;
        \item $D_{R}f(x)\leq 2Mf(x)$ for all $f\in\El{1}_{\loc}(\Rn)$, $R>0$ and $x\in\Rn$;
        \item If $g$ is uniformly continuous on $\Rn$, then for all $\eps>0$ there exists $\delta>0$ such that $D_{R}g(x)<\eps$ for all $R\in(0,\delta)$ and $x\in\Rn$. 
    \end{enumerate}
    Let $\eps>0$ and $\lambda>0$ be fixed. Since $C^{\infty}_{c}(\Rn)$ is dense in $\mathrm{F}^{p}_{\alpha}(\Rn)$ (see \cite[Proposition 1]{Dorronsoro_1986}), there exists $g\in C^{\infty}_{c}(\Rn)$ such that $\norm{f-g}_{\mathrm{C}^{p}_{\alpha}(\Rn)}<\lambda$. Let $h=f-g\in\mathrm{F}^{p}_{\alpha}(\Rn)$, and note that $D_{R}f(x)\leq 2Mh(x) + D_{R}g(x)$ for all $R>0$ and $x\in\Rn$ by properties (i) and (ii). Since $g$ is uniformly continuous on $\Rn$, property (iii) provides some $\delta=\delta(\eps, g)>0$ such that $\sup_{0<R<\delta} D_{R}g(x)<\eps/2$ for all $x\in\Rn$. It follows that
    \begin{equation}\label{basic inclusion for lemma pointwise extension of Cpalpha functions}
         \set{x\in\Rn :\sup_{0<R<\delta}D_{R}f(x)>\eps}\subseteq \set{x\in\Rn : Mh(x)>\eps/4}.
    \end{equation}
    Observe that the map $\eta\mapsto \sup_{0<R<\eta}D_{R}f(x)$ is non-decreasing for all fixed $x\in\Rn$.
    Consequently, if $s>n-\alpha p$ and $\mu\in\mathcal{M}^{s}(\Rn)$, then \eqref{basic inclusion for lemma pointwise extension of Cpalpha functions} and Lemma~\ref{lemma estimate Cpalpha functions against fractional measures} show that
    \begin{equation*}
        \mu\left(\set{x\in\Rn : \lim_{\eta\to 0^{+}}\sup_{0<R<\eta} D_{R}f(x)>\eps}\right)\lesssim \eps^{-1}\norm{Mh}_{\El{1}(\mu)}\lesssim \eps^{-1}\norm{h}_{\mathrm{C}^{p}_{\alpha}(\Rn)}< \lambda\eps^{-1}.
    \end{equation*} 
    Since $\lambda>0$ was arbitrary, the set $E_{\eps}=\set{x\in\Rn : \lim_{\eta\to 0^{+}}\sup_{0<R<\eta} D_{R}f(x)>\eps}$ satisfies $\mu(E_{\eps})=0$. Let $E=\cup_{k\geq 1}E_{1/k}$. Then $\mu(E)=0$ and $\lim_{\eta\to 0^{+}}\sup_{0<R<\delta}D_{R}f(x)=0$ for all $x\in\Rn\setminus E$, i.e. $\set{\dashint_{\Delta(x,r)}f}_{0<r<1}$ is a Cauchy net as $r\to 0^{+}$.
    \end{proof}
\begin{rem}\label{remark properties particular representative of f in cpalpha}
    For later reference we point out that Lemma~\ref{lemma existence of limit of averages away from a small set for Cpalpha} shows that for a given $f\in\mathrm{F}^{p}_{\alpha}(\Rn)$, there exists a set $E\subset\Rn$ such that  representative $\tilde{f}:\Rn\to (-\infty,\infty]$ defined in~\eqref{preferred representative for bessel potential space function} has the property that $\smallabs{\tilde{f}(x)}\lesssim Mf(x)$ for all $x\in\Rn\setminus E$, and $\mu(E)=0$ for all $\mu\in\mathcal{M}^{s}(\Rn)$ with $s>n-\alpha p$. In particular, it follows from Lemma~\ref{lemma estimate Cpalpha functions against fractional measures} that $\smallnorm{\tilde{f}}_{\El{1}(\mu)}\lesssim \norm{f}_{\mathrm{C}^{p}_{\alpha}(\Rn)}$.
\end{rem}

The main ingredient for the proof of Theorem~\ref{main theorem upper half space for Cpalpha} is the following elementary Poincaré-type inequality which is analogous to that of Lemma~\ref{lemma Poincaré type inequality for fractional order Bessel spaces} for the spaces $\mathrm{C}^{p}_{\alpha}(\Rn)$ (see, e.g.,  \cite[p. 25]{Devore-Sharpley_Maximal_functions_smoothness_1984}).
\begin{lem}\label{poincaré type lemma for C p a}
    Let $\alpha>0$, $k=\lfloor \alpha \rfloor$ and $f\in\El{1}_{\loc}(\Rn)$. For all balls  $\Delta\subset\Rn$ it holds that
\begin{equation*}
    \smallabs{f(y) - P_{\Delta}^{k}f(y)}\lesssim \meas{\Delta}^{\alpha/n}f_{\alpha}^{\sharp}(y)
\end{equation*}
for almost every $y\in\Delta$. Consequently, for all $q\in[1,\infty)$ and all balls $\Delta\subset\Rn$ of radius $r>0$ it holds that  
\begin{equation*}
    \left(\dashint_{\Delta}\smallabs{f(y)-P_{\Delta}^{k}f(y)}^{q}\right)^{1/q}\lesssim r^{\alpha}\left(\dashint_{\Delta}f_{\alpha}^{\sharp}(y)^{q}\dd y\right)^{1/q}.
\end{equation*}
\end{lem}
This inequality can be improved (see, e.g., \cite[Proposition 2]{Dorronsoro_1986}) but this is sufficient for our purposes. 

Finally, we make the following important remark about the decomposition of a function in the space $\mathrm{C}^{p}_{\alpha}(\Rn)$ as the difference of two non-negative functions. As we shall see momentarily, this allows matters to be reduced to nonnegative $f$, which is convenient for the purposes of applying Theorem \ref{Dorronsoros theorem}.
\begin{rem}\label{rem signed decomposition of sobolev space Cpa}
    Let $\alpha>0$, $p\in(1,\infty)$ and $f\in\mathrm{C}^{p}_{\alpha}(\Rn)$. There are $\alpha'\in (0,1)$ and $\beta>0$ such that $\alpha =\beta +\alpha'$. Lemma~\ref{lemma lifting property bessel operator} shows that there is $g\in\mathrm{C}^{p}_{\alpha'}(\Rn)$ such that $f=\mathscr{J}_{\beta}g$ and $\norm{g}_{\mathrm{C}^{p}_{\alpha'}(\Rn)}\eqsim \norm{f}_{\mathrm{C}^{p}_{\alpha}(\Rn)}$. We may decompose $g=g^{+}-g^{-}$, where $g^{\pm}(x)=\max\set{\pm g(x),0}$ for all $x\in\Rn$. Since $\alpha'\in(0,1)$, the triangle inequality shows that $g^{\pm}\in\mathrm{C}^{p}_{\alpha'}(\Rn)$ with $\smallnorm{g^{\pm}}_{\mathrm{C}^{p}_{\alpha'}(\Rn)}\lesssim \norm{g}_{\mathrm{C}^{p}_{\alpha'}(\Rn)}$ (see, e.g., the proof of \cite[Proposition 3.1.2]{ModernFourier_2014} for the corresponding result in the space $\mathrm{BMO}(\Rn)$). Let $f_{\pm}=\mathscr{J}_{\beta}g^{\pm}$. It follows from the linearity of $\mathscr{J}_{\beta}$, the nonnegativity of the Bessel kernel and Lemma~\ref{lemma lifting property bessel operator} that $f=f_{+}-f_{-}$ where $f_{\pm}(x)\geq 0$ for a.e. $x\in\Rn$ and $\norm{f_{\pm}}_{\mathrm{C}^{p}_{\alpha}(\Rn)}\lesssim \norm{f}_{\mathrm{C}^{p}_{\alpha}(\Rn)}$.
\end{rem}

We now have all the ingredients needed for the proof of Theorem~\ref{main theorem upper half space for Cpalpha}. The argument will follow closely that of the proof of Theorem~\ref{main theorem upper half space} and for this reason we will omit some of the computations which were carried out in that proof. 
\begin{proof}[Proof of Theorem~\ref{main theorem upper half space for Cpalpha}]
    Let $p\in(1,\infty)$ be such that $(\mathcal{D})^{L}_{p}$ holds. Let $0<\alpha \leq  \frac{n}{p}$ and $f\in\mathrm{C}^{p}_{\alpha}(\Rn)$. Without loss of generality we may assume that  $f\geq 0$ (see Remark~\ref{rem signed decomposition of sobolev space Cpa}). We let $\beta= 1-\frac{\alpha p}{n}>0$ if $\alpha <\frac{n}{p}$, or we let $\beta\in(0,1]$ be arbitrary if $\alpha=\frac{n}{p}$. As in the proof of Theorem~\ref{main theorem upper half space}, the unique solution $u_f$ to $(\mathcal{D})^{L}_{p}$ with boundary datum $f$ can be decomposed as $u_{f}(t,x)=\sum_{j\geq 0}u_{j}(t,x)$, where there are $1<r<p$ and $\alpha_{L}\in(0,1)$ such that the functions $u_j$ satisfy the estimate~\eqref{estimate u_j starting point} for all $j\geq 0$. Let us fix $x_0\in\Rn$, $j\geq 0$ and $(t,x)\in\Gamma^{\beta}(x_0)$. As previously, if $t\geq 2^{-j/(1-\beta)}$, then \eqref{estimate for u_j for large t} holds.

    Let us now assume that $0<t <2^{-j/(1-\beta)}$, and let $k=\lfloor\alpha  \rfloor\in\N$. Lemma~\ref{poincaré type lemma for C p a} and the estimate \eqref{local L infty estimate on minimising polynomials} can be used to write
    \begin{align*}
        \abs{u_{j}(t,x)}&\lesssim 2^{-\alpha_{L}j}\left[\left(\dashint_{\Delta(x,2^{j+1}t)}\smallabs{f-P^{k}_{\Delta(x,2^{j+1}t)}f}^{r}\right)^{1/r} + \left(\dashint_{\Delta(x,2^{j+1}t)}\smallabs{P^{k}_{\Delta(x,2^{j+1}t)}f}^{r}\right)^{1/r}\right]\\
        &\lesssim 2^{-\alpha_{L}j}\left[ (2^{j+1}t)^{\alpha}\left(\dashint_{\Delta(x,2^{j+1}t)} f_{\alpha}^{\sharp}(y)^{r}\dd y\right)^{1/r} + \dashint_{\Delta(x,2^{j+1}t)}f\right]\\
        &\lesssim 2^{-\alpha_{L}j}\left[ (2^{j+1}t)^{\alpha}w(2^{j}t,x) + (P_{2^{j+1}t}\ast f)(x)\right],
    \end{align*}
    where we have introduced the function $w:\Hn\to[0,\infty]$ defined by 
    \begin{equation*}
        w(s,y)=\left(\dashint_{\Delta(y,2s)} f_{\alpha}^{\sharp}(z)^{r}\dd z\right)^{1/r} \quad \textup{ for all }(s,y)\in\Hn.
        \end{equation*}
    As in the proof of Theorem~\ref{main theorem upper half space} (see in particular \eqref{use of dyadic maximal function in proof of main thm}), the maximal function introduced in \eqref{maximal function dilated in the t direction for j integer} can be used to obtain
    \begin{align*}
        \abs{u_{j}(t,x)}\lesssim 2^{-\alpha_{L}j}\left[ (\mathcal{M}_{p,\beta,j}w)(x_0) + (N_{*,\beta}\tilde{u}_{f})(x_0)\right],
    \end{align*}
    where $\tilde{u}_{f}$ is the classical Poisson extension of $f$. Combining with the estimate for $t\geq 2^{-j/(1-\beta)}$, this shows that 
    \begin{equation}\label{introduction of the big maximal function for later use}
        N_{*,\beta}u_{f}(x_0)\lesssim \sum_{j\geq 0} 2^{-\alpha_{L}j}\left[(\mathcal{M}_{p,\beta,j}w)(x_0) + (N_{*,\beta}\tilde{u}_{f})(x_0) + M_{r}f(x_0)\right]=:\mathring{\mathcal{M}}_{p,r}^{\beta}(f),
    \end{equation}
    where we introduce the maximal function $\mathring{\mathcal{M}}_{p,r}^{\beta}(f)$ for later reference. Since $p>r$, it follows from Lemma~\ref{lemma boundedness mitigating maximal function on j dependent intervals}, Theorem~\ref{Dorronsoros theorem} and the Hardy--Littlewood maximal theorem that 
    \begin{equation}\label{final estimate proof of main theorem for cpalpha}
        \norm{N_{*,\beta}u_{f}}_{\Ell{p}}\lesssim \smallnorm{f_{\alpha}^{\sharp}}_{\Ell{p}} + \norm{f}_{\mathrm{C}^{p}_{\alpha}(\Rn)} + \norm{f}_{\Ell{p}}\lesssim \norm{f}_{\mathrm{C}^{p}_{\alpha}(\Rn)}.
    \end{equation}
    We recall that $C^{\infty}_{c}(\Rn)$ is dense in $\mathrm{F}^{p}_{\alpha}(\Rn)$. Therefore, if $f\in\mathrm{F}^{p}_{\alpha}(\Rn)$, then the classical approximation argument can be used along with \eqref{final estimate proof of main theorem for cpalpha} to obtain the almost everywhere convergence \eqref{ae tangential convergence for sol of div form with frac sobolev bdry datum}.
    
    We now turn to the proof of the estimate \eqref{upper bound hausdorff dimension of divergence set for Cpalpha} on the Hausdorff dimension of the divergence set. If $\beta'\in(\beta,1]$, then there is $\alpha'\in [0,\alpha)$ such that $\beta'=1-\frac{\alpha' p}{n}$. If $f\in\mathrm{C}^{p}_{\alpha}(\Rn)$, then Lemma~\ref{lemma lifting property bessel operator} shows that there is $g\in\mathrm{C}^{p}_{\alpha'}(\Rn)$ such that $f=G_{\alpha-\alpha'}\ast g$ and $\norm{g}_{\mathrm{C}^{p}_{\alpha'}(\Rn)}\eqsim \norm{f}_{\mathrm{C}^{p}_{\alpha}(\Rn)}$. As in the proof of Theorem~\ref{main theorem upper half space}, note that $\abs{u_{f}(t,x)}\lesssim (G_{\alpha-\alpha'}\ast w(t,\cdot))(x)$ for all $(t,x)\in\Hn$, where
    \begin{equation*}
        w(t,x)=\sum_{j=0}^{\infty} 2^{-\alpha_{L}j} \left(\dashint_{\Delta(x,2^{j+1}t)}\abs{g}^{r}\right)^{1/r} \quad \textup{ for all } (t,x)\in\Hn.
    \end{equation*}
    The first part of the proof shows that $\norm{N_{*,\beta'}w}_{\Ell{p}}\lesssim \norm{g}_{\mathrm{C}^{p}_{\alpha'}(\Rn)}\eqsim \norm{f}_{\mathrm{C}^{p}_{\alpha}(\Rn)}$, and therefore Remark~\ref{remark 2 - convolution commute  with general maximal functions approach} and Lemma~\ref{lemma dimensional measures and Bessel potential functions} can be used to show that $\norm{N_{*,\beta'}u_{f}}_{\El{1}(\mu)}\lesssim \norm{f}_{\mathrm{C}^{p}_{\alpha}(\Rn)}$ for all $s> n -(\alpha-\alpha')n$ and $\mu\in\mathcal{M}^{s}(\Rn)$. If now $f\in\mathrm{F}^{p}_{\alpha}(\Rn)$, then Remark~\ref{remark properties particular representative of f in cpalpha} shows that the argument of the proof of  Proposition~\ref{proposition Hausdorff dimension div set for cones and Bessel} can be used to prove the estimate~\eqref{upper bound hausdorff dimension of divergence set for Cpalpha}. 
\end{proof}
\begin{rem}\label{remark main implication of main thm within proof}
    For later reference we make the following two observations:
    \begin{enumerate}[label=(\roman*), wide]
        \item Let $p\in(1,\infty)$, $s>0$ and $f\in \mathrm{C}^{p}_{s}(\Rn)$. The proof of Theorem~\ref{main theorem upper half space for Cpalpha} above shows that if a function $v_{f}:\Hn\to \R$ is such that there are $1<r<p$ and $\alpha>0$ such that 
    \begin{align*}
        \abs{v_{f}(t,x)}\lesssim \sum_{j\geq 0} 2^{-\alpha j} \left(\dashint_{\Delta(x,2^{j+1}t)}\abs{f}^r\right)^{1/r}
    \end{align*}
    for all $(t,x)\in\Hn$, then if $0<s<\frac{n}{p}$ and $\beta=1-\frac{s p}{n}$, or if $s p=n$ and $\beta>0$ is arbitrary, then 
    \begin{equation*}
        N_{*,\beta}(v_f)(x_0)\lesssim \mathring{\mathcal{M}}_{p,r}^{\beta}(f_{+})(x_0) +\mathring{\mathcal{M}}_{p,r}^{\beta}(f_{-})(x_0)=: \mathcal{M}_{p,r}^{\beta}(f)(x_0)\quad \textup{for all }x_{0}\in\Rn,
    \end{equation*}
    where we have decomposed $f=f_{+}-f_{-}$ as in Remark~\ref{rem signed decomposition of sobolev space Cpa} (in particular, the functions $f_{\pm}$ should not be confused with the positive and negative parts $f^{\pm}$), and the maximal function $\mathring{\mathcal{M}}_{p,r}^{\beta}(f)$ introduced in \eqref{introduction of the big maximal function for later use} satisfies  $\smallnorm{\mathring{\mathcal{M}}_{p,r}^{\beta}(f_{\pm})}_{\Ell{p}}\lesssim \norm{f_{\pm}}_{\mathrm{C}^{p}_{s}(\Rn)}\lesssim\norm{f}_{\mathrm{C}^{p}_{s}(\Rn)}$. 

\item The space $\mathrm{W}^{s,p}(\Rn)$ is continuously embedded in $\mathrm{C}^{p}_{s}(\Rn)$ for all $p\in(1,\infty)$ and $s>0$. Indeed, \cite[Proposition 3]{Dorronsoro_1986} shows that the Triebel--Lizorkin spaces $\mathrm{F}^{s}_{p,q}(\Rn)$ are continuously embedded in $\mathrm{C}^{p}_{s}(\Rn)$ for all $1\leq p,q<\infty$ and $s>0$. In addition, \cite[Section 2.3.2, Proposition 2]{Triebel_Function_spaces_1978} shows that the spaces $\mathrm{F}^{s}_{p,p}(\Rn)$ coincide (up to equivalent norms) with the Besov spaces $\mathrm{B}^{s}_{p,p}(\Rn)$ for all $1\leq p<\infty$, and \cite[Theorem 2.5.7]{Triebel_Function_spaces_1978} shows that the Besov spaces $\mathrm{B}^{s}_{p,q}(\Rn)$ coincide with the classical Besov (Lipschitz) spaces $\Lambda^{s}_{p,q}(\Rn)$ (which are also treated, for example, in \cite[Chapter \RNum{5}, Section 5]{Stein_Singular_Integrals}) for all $1\leq p,q<\infty$ and $s>0$. If $p\in[1,\infty)$ and $s>0$ is not an integer, then \cite[Section 2.2.2, Remark 3]{Triebel_Function_spaces_1978} shows that $\mathrm{W}^{s,p}(\Rn)=\Lambda^{s}_{p,p}(\Rn)$. If $p\in(1,\infty)$ and $s>0$ is an integer, then \cite[Section 2.5.6]{Triebel_Function_spaces_1978} shows that $\mathrm{W}^{s,p}(\Rn)=\mathscr{L}^{p}_{s}(\Rn)=\mathrm{F}^{s}_{p,2}(\Rn)$. In summary,
\begin{equation*}
    \mathrm{W}^{s,p}(\Rn)=\left\{
    \begin{array}{ll}
    \Lambda^{s}_{p,p}(\Rn)=\mathrm{B}^{s}_{p,p}(\Rn)=\mathrm{F}^{s}_{p,p}(\Rn) & \textup{ if } s\not\in\Z \textup{ and }p\in[1,\infty);\\
    \mathscr{L}^{p}_{s}(\Rn)= \mathrm{F}^{s}_{p,2}(\Rn)& \textup{ if } s\in\Z \textup{ and }p\in(1,\infty).
    \end{array}
    \right.
\end{equation*}
This shows that $\mathrm{W}^{s,p}(\Rn)\subseteq\mathrm{C}^{p}_{s}(\Rn)$ for all $p\in(1,\infty)$ and $s>0$. In particular, the estimate obtained in (i) shows that $\smallnorm{\mathcal{M}_{p,r}^{\beta}(f)}_{\Ell{p}}\lesssim\norm{f}_{\mathrm{W}^{s,p}(\Rn)}$.
        \end{enumerate}
\end{rem}
\section{The Dirichlet problem in bounded Lipschitz domains and the proof of Theorem \ref{second main theorem bounded Lipschitz domain}}
\label{subsection proof second main theorem}
This section is concerned with the proof of Theorem~\ref{second main theorem bounded Lipschitz domain} (in Sections~\ref{section first part of the proof of second main thm} and \ref{section second part of the proof of second main thm} below), which shall essentially be reduced to an application of Remark~\ref{remark main implication of main thm within proof} after a suitable localisation procedure. This is achieved by relying on a number of lemmas which are collected in the following subsections. Although the tools developed here are natural and widely understood when the approach is nontangential (see, e.g., \cite{Kenig_cbms_1994}), we take the opportunity to include full technical details to treat tangential approach in Lipschitz domains for Sobolev boundary data. We start by introducing the notation that we shall use. 
\subsection{Lipschitz domains}\label{subsection with definition of Lipschitz domains}
Let $\Omega\subset \R^{1+n}$ be a bounded open set. 
For $X\in\R^{1+n}$ and $r>0$, we let $B(X,r)$ be the Euclidean ball of radius $r$ in $\R^{1+n}$. If $Q\in\partial\Omega$ (not to be confused with our use of $Q$ for a cube in earlier sections) then we let $\Delta_{\partial\Omega}(Q,r)= B(Q,r)\cap\partial\Omega$ be the corresponding surface ball.

We say that $\Omega$ is a Lipschitz domain if every point on the boundary $\partial\Omega$ has a neighbourhood on which, in new coordinates obtained from the canonical ones by a rigid motion, $\Omega$ coincides with the epigraph of a Lipschitz function.  More precisely, we assume that there is $M>0$ (referred to as the Lipschitz constant of $\Omega$) such that for all $P\in\partial\Omega$ there exist
\begin{enumerate}[label=(\roman*)]
    \item a real number $r_{P}>0$;
    \item a Lipschitz function $\varphi_{P}:\Rn\to \R$ with Lipschitz constant (at most) $M$, such that $\varphi_{P}(0)=0$;
    \item an isometry $T_{P}:\R^{1+n}\to\R^{1+n}$ with $T_{P}(0)=P$,
    \end{enumerate}
such that the special Lipschitz domain $\Omega_{P}=\set{(t,x)\in\R^{1+n} : t>\varphi_{P}(x)}$ and the cylinder 
\begin{equation*}
R_{P}=R(r_P)=\set{(t,x)\in\R^{1+n} : \abs{x}<r_P \textup{ and }\abs{t}<(1+M)r_{P}}
\end{equation*}
are such that 
\begin{equation*}
    \Omega\cap T_{P}(R_{P})=T_{P}(\Omega_{P})\cap T_{P}(R_{P}). 
\end{equation*}
It follows that 
\begin{equation}\label{local representation of Omega implies local rep of its boundary}
\partial{\Omega}\cap T_{P}(R_{P})=T_{P}(\partial\Omega_{P})\cap T_{P}(R_{P})
\end{equation}
for all $P\in\partial\Omega$, and a local parametrisation of $\partial\Omega$ around $P$ is given by
\begin{equation}\label{local parametrisation of boundary Omega around P}
    \Phi_{P}:\Delta(0,r_{P})\to \partial{\Omega}\cap T_{P}(R_{P}): x \mapsto T_{P}\left((\varphi_{P}(x),x)\right).
\end{equation}
Its inverse is given by 
\begin{equation*}
    \Phi_{P}^{-1}: \partial\Omega\cap T_{P}(R_{P})\to\Delta(0,r_P) : X\to \Pi(T_{P}^{-1}(X)),
\end{equation*}
where $\Pi:\R^{1+n}\to\Rn:(t,x)\mapsto x$.

More generally, we say that such a domain $\Omega$ is of class $C^{k,1}$ for an integer $k\geq 0$ if the functions $\varphi_{P}$ can additionally be chosen to be of class $C^{k,1}(\Rn)$, i.e. with Lipschitz continuous partial derivatives of order $k$, for all $P\in\partial\Omega$.
\begin{rem}\label{remark on distances on the Lipschitz boundary in cylinders}
    Note that if $Q\in \partial\Omega \cap T_{P}(R_{P})=T_{P}(\partial\Omega_{P})\cap T_{P}(R_P)$, then there exists $x\in\Delta(0,r_P)$ such that $\tilde{Q}=(\varphi_{P}(x) , x)\in\partial\Omega_{P}\cap R_{P}$ and $Q=T_{P}(\tilde{Q})$. Since $\varphi_{P}(0)=0$ and $T_{P}$ is an isometry, it follows that 
    \begin{align*}
        \abs{Q-P}&=\smallabs{T_{P}(\tilde{Q}) - T_{P}(0)}=\smallabs{\tilde{Q}}= \left(\abs{\varphi_{P}(x)}^{2} + \abs{x}^{2}\right)^{1/2}\leq (M^{2}+1)^{1/2}\abs{x}<(1+M)r_{P}.
    \end{align*}
\end{rem}

The boundary $\partial\Omega$ is equipped with the subspace (metric) topology and the surface measure $\sigma=\mathcal{H}^{n}|_{\partial\Omega}$, where $\mathcal{H}^{n}$ denotes the $n$-dimensional Hausdorff measure on $\R^{1+n}$.

Let us henceforth assume that $\Omega\subset\R^{1+n}$ is a bounded (at least) Lipschitz domain. For any Borel subset $A\subseteq \partial\Omega\cap T_{P}(R_P)$, it follows from the area formula (see, e.g., \cite[Sections 3.3.2 and 3.3.4 (c)]{Evans_Gariepy_book_1992}) and the bijectivity of $\Phi_{P}$ that 
\begin{equation}\label{local representation of surface measure from area formula}
    \sigma(A)=\int_{\Delta(0,r_P)}\ind{A}\left(\Phi_{P}(x)\right)\left(1+\abs{\nabla \varphi_{P}(x)}^{2}\right)^{\frac{1}{2}}\dd x.
\end{equation}
Consequently, for any non-negative Borel measurable function $f:\partial\Omega\to [0,\infty]$ with support in $\partial\Omega\cap T_{P}(R_P)$, it follows from classical measure-theoretic arguments that 
\begin{equation}\label{formula for local integral wrt surface measure from area formula}
    \int_{\partial\Omega}f\dd \sigma = \int_{\Delta(0,r_P)}f\left(\Phi_{P}(x)\right)\left(1+\abs{\nabla \varphi_{P}(x)}^{2}\right)^{\frac{1}{2}}\dd x.
\end{equation}
In particular, this shows that $\sigma$ is $n$-Ahlfors--David regular, i.e. that 
\begin{equation}\label{surface measure sigma is ahlfors regular}
\sigma(\Delta_{\partial\Omega}(Q,r))\eqsim r^{n} \quad \textup{ for all } r\in(0,\diam{\Omega}].
\end{equation}
\subsection{The \texorpdfstring{$L$}{L}-harmonic measure in Lipschitz domains}\label{subsection on L harmonic measure in Lipschitz domains}
Recall the definition of the elliptic operator $L\!=\!-\dvg{(A\nabla)}$ on $\Omega$ introduced in Section~\ref{section on the dirichlet problem for elliptic equations}. For all ${f\in C^{0}(\partial\Omega)}$ there is a unique $u_{f}\in\mathrm{W}^{1,2}_{\loc}(\Omega)\cap C^{0}(\clos{\Omega})$ such that $Lu_{f}=0$ and $u_{f}|_{\partial\Omega}=f$ (see \cite[Chapter 1, Section 2]{Kenig_cbms_1994}). By the maximum principle and the Riesz representation theorem there exists a family $\set{\omega^{X}_{L}}_{X\in\Omega}$ of regular probability measures on $\partial\Omega$ such that 
\begin{equation}\label{representation of classical solutions by harmonic measure bounded Lipschitz dom again}
    u_{f}(X)=\int_{\partial\Omega}f\dd \omega^{X}_{L}
\end{equation}
for all $X\in\Omega$. 
In accordance with \cite{Kenig_cbms_1994}, we say that $(\mathcal{D})_{p,\Omega}^{L}$ \emph{holds} if for all $f\in C^{0}(\partial\Omega)$, the classical solution $u_f$ in \eqref{representation of classical solutions by harmonic measure bounded Lipschitz dom again} satisfies the estimate
    \begin{equation*}
        \norm{N_{*}u_{f}}_{\El{p}(\partial\Omega)}\lesssim \norm{f}_{\El{p}(\partial\Omega)}.
    \end{equation*}
As stated in the introduction, this implies the well-posedness of $(\mathcal{D})^{L}_{p,\Omega}$ (see \cite[Theorem 1.7.7]{Kenig_cbms_1994}).

We shall now fix a point $X_0\in\Omega$, and refer to $\omega_{L}=\omega_{L}^{X_0}$ as the \emph{$L$-harmonic measure}.
The following theorem is well-known (see \cite[Theorem 1.7.3]{Kenig_cbms_1994}). It is also well-known that the symmetry assumption on the coefficients made therein is not required (see, e.g., \cite{KKPT_nonsymmetric_2000}).
\begin{thm}\label{thm equivalent properties for solvability in terms of  L harmonic measure}
    The following are equivalent: 
    \begin{enumerate}[label=\emph{(\roman*)}]
    \item $\omega_{L}\in A_{\infty}(\sigma)$;
    \item There exists $p_0\in(1,\infty)$ such that $\left(\mathcal{D}\right)^{L}_{p_0}$ holds;
    \item There exists $q_0\in(1,\infty)$ such that $\omega_{L}\in B^{q_0}(\dd \sigma)$, i.e. $\omega_{L}$ is absolutely continuous with respect to $\sigma$, and the Radon--Nikodym derivative $k=\frac{\dd \omega_{L}}{\dd \sigma}\in\El{q_0}(\dd \sigma)$ satisfies the reverse Hölder inequality
    \begin{equation*}
        \left(\dashint_{\Delta}k^{q_0}\dd \sigma\right)^{1/q_0}\lesssim \dashint_{\Delta} k \dd \sigma,
    \end{equation*}
    for all surface balls $\Delta\subseteq \partial\Omega$. Moreover, it holds that $\frac{1}{q_0}+\frac{1}{p_0}=1$ in \emph{(ii)} and \emph{(iii)}.
    \end{enumerate}
\end{thm}

\subsection{Sobolev spaces on the boundary of a Lipschitz domain}\label{subsection on sobolev spaces on boundary of Lipschitz domains}
We recall the following two elementary lemmas from the theory of Sobolev spaces (see \cite[Chapter 1]{Grisvard_book_1985} or \cite[Chapter 3]{McLean_book_2000}).
If $E\subseteq\Rn$ and $k$ is a non-negative integer, we denote by $C^{k,1}(E)$ the space of all functions on $E$ which are restrictions to $E$ of functions of class $C^{k,1}(\Rn)$.
\begin{lem}\label{biLipschitz change variable Sobolev spaces}
    Let $1\leq p<\infty$. Let $U_1$ and $U_2$ be two bounded open subsets of $\Rn$, and let $\psi:\clos{U_1}\to \clos{U_2}$ be a bi-Lipschitz bijection. Let $k\geq 0$ be an integer, and let $0\leq s\leq k+1$. If $\psi$ and $\psi^{-1}$ are both of class $C^{k,1}$, then $f\in\mathrm{W}^{s,p}(U_2)$ if and only if $f\circ \psi \in\mathrm{W}^{s,p}(U_1)$, and in this case $\norm{f\circ \psi}_{\mathrm{W}^{s,p}(U_1)}\eqsim \norm{f}_{\mathrm{W}^{s,p}(U_2)}$.
    
\end{lem}
\begin{lem}\label{multiplication by Lipschitz in Sobolev}
    Let $1\leq p <\infty$. Let $U$ be an open subset of $\Rn$, $k\geq 0$ be an integer, and let $\varphi\in C_{c}^{k,1}(\clos{U};\R)$. If $0\leq s\leq k+1$, then $\varphi f\in\mathrm{W}^{s,p}(U)$ for all $f\in\mathrm{W}^{s,p}(U)$, and ${\norm{\varphi f}_{\mathrm{W}^{s,p}(U)}\lesssim \norm{f}_{\mathrm{W}^{s,p}(U)}}$.
\end{lem}

Let $k\geq 0$ be an integer, $p\in[1,\infty)$ and let $0\leq s\leq k+1$. Let $\Omega$ be a domain of class $C^{k,1}$.
Following \cite[Definition 1.3.3.2]{Grisvard_book_1985}, the Sobolev space $\mathrm{W}^{s,p}(\partial\Omega)$ is defined as the space of functions $f\in\El{p}(\partial\Omega,\dd\sigma)$ such that $f\circ \Phi_{P}\in\mathrm{W}^{s,p}(\Delta(0,r_{P}))$ for all $P\in\partial\Omega$ and corresponding local parametrisations $\Phi_{P}$ of $\partial\Omega$ defined in \eqref{local parametrisation of boundary Omega around P}. Lemma \ref{biLipschitz change variable Sobolev spaces} shows that this space is well-defined.
To equip $\mathrm{W}^{s,p}(\partial\Omega)$ with a norm we use the compactness of $\partial\Omega$ to find an integer $N\geq 1$ and $\set{P_{1},\ldots,P_{N}}\subset\partial\Omega$ such that $\partial\Omega\subset \bigcup_{i=1}^{N}T_{P_{i}}(R_{P_{i}})$.  We then let
\begin{equation*}
    \norm{f}_{\mathrm{W}^{s,p}(\partial\Omega)}=\left(\sum_{i=1}^{N}\norm{f\circ \Phi_{P_{i}}}_{\mathrm{W}^{s,p}(\Delta(0,r_{P_{i}}))}^{p}\right)^{1/p}. 
\end{equation*}
It follows again from Lemma \ref{biLipschitz change variable Sobolev spaces} that any other such finite covering of $\partial\Omega$ will yield an equivalent norm on $\mathrm{W}^{s,p}(\partial\Omega)$. Moreover, when $s=0$, the formula \eqref{local representation of surface measure from area formula} shows that this norm is equivalent to the usual norm on $\El{p}(\partial\Omega,\dd \sigma)$.

In the same way, if $\Omega$ is a domain of class $C^{k,1}$ for some integer $k\geq 0$, and $0\leq l\leq k$ is an integer, we define the function space $C^{l,1}(\partial\Omega)$ as the space of continuous functions $f:\partial\Omega\to\R$ such that $f\circ \Phi_{P}\in C^{l,1}(\Delta(0,r_{P}))$ for all $P\in\partial\Omega$. The following elementary lemma, whose proof is left to the reader, justifies that this is a well-defined function space, and shows that it embeds into the Sobolev spaces $\mathrm{W}^{s,p}(\partial\Omega)$ for suitable $s\geq 0$.
\begin{lem}\label{elementary properties of C^{k,l} functions and sobolev}
    Let $U_1$ and $U_2$ be two bounded open subsets of $\Rn$, and let $\psi:\clos{U_1}\to \clos{U_2}$ be a bi-Lipschitz bijection. Let $0\leq l\leq k$ be integers. If $\psi$ and $\psi^{-1}$ are both of class $C^{k,1}$, then 
    \begin{enumerate}[label=\emph{(\roman*)}]
    \item If $f,g\in C^{l,1}(U_2)$, then $fg\in C^{l,1}(U_2)$;
    \item $f\in C^{l,1}(U_2)$ if and only if $f\circ \psi \in C^{l,1}(U_1)$;
    \item If $f\in C^{l,1}(U_2)$, $1\leq p<\infty$ and $0\leq s \leq l+1$, then $f\in\mathrm{W}^{s,p}(U_2)$.
    \end{enumerate}
\end{lem}

We note that our definition of $C^{0,1}(\partial\Omega)$ coincides with the classical space of Lipschitz functions on $\partial\Omega$. The reader can refer to \cite[Section 6.2]{Gilbarg_Trudinger_2nd_order_PDE} for more details and some properties of these spaces. 
We shall need the following density result.
\begin{lem}\label{density lemma in boundary sobolev spaces}
    Let $k\geq 0$ be an integer, and let $\Omega\subset\R^{1+n}$ be a bounded domain of class $C^{k,1}$. If $1\leq p<\infty$ and $0\leq s \leq k+1$, then $C^{k,1}(\partial\Omega)$ is a dense subspace of $\mathrm{W}^{s,p}(\partial\Omega)$.
\end{lem}
\begin{proof}
    By compactness of $\partial\Omega$, there are 
\begin{enumerate}[label=(\roman*)]
    \item An integer $N\geq 1$, numbers $r_{1},\ldots, r_N>0$ and $M>0$;
    \item Points $\set{P_{1},\ldots, P_{N}}\subset\partial\Omega$;
    \item Functions $\set{\varphi_{j}:\Rn\to \R}_{j=1}^{N}$ of class $C^{k,1}$, such that $\varphi_{j}(0)=0$, $\abs{\varphi_{j}(x)-\varphi_{j}(y)}\leq M\abs{x-y}$ for all $x,y\in\Rn$ and all $j\in\set{1,\ldots,N}$;
    \item Isometries $\set{T_{j}:\R^{1+n}\to\R^{1+n}}_{j=1}^{N}$ such that $T_{j}(0)=P_{j}$ for all $j\in\set{1,\ldots,N}$;
    \item Special Lipschitz domains $\Omega_{j}=\set{(t,x)\in\R^{1+n} : t>\varphi_{j}(x)}$;
    \item Cylinders $R_{j}=\set{(t,x)\in\R^{1+n} : \abs{x}<r_j \textup{ and }\abs{t}<(1+M)r_j}$,
    \end{enumerate}
    such that 
    \begin{equation*}
    \partial\Omega\subseteq \bigcup_{j=1}^{N} T_{j}(R_{j}),
    \end{equation*}
    and 
    \begin{equation*}
        \Omega\cap T_{j}(R_{j})= T_{j}(\Omega_{j})\cap T_{j}(R_{j})
    \end{equation*}
    for all $j\in\set{1,\ldots, N}$.
    Let 
    \begin{equation*}
    \Phi_{j}: \Delta(0, r_{j})\to \partial\Omega \cap T_{j}(R_j):x \mapsto T_{j}((\varphi_{j}(x) , x))
    \end{equation*}
    denote the corresponding local parametrisations of $\partial\Omega$ for $j\in\set{1,\ldots, N}$.
    
    Let $(\eta_{j})_{1\leq j\leq N}$ be a smooth partition of unity for this covering, i.e. $\eta_{j}\in C^{\infty}_{c}(\R^{1+n}, [0,1])$ with $\supp{\eta_{j}}\subseteq T_{j}(R_{j})$ for all $1\leq j\leq N$ and $\sum_{j=1}^{N}\eta_{j}=1$ on a neighbourhood of $\partial\Omega$. 
    Let $f\in\mathrm{W}^{s,p}(\partial\Omega)$ and $\eps>0$. By definition, it holds that $f\circ \Phi_{j}\in\mathrm{W}^{s,p}(\Delta(0,r_j))$ for all $j\in\set{1,\ldots, N}$, hence there exist $\tilde{g}_{j}\in C^{\infty}(\clos{\Delta(0,r_j)})$ (i.e. the restriction to $\clos{\Delta(0,r_j)}$ of a function in $C^{\infty}(\Rn)$) such that $\norm{f\circ \Phi_{j} - \tilde{g}_{j}}_{\mathrm{W}^{s,p}(\Delta(0,r_j))}<\eps$ for all $j\in\set{1,\ldots,N}$ (this is because $\Delta(0,r_j)$ is a domain with a continuous boundary; see, e.g., \cite[Theorem~1.4.2.1]{Grisvard_book_1985}).

    For each $j\in\set{1,\ldots, N}$, let $\chi_{j}=\eta_{j}|_{\partial\Omega}:\partial\Omega \to\R$ denote the restriction of $\eta_{j}$ to $\partial\Omega$, and note that $\supp{\chi_{j}}\subseteq \partial\Omega\cap T_{j}(R_{j})$ and $\sum_{j=1}^{N}\chi_{j}=1$. We also let $g_{j}=\tilde{g}_{j}\circ \Phi_{j}^{-1}:\partial\Omega\cap T_{j}(R_j) \to \R$ for all $j\in\set{1,\ldots, N}$ and define $g=\sum_{j=1}^{N}\chi_{j}g_{j}:\partial\Omega \to\R$, where it is understood that 
    \begin{equation*}
        (\chi_{j}g_{j})(P)=\left\{ \begin{array}{ll}
            \chi_{j}(P) g_{j}(P), & \textup{if }P\in\partial\Omega\cap T_{j}(R_{j}),  \\
              0, & \textup{if }P\in \partial\Omega\setminus (\partial\Omega \cap T_{j}(R_{j})).
        \end{array}\right.
    \end{equation*}
    
    We claim that $\chi_{j}g_{j}\in C^{k,1}(\partial\Omega)$ for all $j\in\set{1,\ldots, N}$, and hence that $g\in C^{k,1}(\partial\Omega)$. Let $i,j\in\set{1,\ldots, N}$ be fixed. The function $(\chi_{j}g_{j})\circ \Phi_{i}$ is only non-zero on the open set $\Phi_{i}^{-1}(\partial\Omega \cap T_{j}(R_j))\subseteq \Delta(0,r_i)$, and on this set it holds that 
    \begin{align*}
        (\chi_{j}g_{j})\circ \Phi_{i} = (\eta_{j}\circ \Phi_{i})(\tilde{g}_{j}\circ (\Phi_{j}^{-1}\circ \Phi_{i})).
    \end{align*}
By assumption, the map 
\begin{equation}\label{change of coordinate map between i and j}
    \Phi_{j}^{-1}\circ \Phi_{i}|_{\Delta(0,r_i)\cap \Phi_{i}^{-1}(\partial\Omega\cap T_{j}(R_j))} : \Delta(0,r_i)\cap \Phi_{i}^{-1}(\partial\Omega\cap T_{j}(R_j))\to \Delta(0,r_j)\cap \Phi_{j}^{-1}(\partial\Omega\cap T_{i}(R_i))
\end{equation}
and its inverse are of class $C^{k,1}$. Since $\tilde{g}_{j}\in C^{\infty}(\clos{\Delta(0,r_j)})$ and $\eta_{j}\circ \Phi_{i}\in C^{k,1}(\Delta(0,r_i))$, it follows from Lemma~\ref{elementary properties of C^{k,l} functions and sobolev} that $(\chi_{j}g_{j})\circ \Phi_{i}\in C^{k,1}(\Delta(0,r_i))$, and this proves that $g\in C^{k,1}(\partial\Omega)\subseteq \mathrm{W}^{s,p}(\partial\Omega)$. 

Since $f-g=\sum_{j=1}^{N}(\chi_{j}f -\chi_{j}g_{j})$, it holds that 
\begin{align*}
    \norm{f-g}_{\mathrm{W}^{s,p}(\partial\Omega)}&\leq \sum_{j=1}^{N}\norm{f\chi_{j}-\chi_{j}g_{j}}_{\mathrm{W}^{s,p}(\partial\Omega)}\\
    &\eqsim \sum_{j=1}^{N}\sum_{i=1}^{N}\norm{(f\chi_{j})\circ \Phi_i - (\chi_{j}g_j)\circ \Phi_{i}}_{\mathrm{W}^{s,p}(\Delta(0,r_i))}\\
    &= \sum_{j=1}^{N}\sum_{i=1}^{N}\norm{(f\chi_{j})\! \circ\! \Phi_{j}\!\circ\! \Phi_{j}^{-1}\!\circ\! \Phi_i \! - (\chi_{j}g_j)\!\circ\! \Phi_{j}\!\circ\! \Phi_{j}^{-1}\!\circ\!\Phi_{i}}_{\mathrm{W}^{s,p}(\Delta(0,r_i)\cap \Phi_{i}^{-1}(\partial\Omega\cap T_{j}(R_j)))}\\
    &\lesssim \sum_{j=1}^{N}\sum_{i=1}^{N}\norm{(f\chi_{j})\circ \Phi_{j}  - (\chi_{j}g_j)\circ \Phi_{j}}_{\mathrm{W}^{s,p}(\Delta(0,r_j)\cap \Phi_{j}^{-1}(\partial\Omega\cap T_{i}(R_i)))}\\
    &\lesssim \sum_{j=1}^{N}\norm{(f\chi_{j})\circ \Phi_{j}  - (\chi_{j}g_j)\circ \Phi_{j}}_{\mathrm{W}^{s,p}(\Delta(0,r_j))}\\
    &=\sum_{j=1}^{N}\norm{(\chi_{j}\circ \Phi_{j})\cdot (f-g_{j})\circ \Phi_{j}}_{\mathrm{W}^{s,p}(\Delta(0,r_j))}\\
    &\lesssim \sum_{j=1}^{N}\norm{f\circ \Phi_{j} - \tilde{g}_{j}}_{\mathrm{W}^{s,p}(\Delta(0,r_j))}\lesssim \eps,
\end{align*}
where in the fourth line we have used Lemma~\ref{biLipschitz change variable Sobolev spaces} and the fact that the map in \eqref{change of coordinate map between i and j} and its inverse are of class $C^{k,1}$, and in the last line we have used Lemma~\ref{multiplication by Lipschitz in Sobolev} and the fact that $\chi_{j}\circ \Phi_{j}\in C^{k,1}(\clos{\Delta(0,r_j)})$. Since $\eps>0$ was arbitrary, this concludes the proof.
\end{proof}
\subsection{Corkscrew points and geometry in Lipschitz domains}\label{section on corkscrew points}
Since $\Omega$ is a bounded Lipschitz domain, it satisfies a corkscrew condition. This means that there is a (small) constant $c_{\Omega}>0$ such that for all $t\in(0,c_{\Omega})$ and $Q\in\partial\Omega$, there exists (at least) one corresponding \emph{corkscrew point} $A_{t}(Q)\in\Omega$ with the property that 
\begin{equation*}
    \dist(A_{t}(Q), \partial\Omega)\eqsim \abs{A_{t}(Q)-Q}\eqsim t.
\end{equation*}

We first give an explicit description of corkscrew points in special Lipschitz domains. Let $\varphi:\Rn\to \R$ be Lipschitz with constant $M>0$, and let $\Omega_{\varphi}=\set{(t,x)\in\R^{1+n}: t>\varphi(x)}$.
\begin{lem}\label{explicit corkscrew points in special Lipschitz domains}
    Let $x_0\in\Rn$ and $t>0$. If $Q_{0}=(\varphi(x_0),x_0)\in\partial\Omega_{\varphi}$, then $A_{t}(Q_0)=(\varphi(x_0)+t,x_0)\in\Omega_{\varphi}$ is a corkscrew point for $Q_0$.
\end{lem}
\begin{proof}
    It clearly suffices to prove that $\dist(A_{t}(Q_0),\partial\Omega_{\varphi})\gtrsim t$. We first treat the case when the Lipschitz constant $M$ of $\varphi$ satisfies $M\leq 1$. Then, for all $x\in\Rn$ it holds that
    \begin{align*}
        \abs{(\varphi(x_0)+t,x_0)-(\varphi(x),x)}&=\left(\abs{\varphi(x_0)-\varphi(x)+t}^{2} + \abs{x-x_0}^{2}\right)^{1/2}\\
        &\geq \frac{1}{2}\left(\abs{\varphi(x_0)-\varphi(x)+t} + \abs{x-x_0}\right)\\
        &\geq \frac{1}{2}\left(t-\abs{\varphi(x_0)-\varphi(x)} + \abs{x-x_0}\right)\\
        &\geq \frac{1}{2}\left(t+ \left(M^{-1}-1\right)\abs{\varphi(x_0)-\varphi(x)}\right)\geq \frac{t}{2},
    \end{align*}
    hence $\dist(A_{t}(Q_0), \partial\Omega_{\varphi})\geq \frac{t}{2}$. We now assume that $M>1$. First, if $\abs{x-x_0}<\frac{t}{2(M-1)}$, then
    \begin{align*}
        \abs{(\varphi(x_0)+t,x_0)-(\varphi(x),x)}&
        \geq \frac{1}{2}\left(t-\abs{\varphi(x_0)-\varphi(x)} + \abs{x-x_0}\right)\\
        &\geq \frac{1}{2}\left(t+(1-M)\abs{x-x_0}\right)\geq \frac{t}{4}.
    \end{align*}
    On the other hand, if $\abs{x-x_0}\geq \frac{t}{2(M-1)}$, then we simply have 
    \begin{equation*}
    \abs{(\varphi(x_0)+t,x_0)-(\varphi(x),x)}\geq \abs{x-x_0}\geq \frac{t}{2(M-1)},  
    \end{equation*}
    hence $\dist(A_{t}(Q_0), \partial\Omega_{\varphi})\geq t\min\set{\frac{1}{4} , \frac{1}{2(M-1)}}$. This concludes the proof.
\end{proof}
In the rest of this section, we assume that $\Omega$ is a bounded Lipschitz domain with Lipschitz constant $M>0$.
The following lemma provides explicit corkscrew points in coordinate cylinders. 
\begin{lem}\label{explicit representation of corkscrew points in local cylinder coordinates}
    Let $P\in\partial\Omega$ and $r>0$ be such that
    \begin{equation*}
        \Omega\cap T_{P}(R(C_M r))=T_{P}(\Omega_P)\cap T_{P}(R(C_M r))
    \end{equation*}
    for $C_{M}=10(1+M)$.
    Let $Q\in T_{P}(R(r))\cap\partial\Omega=T_{P}(R(r))\cap T_{P}(\partial\Omega_{P})$, and let $x\in \Delta(0,r)$ such that $Q=\Phi_{P}(x)= T_{P}((\varphi_{P}(x) , x))$ in \eqref{local parametrisation of boundary Omega around P}. 
    If $0<t\leq 3(1+M)r$, then $A_{t}(Q)=T_{P}((\varphi_{P}(x)+t,x))$ is a corkscrew point for $Q$. 
\end{lem}
\begin{proof}
Observe that $\abs{Q-P}<(1+M)r$ by Remark~\ref{remark on distances on the Lipschitz boundary in cylinders}.
Similarly, if $Q_{0}\in \partial\Omega\setminus T_{P}(R(C_M r)) $, then $\abs{P-Q_0}\geq C_M r$, and therefore $\abs{Q-Q_0}\geq 9(1+M)r$. Let $0<t\leq 3(1+M)r$. Note that $A_{t}(Q)\in T_{P}(R(C_M r))$ and therefore $A_{t}(Q)\in\Omega$. Observe that
\begin{align*}
    \abs{Q_{0}-A_{t}(Q)}&\geq \abs{Q_{0}-Q}-\abs{Q-A_{t}(Q)}\geq 9(1+M)r-t \geq 2t.
\end{align*}
Since $Q\in T_{P}(R(r))\cap\partial\Omega$ and $\abs{Q-A_{t}(Q)}=t<2t$, we have that
\begin{align*}
    \dist(A_{t}(Q),\partial\Omega)&=\inf_{Q_{0}\in\partial\Omega}\abs{A_{t}(Q)-Q_0}=\inf_{Q_{0}\in\partial\Omega\cap T_{P}(R(C_M r))}\abs{A_{t}(Q)-Q_0}\\
    &=\inf_{Q_{0}\in T_{P}(\partial\Omega_{P})\cap T_{P}(R(C_M r))}\abs{A_{t}(Q)-Q_0}=\inf_{P_{0}\in \partial\Omega_{P}\cap R(C_M r)}\abs{T_{P}^{-1}(A_{t}(Q))-P_0}\\
    &\geq \dist(T_{P}^{-1}(A_{t}(Q)), \partial\Omega_{P})\gtrsim t,
\end{align*}
where we have used Lemma~\ref{explicit corkscrew points in special Lipschitz domains} to obtain the last equivalence.
\end{proof}
The following lemma will allow us to localise the study of the approach regions $\Gamma^{\beta,c}_{\Omega}(Q_0)$. Note that equality actually holds in \eqref{inclusion relation tangential approach regions in Lipschitz domains} below, but we shall only need the stated inclusion. 
\begin{lem}\label{lemma tangential approach regions from bounded to special Lipschitz}
    Let $P\in\partial\Omega$ and $r>0$ be such that 
    \begin{equation*}
        \Omega\cap T_{P}(R(C_M r))=T_{P}(\Omega_P)\cap T_{P}(R(C_M r))
    \end{equation*}
    for $C_{M}=10(1+M)$. For all $Q_{0}\in\partial\Omega$, $\beta\in(0,1]$ and $c>0$ it holds that 
    \begin{equation}\label{inclusion relation tangential approach regions in Lipschitz domains}
        T_{P}^{-1}(\Gamma_{\Omega}^{\beta, c}(Q_0)\cap T_{P}(R(r)))\subseteq \Gamma_{\Omega_{P}}^{\beta,c}(T_{P}^{-1}(Q_0))\cap R(r).
    \end{equation}
\end{lem}
\begin{proof}
Let $X\in  \Gamma_{\Omega}^{\beta, c}(Q_0)\cap T_{P}(R(r))$. Since $T_{p}$ is an isometry, it holds that
\begin{align*}
    \dist(X,\partial\Omega)&=\dist(T_{P}^{-1}(X), T_{P}^{-1}(\partial\Omega))=\inf_{Y\in T_{P}^{-1}(\partial\Omega)}\abs{T_{P}^{-1}(X)-Y}\\
    &\leq \inf_{Y\in T_{P}^{-1}(\partial\Omega \cap T_{P}(R(C_{M}r)))}\abs{T_{P}^{-1}(X)-Y}=\inf_{Y\in \partial\Omega_{P} \cap R(C_{M}r)}\abs{T_{P}^{-1}(X)-Y},
\end{align*}
where in the last equality we have used the assumption to deduce the property corresponding to \eqref{local representation of Omega implies local rep of its boundary}.
By assumption it holds that $ T_{P}^{-1}(X)\in R(r)$, and therefore $\abs{T_{P}^{-1}(X)}< (2+M)r$.
If $Y\in\partial\Omega_{P}\setminus R(C_M r)$, then $\abs{Y}\geq C_{M}r$, and therefore $\abs{T_{P}^{-1}(X)-Y}\geq \abs{Y}-\abs{T_{P}^{-1}(X)}> 8(1+M)r$. Since $0\in\partial\Omega_{P}\cap R(C_{M}r)$ and $\abs{T_{P}^{-1}(X)-0}<(2+M)r< 8(1+M)r$, it follows that 
\begin{equation*}
\inf_{Y\in \partial\Omega_{P} \cap R(C_{M}r)}\abs{T_{P}^{-1}(X)-Y}=\inf_{Y\in\partial\Omega_{P}}\abs{T_{P}^{-1}(X)-Y}.
\end{equation*}
This shows that $\dist(X,\partial\Omega)\leq \dist(T_{P}^{-1}(X),\partial\Omega_{P})$. From this and the fact that $\beta\in(0,1]$ it easily follows that $T_{P}^{-1}(X)\in\Gamma_{\Omega_{P}}^{\beta,c}(T_{p}^{-1}(Q_0))$.
\end{proof}
\subsection{Pointwise estimates on solutions}\label{section on pointwise estimates on solutions}
In this section we fix $p_0\in(1,\infty)$ such that $\left(\mathcal{D}\right)^{L}_{p_0}$ holds. It follows from Theorem~\ref{thm equivalent properties for solvability in terms of  L harmonic measure} that $\omega_{L}\in B^{q_0}(\dd \sigma)$, and we let $k=\frac{\dd \omega_{L}}{\dd \sigma}\in\El{q_0}(\dd \sigma)$ denote the corresponding Radon--Nikodym derivative.

We need to introduce some notation for the proofs of the following lemmas. For any $X\in\Omega$, the measures $\omega^{X}_{L}$ and $\omega_{L}$ are mutually absolutely continuous (see, e.g., \cite[Lemma 1.2.7]{Kenig_cbms_1994}). In accordance with \cite[Definition 1.3.10]{Kenig_cbms_1994}, we denote by $K(X,Q)=\frac{\dd \omega^{X}_{L}}{\dd \omega_{L}}(Q)$ the corresponding Radon--Nikodym derivative.
\begin{lem}\label{Lemma domination series dyadic averages at corkscrew points}
 Let $f\in\El{p_0}(\partial\Omega)$ and $u(X)=\int_{\partial\Omega}f\dd \omega^{X}_{L}$ for $X\in\Omega$. There is $\alpha_{L}>0$ depending only on $n$ and $\lambda$ such that for all $Q_{0}\in\partial\Omega$, $t\in(0,c_{\Omega})$ and corresponding corkscrew point $A_{t}(Q_0)\in\Omega$ it holds that 
\begin{equation*}
    \abs{u(A_{t}(Q_{0}))}\lesssim \sum_{j=0}^{\infty} 2^{-\alpha_{L} j}\left(\frac{1}{\sigma(\Delta_{j})}\int_{D_{j}}\abs{f}^{p_0}\dd \sigma\right)^{1/p_0},
\end{equation*}
where $\Delta_{j}=\Delta_{\partial\Omega}(Q_{0},2^{j+1}t)$ for all $j\geq 0$, $D_{j}=\Delta_{j}\setminus \Delta_{j-1}$ for $j\geq 1$, and $D_{0}=\Delta_{0}$.
\end{lem}
\begin{proof}
    Let $j\geq 1$. It follows from \cite[Lemma 1.3.12]{Kenig_cbms_1994} that
    \begin{equation*}
        \textup{ess sup}_{P\in D_{j}} K(A_{t}(Q_0),P)\lesssim 2^{-\alpha_{L} j}\omega(\Delta_{j})^{-1} 
    \end{equation*}
    where $\alpha_{L}>0$ depends only on $n$ and $\lambda$. By a similar and simpler reasoning, this also holds for $j=0$.
    The essential supremum in this estimate is understood with respect to the $L$-harmonic measure $\omega_{L}$. Since $\omega_{L}$ is absolutely continuous with respect to the surface measure $\sigma$, the essential supremum can be taken equivalently with respect to $\sigma$. As a consequence, if we let $X=A_{t}(Q_0)$, then we can estimate
    \begin{align*}
        \abs{u(X)}&\leq \int_{\partial\Omega}\abs{f}\dd\omega^{X}_{L}=\int_{\partial\Omega}K(X,P)k(P)\abs{f(P)}\dd\sigma(P)\\
        &= \sum_{j=0}^{\infty}\int_{D_{j}}K(X,P)k(P)\abs{f(P)}\dd\sigma(P)\\
        &\lesssim \sum_{j=0}^{\infty} 2^{-\alpha_{L} j}\omega_{L}(\Delta_{j})^{-1}\left(\int_{D_j}\abs{f}^{p_0}\dd\sigma\right)^{1/p_0}\left(\int_{\Delta_j} k^{q_0} \dd\sigma\right)^{1/q_0}\\
        &=\sum_{j=0}^{\infty} 2^{-\alpha_{L} j}\sigma(\Delta_{j})\omega_{L}(\Delta_{j})^{-1}\left(\frac{1}{\sigma(\Delta_{j})}\int_{D_j}\abs{f}^{p_0}\dd\sigma\right)^{1/p_0}\left(\dashint_{\Delta_j} k^{q_0} \dd\sigma\right)^{1/q_0}\\
        &\lesssim \sum_{j=0}^{\infty} 2^{-\alpha_{L} j}\omega_{L}(\Delta_{j})^{-1}\left(\frac{1}{\sigma(\Delta_{j})}\int_{D_j}\abs{f}^{p_0}\dd\sigma\right)^{1/p_0}\left(\int_{\Delta_j} k \dd\sigma\right)\\
        &=\sum_{j=0}^{\infty} 2^{-\alpha_{L} j}\left(\frac{1}{\sigma(\Delta_{j})}\int_{D_j}\abs{f}^{p_0}\dd\sigma\right)^{1/p_0}.\qedhere
    \end{align*}
\end{proof}
\begin{lem}\label{uniform boundedness of solutions near where boundary data vanishes}
    Let $Q_0\in\partial\Omega$ and $r_{0}>0$. If $f\in\El{p_0}(\partial\Omega)$ is such that $\supp{f}\subseteq \partial\Omega\setminus \Delta_{\partial\Omega}(Q_0,r_0)$ and $u(X)=\int_{\partial\Omega}f\dd \omega^{X}$, then $\abs{u(X)}\lesssim \norm{f}_{\El{p_0}(\partial\Omega)}$ for all $X\in\Omega\cap B(Q_0,r_{0}/2)$.
\end{lem}
\begin{proof}
    The argument of the proof of \cite[Lemma 1.3.13]{Kenig_cbms_1994} shows that there are constants $C>0$ (depending on $r_{0}$) and $\alpha>0$ (depending only on $n$ and $\lambda$) such that
    \begin{equation*}
        K(X,P)\leq C\left(\frac{\abs{X-Q_{0}}}{r_0}\right)^{\alpha}\leq C
    \end{equation*}
    for all $P\in\partial\Omega\setminus \Delta_{\partial\Omega}(Q_{0}, r_0)$ and $X\in\Omega\cap B(Q_{0},r_{0}/2)$. As a consequence, for these values of $X$, we can estimate
    \begin{align*}
        \abs{u(X)}&\leq \int_{\partial\Omega\setminus \Delta_{\partial\Omega}(Q_0,r_0)} K(X,P)k(P)\abs{f(P)}\dd \sigma(P)\\
        &\lesssim \int_{\partial\Omega} k(P)\abs{f(P)}\dd \sigma(P)\leq \norm{k}_{\El{q_0}(\partial\Omega)}\norm{f}_{\El{p_0}(\partial\Omega)}.\qedhere
    \end{align*}
\end{proof}
\subsection{Tangential approach regions in Lipschitz domains}\label{section on tangential approach regions in lipschitz domains}
Let $\Omega\subset \R^{1+n}$ be a Lipschitz domain. For $Q_{0}\in\partial\Omega$, $\beta\in(0,1]$ and $c>0$, we recall that the tangential approach region $\Gamma^{\beta,c}_{\Omega}(Q_0)$ is given by
\begin{equation*}
    \Gamma^{\beta,c}_{\Omega}(Q_0)=\set{X\in\Omega : \abs{X-X_{0}}< \left\{ 
    \begin{array}{ll}
         (1+c)\dist(X,\partial\Omega)^{\beta} & \textup{if }\dist(X,\partial\Omega)\leq 1;  \\
         (1+c)\dist(X,\partial\Omega) & \textup{if }\dist(X,\partial\Omega)\geq 1.  
    \end{array}\right.\quad},
\end{equation*}
and the corresponding approach regions in the upper half-space are given for all $x_0\in\Rn$ by 
\begin{equation*}
     \Gamma^{\beta}_{c}(x_0)=\set{(t,x)\in\Hn : \abs{x-x_{0}}<ct^{\beta} \textup{ if }t\in(0,1], \textup{ and } \abs{x-x_0}<ct \textup{ if }t\geq 1}.
 \end{equation*}

Let $\varphi:\Rn\to\R$ be a Lipschitz function, and consider the associated special Lipschitz domain $\Omega_{\varphi}=\set{(t,x)\in\R^{1+n} : t>\varphi(x)}$. 
The domain $\Omega_{\varphi}$ can be transformed to the upper half-space $\Hn$ by the bi-Lipschitz mapping 
\begin{equation*}
    F_{\varphi}:\Omega_{\varphi}\to\Hn :(t,x)\mapsto (t-\varphi(x),x),
\end{equation*}
with inverse 
\begin{equation*}
    F_{\varphi}^{-1}:\Hn \to\Omega_{\varphi} :(t,x)\mapsto (t+\varphi(x),x).
\end{equation*}
\begin{lem}\label{lemma inclusion of domain dependent tangential approach regions into pullback of upper half space regions}
    Let $Q_0=(\varphi(x_0),x_0)\in\partial\Omega_{\varphi}$ for some $x_0\in\Rn$. For all $\beta\in(0,1]$ and $c>0$ it holds that $\Gamma_{\Omega_{\varphi}}^{\beta,c}(Q_0)\subseteq F_{\varphi}^{-1}\left(\Gamma^{\beta}_{1+c}(x_0)\right)$.
\end{lem}
\begin{proof}
Note that 
\begin{equation*}
    F_{\varphi}^{-1}\left(\Gamma^{\beta}_{1+c}(x_0)\right)=\set{(t+\varphi(x),x) : \abs{x-x_{0}}<\left\{
    \begin{array}{ll}
    (1+c)t^{\beta} & \textup{ if }t\in(0,1];\\
    (1+c)t & \textup{ if }t\geq 1
    \end{array}
    \right.}.
\end{equation*}
Let $X=(t,x)\in \Gamma_{\Omega_{\varphi}}^{\beta,c}(Q_0)$. Since $(t,x)=(t-\varphi(x)+\varphi(x), x)$, it suffices to show that $\abs{x-x_0}<(1+c)(t-\varphi(x))^{\beta}$ if $0<t-\varphi(x)\leq 1$, and that $\abs{x-x_0}<(1+c)(t-\varphi(x))$ if $t-\varphi(x)\geq 1$. If $\dist((t,x),\partial\Omega_{\varphi})\leq 1$, then by definition of $\Gamma_{\Omega_{\varphi}}^{\beta,c}(Q_0)$, we have 
\begin{align*}
    \abs{x-x_0}&\leq \abs{(t,x)-(\varphi(x_0),x_0)}<(1+c)\dist((t,x),\partial\Omega_{\varphi})^{\beta}\\
    &\leq (1+c)\abs{(t,x)-(\varphi(x),x)}^{\beta}=(1+c)(t-\varphi(x))^{\beta}.
\end{align*}
Similarly, if $\dist((t,x),\partial\Omega_{\varphi})\geq 1$, then $\abs{x-x_0}<(1+c)(t-\varphi(x))$.
Observe that since $\beta\in(0,1]$, if $t-\varphi(x)\geq 1$ then $(t-\varphi(x))^{\beta}\leq t-\varphi(x)$, and the converse inequality holds if $t-\varphi(x)\leq 1$. This concludes the proof.
\end{proof}
\subsection{First part of the proof of Theorem~\ref{second main theorem bounded Lipschitz domain}}\label{section first part of the proof of second main thm}

We now have all the ingredients needed for the proof of the tangential maximal function estimate~\eqref{strong Lp bound tangential max at the boundary of Lipschitz domain} and the corresponding $\sigma$-a.e. convergence~\eqref{convergence sigma a.e. at the boundary of Lipschitz domain} in  Theorem~\ref{second main theorem bounded Lipschitz domain}. 

Before starting with the proof, let us say a few words about the essential idea of the argument. The first step is to use linearity to reduce matters to the case when $f$ is supported in a small cylinder $R_{P}$ on which the boundary $\partial\Omega$ can be represented by a Lipschitz graph as in~\eqref{local representation of Omega implies local rep of its boundary}. Given a point $Q_{0}\in\partial\Omega$, we estimate the maximal function $N_{*,\beta}u_{f}(Q_0)$ by splitting the tangential approach region $\Gamma^{\beta}_{\Omega}(Q_0)$ into two parts - one that is bounded away from the boundary, and one that is close to the boundary and is denoted by $\Gamma^{\beta}_{\Omega, \loc}(Q_0)$. The part that is bounded away from the boundary is dominated in a simple way by the $\El{p}$ norm of $f$ using the estimate~\eqref{interior estimate for u(X) epsilon away from boundary} below. To estimate the part of the maximal function $N_{*,\beta}u_{f}(Q_0)$ corresponding to the localised region $\Gamma^{\beta}_{\Omega,\loc}(Q_0)$ we treat two cases. If $Q_{0}$ is far from the support of $f$, then the maximal function is estimated in a simple way using Lemma~\ref{uniform boundedness of solutions near where boundary data vanishes}. If $Q_{0}$ is close to the support of $f$, then the situation is essentially reduced to that of the region above a Lipschitz graph. The boundary can be “flattened" using the mappings $F_{\varphi}$ introduced in Section~\ref{section on tangential approach regions in lipschitz domains}, and Lemma~\ref{lemma inclusion of domain dependent tangential approach regions into pullback of upper half space regions} reduces matters to the situation of the upper half-space considered in Section~\ref{subsection proof of main thm upper half space}, and effectively to an application of Remark~\ref{remark main implication of main thm within proof}.
\begin{proof}[Proof of \eqref{strong Lp bound tangential max at the boundary of Lipschitz domain} and \eqref{convergence sigma a.e. at the boundary of Lipschitz domain} in Theorem~\ref{second main theorem bounded Lipschitz domain}]
Let $k\geq 0$ and let $\Omega$ be a bounded domain of class $C^{k,1}$. Let $0\leq s\leq k+1$ and let $p\in (1,\infty)$ be such that $\left(\mathcal{D}\right)^{L}_{p}$ holds. The self-improvement property of reverse Hölder inequalities (see, e.g., \cite[Theorem 2]{Iwaniec_Nolder_1985}) and Theorem~\ref{thm equivalent properties for solvability in terms of  L harmonic measure} show that there is $p_{0}\in (1,p)$ such that $\left(\mathcal{D}\right)^{L}_{p_0}$ holds. Suppose that $0< s\leq \frac{n}{p}$ and let $f\in\mathrm{W}^{s,p}(\partial\Omega)$. We let $\beta=1-\frac{s p}{n}\in(0,1)$ if $s<\frac{n}{p}$, and we let $\beta\in(0,1]$ be arbitrary if $s=\frac{n}{p}$.

Let $M>0$ be the Lipschitz constant of $\Omega$. For every $P\in\partial\Omega$, following the definition from Section~\ref{subsection with definition of Lipschitz domains}, there is $r_{P}>0$ such that
\begin{equation*}
    \Omega\cap T_{P}(R(r_P))=T_{P}(\Omega_{P})\cap T_{P}(R(r_{P})).
\end{equation*}
For any fixed constant $C_{M}\geq 1$, and letting $\tilde{r}_{P}=r_{P}/C_{M}$, this clearly implies that 
\begin{equation*}
    \Omega\cap T_{P}(R(\tilde{r}_P))=T_{P}(\Omega_{P})\cap T_{P}(R(\tilde{r}_{P})).
\end{equation*}
By extracting a finite covering from the open cover $\partial\Omega \subseteq \bigcup_{P\in\partial\Omega}T_{P}(R(\tilde{r}_{P}/2))$, we obtain
\begin{enumerate}[label=(\roman*)]
    \item An integer $N\geq 1$;
    \item Points $\set{P_{1},\ldots, P_{N}}\subset\partial\Omega$ and real numbers $r_{1},\ldots,r_N>0$;
    \item Functions $\set{\varphi_{i}:\Rn\to \R}_{i=1}^{N}$ of class $C^{k,1}$, such that $\varphi_{i}(0)=0$\\ and $\abs{\varphi_{i}(x)-\varphi_{i}(y)}\leq M\abs{x-y}$ for all $x,y\in\Rn$ and all $i\in\set{1,\ldots,N}$;
    \item Isometries $\set{T_{i}:\R^{1+n}\to\R^{1+n}}_{i=1}^{N}$ such that $T_{i}(0)=P_{i}$ for all $i\in\set{1,\ldots,N}$;
    \item Special Lipschitz domains $\Omega_{i}=\set{(t,x)\in\R^{1+n} : t>\varphi_{i}(x)}$,
    \end{enumerate}
    such that 
    \begin{equation}\label{smallest covering of boundary of Omega}
    \partial\Omega\subseteq \bigcup_{i=1}^{N} T_{i}(R(r_{i}/2)),
    \end{equation}
    and 
    \begin{equation}\label{local representation of Omega close to the boundary}
        \Omega\cap T_{i}(R(C_{M}r_{i}))= T_{i}(\Omega_{i})\cap T_{i}(R(C_{M}r_{i}))
    \end{equation}
    for all $i\in\set{1,\ldots, N}$, where $C_{M}=20(1+M)^{3}$ and we recall the notation for the cylinders
    \begin{equation*}
        R(s)=\set{(t,x)\in\R^{1+n} : \abs{x}<s \textup{ and }\abs{t}<(1+M)s},\quad \textup{ for all }s>0.
    \end{equation*}
    Let $\Phi_{i}: \Delta(0, C_{M}r_{i})\to \partial\Omega \cap T_{i}(R(C_M r_{i}))$ denote the corresponding local parametrisations of $\partial\Omega$ for $i\in\set{1,\ldots, N}$ (see \eqref{local parametrisation of boundary Omega around P}). Note that the functions $\Phi_{i}$ can be naturally extended to functions of class $C^{k,1}$ on $\Rn$ by (iii) above.

    To simplify the notation we shall only prove Theorem~\ref{second main theorem bounded Lipschitz domain} for the regions $\Gamma^{\beta,c}_{\Omega}(Q_0)$ corresponding to $c=1/2$, but it will be clear that our argument works for any $c>0$. In doing so we will also require (the proof of) Theorem~\ref{main theorem upper half space for Cpalpha} for the regions $\Gamma_{a}^{\beta}(x_0)$ with general aperture $a>0$. As mentioned in the introduction the argument of the proofs of Theorems~\ref{main theorem upper half space} and \ref{main theorem upper half space for Cpalpha} applies equally well to these regions. 
    
    Let $r=\min_{i\in\set{1,\ldots, N}} r_{i}>0$. For $Q_{0}\in\partial\Omega$, we introduce the local approach regions
    \begin{equation*}
        \Gamma_{\Omega,\loc}^{\beta}(Q_0)=\set{X\in\Omega : \abs{X-Q_{0}}<(1+\tfrac{1}{2})\dist(X,\partial\Omega)^{\beta} \textup{ and }\dist(X,\partial\Omega)<\left(\tfrac{r}{12}(1+M)\right)^{1/\beta}}.
    \end{equation*}
      
    Let $\set{\eta_{i}}_{i=1}^{N}$ be a smooth partition of unity with respect to the covering \eqref{smallest covering of boundary of Omega}, i.e. $\eta_{i}\in C^{\infty}_{c}(\R^{1+n};[0,1])$ are such that $\supp{\eta_{i}}\subseteq T_{i}(R(r_{i}/2))$ for all $i\in\set{1,\ldots, N}$ and $\sum_{i=1}^{N}\eta_{i}=1$ in a neighborhood of $\partial\Omega$.
    It follows from arguments similar to those used in the proof of Lemma~\ref{density lemma in boundary sobolev spaces} that $f_{i}:=f\eta_{i}|_{\partial\Omega}\in\mathrm{W}^{s,p}(\partial\Omega)$ with $\norm{f_{i}}_{\mathrm{W}^{s,p}(\partial\Omega)}\lesssim \norm{f}_{\mathrm{W}^{s,p}(\partial\Omega)}$. Moreover, $\supp{f_{i}}\subseteq T_{i}(R(r_{i}/2))\cap \partial\Omega$ for all $i\in\set{1,\ldots,N}$, and $f=\sum_{i=1}^{N}f_i$. 
    It suffices to prove that $\norm{N_{*,\beta}u_{i}}_{\El{p}(\partial\Omega)}\lesssim \norm{f_i}_{\mathrm{W}^{s,p}(\partial\Omega)}$, where $u_{i}(X)=\int_{\partial\Omega}f_{i}\dd \omega^{X}_{L}$ for all $X\in\Omega$ and all $i\in\set{1,\ldots, N}$. Indeed, if this holds then we get that $u(X)=\int_{\partial\Omega} f\dd\omega^{X}_{L}=\sum_{i=1}^{N}u_{i}(X)$ and
    \begin{align*}
        \norm{N_{*,\beta}u}_{\El{p}(\partial\Omega)}\leq \sum_{i=1}^{N}\norm{N_{*,\beta}u_{i}}_{\El{p}(\partial\Omega)}\lesssim \sum_{i=1}^{N}\norm{f_i}_{\mathrm{W}^{s,p}(\partial\Omega)}\lesssim \norm{f}_{\mathrm{W}^{s,p}(\partial\Omega)},
    \end{align*}
    as desired.
    
    Let us now fix $i\in\set{1,\ldots, N}$ and suppose that $f\in\mathrm{W}^{s,p}(\partial\Omega)$ is such that $\supp{f}\subseteq T_{i}(R(r_{i}/2))\cap \partial\Omega$. Let $u(X)=\int_{\partial\Omega}f\dd \omega^{X}_{L}$ for $X\in\Omega$. 
    If $\eps>0$ is such that 
    \begin{equation*}
    \Omega_{\eps}=\set{X\in\Omega : \dist(X,\partial\Omega)\geq \eps}\neq \emptyset, 
    \end{equation*}
    then we claim that 
    \begin{equation}\label{interior estimate for u(X) epsilon away from boundary}
        \sup_{X\in\Omega_{\eps}}\abs{u(X)}\lesssim_{\eps} \norm{f}_{\El{p}(\partial\Omega)}.
    \end{equation}
    Indeed, without loss of generality we may assume that $X_0\in\Omega_{\eps}$ (where $X_0$ is the point fixed in Section~\ref{subsection on L harmonic measure in Lipschitz domains}), and therefore (see, e.g., \cite[Definition 1.3.10]{Kenig_cbms_1994}) for all $X\in\Omega_{\eps}$ and $Q\in\partial\Omega$ it holds that 
    \begin{equation*}
        K(X,Q)  = \lim_{r\to 0^{+}} \frac{\omega^{X}_{L}(\Delta_{\partial\Omega}(Q,r))}{\omega^{X_0}_{L}(\Delta_{\partial\Omega}(Q,r))}\leq C_{\eps, \Omega},
    \end{equation*}
    since a Harnack chain argument, connecting $X$ to $X_0$, shows that 
    \begin{equation*}
    \omega^{X}_{L}(\Delta_{\partial\Omega}(Q,r))\leq  C_{\eps, \Omega} \omega^{X_0}_{L}(\Delta_{\partial\Omega}(Q,r)).
    \end{equation*}
    See, for example, \cite[Chapter 1, Section 3]{Kenig_cbms_1994}. As a consequence, for all $X\in\Omega_{\eps}$ it follows from Hölder's inequality that 
    \begin{align*}
        \abs{u(X)}&\leq \int_{\partial\Omega}\abs{f(Q)} K(X,Q) k(Q)\dd \sigma(Q)\leq C_{\eps,\Omega}\norm{f}_{\El{p}(\partial\Omega)}\norm{k}_{\El{q}(\partial\Omega)},
    \end{align*}
    where $\frac{1}{p}+\frac{1}{q}=1$. Since $k\in\El{q}(\partial\Omega)$ by Theorem~\ref{thm equivalent properties for solvability in terms of  L harmonic measure}, \eqref{interior estimate for u(X) epsilon away from boundary} follows.
    
    We first treat the case $Q_{0}\in\partial\Omega\setminus T_{i}(R((1+M)r_{i}))$. If $Q\in\supp{f}\subseteq T_{i}(R(r_{i}/2))\cap\partial\Omega$, then 
    \begin{align*}
        \abs{Q-P_{i}}< (1+M)\frac{r_i}{2}
    \end{align*}
    by Remark~\ref{remark on distances on the Lipschitz boundary in cylinders}. As a consequence,
    \begin{align}\label{equation support interesect ball is empty}
        \abs{Q-Q_{0}}\geq \abs{P_i-Q_{0}}-\abs{Q-P_{i}}> (1+M)r_{i}-(1+M)\frac{r_{i}}{2}=(1+M)\frac{r_{i}}{2},
    \end{align}
    and therefore $\Delta_{\partial\Omega}(Q_{0}, \frac{r_{i}}{2}(1+M))\cap\supp{f}=\emptyset$. Since $\Gamma^{\beta}_{\Omega,\loc}(Q_0)\subseteq \Omega\cap B(Q_{0}, \frac{r_{i}}{8}(1+M))$, it follows from Lemma~\ref{uniform boundedness of solutions near where boundary data vanishes} that
    \begin{equation}\label{tangential max function estimate for points Q_0 away from support of f}
    \sup_{X\in\Gamma^{\beta}_{\Omega,\loc}(Q_0)} \abs{u(X)}\lesssim \norm{f}_{\El{p}(\partial\Omega)}
    \end{equation}
    for all $Q_{0}\in\partial\Omega\setminus T_{i}(R((1+M)r_{i}))$. 
    
    We now treat the case $Q_{0}\in\partial\Omega \cap T_{i}(R((1+M)r_{i}))$ and let $x_0\in \Delta(0, (1+M)r_{i})$ be such that $Q_{0}=\Phi_{i}(x_0)=T_{i}((\varphi_{i}(x_0),x_0))$. It follows from Remark~\ref{remark on distances on the Lipschitz boundary in cylinders} that $\abs{Q_{0}-P_{i}}\leq (1+M)\abs{x}$ where $\abs{x}<(1+M)r_{i}$, hence $\abs{Q_{0}-P_{i}}<(1+M)^{2}r_{i}$. As a consequence,
    \begin{align*}
        \Gamma^{\beta}_{\Omega,\loc}(Q_0)&\subseteq \Omega \cap B(Q_{0} , \tfrac{r_{i}}{8}(1+M))\subseteq \Omega \cap B(P_{i}, \tfrac{r_{i}}{8}(1+M) + r_{i}(1+M)^{2})\\
        &\subseteq \Omega\cap B(P_i , 2(1+M)^{2}r_{i})\subseteq \Omega \cap T_{i}\left(R(2(1+M)^{2}r_{i})\right).
    \end{align*}
    Let $0<t\leq 6(1+M)^{3}r_{i}$ and $x\in \Delta(0,2(1+M)^{2}r_{i})$. Let $Q=\Phi_{i}(x)=T_{i}((\varphi_{i}(x),x))\in T_{i}(\partial\Omega_{i}) \cap T_{i}\left(R(2(1+M)^{2}r_{i})\right)=\partial\Omega\cap T_{i}\left(R(2(1+M)^{2}r_{i})\right)$. It follows from the assumption \eqref{local representation of Omega close to the boundary} and Lemma~\ref{explicit representation of corkscrew points in local cylinder coordinates} that $A_{t}(Q)=T_{i}((\varphi_{i}(x)+t , x))=T_{i}(F_{\varphi_{i}}^{-1}(t,x))$ is a corkscrew point for $Q$. It follows from Lemma~\ref{Lemma domination series dyadic averages at corkscrew points} that  
    \begin{equation*}
    \abs{u(T_{i}(F_{\varphi_{i}}^{-1}(t,x)))}\lesssim \sum_{j=0}^{\infty} 2^{-\alpha_{L} j}\left(\frac{1}{\sigma(\Delta_{j})}\int_{D_{j}}\abs{f}^{p_0}\dd \sigma\right)^{1/p_0},
\end{equation*}
where $\Delta_{j}=\Delta_{\partial\Omega}(Q,2^{j+1}t)$ for all $j\geq 0$, $D_{j}=\Delta_{j}\setminus \Delta_{j-1}$ for $j\geq 1$, and $D_{0}=\Delta_{0}$. 

Let $j\geq 1$, and assume that $\supp{f}\cap D_{j}\neq \emptyset$. This implies that there is $P\in T_{i}(R(r_{i}/2))\cap\partial\Omega$ such that $2^{j}t\leq \abs{Q-P}<2^{j+1}t$. Since $\abs{P-P_i}<(1+M)\frac{r_{i}}{2}$ and $\abs{Q-P_{i}}\leq 2(1+M)^{3}r_{i}$ by Remark~\ref{remark on distances on the Lipschitz boundary in cylinders}, it follows that
\begin{align}\label{smallness condition on t if non trivial intersection}
    2^{j}t\leq \abs{Q-P}\leq \abs{Q-P_i}+\abs{P_i- P}<2(1+M)^{3}r_{i} + (1+M)\frac{r_{i}}{2}<3(1+M)^{3}r_{i}.
\end{align}
Observe that if $y\in \Delta(0,C_{M}r_i)$ and $\abs{x-y}<\frac{2^{j+1}t}{(1+M^{2})^{1/2}}$, then
\begin{align*}
    \abs{Q-\Phi_{i}(y)}&=\abs{\Phi_{i}(x)-\Phi_{i}(y)}=\left(\abs{x-y}^{2} + \abs{\varphi_{i}(x)-\varphi_{i}(y)}^{2}\right)^{1/2}\\
    &\leq (1+M^{2})^{1/2}\abs{x-y}<2^{j+1}t,
\end{align*}
which means that $\Phi_{i}(y)\in\Delta_{\partial\Omega}(Q,2^{j+1}t)$.
Therefore, it follows from \eqref{local representation of surface measure from area formula} that
\begin{align*}
    \sigma(\Delta_{j})&\geq \sigma(\Delta_{\partial\Omega}(Q,2^{j+1}t)\cap T_{i}(R(2(1+M^2)r_{i})))\\
    &=\int_{\Delta(0, 2(1+M^2)r_{i})}\ind{\Delta_{\partial\Omega}(Q,2^{j+1}t)}(\Phi_{i}(y))(1+\abs{\nabla\varphi_{i}(y)}^{2})^{1/2}\dd y\\
    &\geq \int_{\Delta(0, 2(1+M^2)r_{i}) }\ind{\Delta\left(x, \frac{2^{j+1}t}{(1+M^{2})^{1/2}}\right)}(y)\dd y \\
    &=\meas{\Delta(0, 2(1+M^2)r_{i}) \cap \Delta\left(x, \frac{2^{j+1}t}{(1+M^{2})^{1/2}}\right)}\\
    &\geq \meas{\Delta\left(\alpha x , \frac{2^{j-3}t}{(1+M^{2})^{1/2}}\right)}
    \gtrsim \meas{\Delta(x,2^{j+1}t)},
\end{align*}
where it is easy to check that $\Delta(\alpha x , \frac{2^{j-3}t}{(1+M^{2})^{1/2}})\subseteq \Delta(0, 2(1+M^2)r_{i}) \cap \Delta(x, \frac{2^{j+1}t}{(1+M^{2})^{1/2}})$, where $\alpha= 1-\frac{2^{j-3}t}{r_{i}(1+M)^{2}(1+M^{2})^{1/2}}\in(0,1)$ because of \eqref{smallness condition on t if non trivial intersection}. In addition, a similar but easier argument shows that $\sigma(\Delta_{0})\gtrsim \meas{\Delta(x,2t)}$. 

Given \eqref{local representation of Omega close to the boundary}, the composition $ f\circ \Phi_{i}$ is in principle only defined on the ball $\Delta(0, C_{M}r_{i})\subset\Rn$. However, since $\supp{f}\subseteq T_{i}(R(r_{i}/2))\cap \partial\Omega$, the function $f\circ \Phi_{i}$ is compactly supported in $\Delta(0, r_{i}/2)$ and therefore trivially extends to a function on $\Rn$ that belongs to $\mathrm{W}^{s,p}(\Rn)$.

Note that $y\in \Delta(x,2^{j+1}t)$ if $\Phi_{i}(y)\in\Delta_{\partial\Omega}(Q,2^{j+1}t)$.
As a consequence, if $j\geq 0$ and $\supp{f}\cap D_{j}\neq\emptyset$, then we can use the previous observations and \eqref{formula for local integral wrt surface measure from area formula} to write
\begin{align*}
    \left(\frac{1}{\sigma(\Delta_{j})}\int_{D_{j}}\abs{f}^{p_0}\dd \sigma\right)^{\frac{1}{p_0}}&\lesssim \left(\frac{1}{\meas{\Delta(x,2^{j+1}t)}}\int_{\Delta(0,C_{M}r_{i})}\ind{\Delta_{\partial\Omega}(Q,2^{j+1}t)}(\Phi_{i}(y))\abs{(f\circ \Phi_{i})(y)}^{p_0}\dd y\right)^{\frac{1}{p_0}}\\
    & \leq \left(\dashint_{\Delta(x,2^{j+1}t)}\abs{(f\circ \Phi_{i})(y)}^{p_0}\dd y\right)^{\frac{1}{p_0}}.
\end{align*}
If $ \supp{f}\cap D_{j}=\emptyset$, then the estimate above trivially holds because the left-hand side is zero. We have therefore proved that 
\begin{equation}\label{local estimate on u in terms of infinite sum on dyadic annuli}
    \abs{u(T_{i}(F_{\varphi_{i}}^{-1}(t,x)))}\lesssim \sum_{j=0}^{\infty} 2^{-\alpha_{L} j}\left(\dashint_{\Delta(x,2^{j+1}t)}\abs{(f\circ \Phi_{i})(y)}^{p_0}\dd y\right)^{1/p_0}
\end{equation}
for all $x\in \Delta(0, 2(1+M)^{2}r_{i})$ and all $0<t\leq 6(1+M)^{3}r_{i}$. 

Recall that we let $c=\frac{1}{2}$. Observe that $\Gamma_{\Omega,\loc}^{\beta}(Q_0)\subseteq \Gamma_{\Omega}^{\beta,\frac{1}{2}}(Q_0)\cap T_{i}(R(2(1+M)^{2}r_{i}))$, and recall that $Q_{0}=\Phi_{i}(x_0)=T_{i}((\varphi_{i}(x_0),x_0))$. It follows from \eqref{local representation of Omega close to the boundary}, Lemma \ref{lemma tangential approach regions from bounded to special Lipschitz} and Lemma \ref{lemma inclusion of domain dependent tangential approach regions into pullback of upper half space regions} that 
\begin{align*}
    T_{i}^{-1}(\Gamma_{\Omega,\loc}^{\beta}(Q_{0}))&\subseteq T_{i}^{-1}(\Gamma_{\Omega}^{\beta,\frac{1}{2}}(Q_0)\cap T_{i}(R((1+M)^{2}r_{i})))\\
    & \subseteq \Gamma_{\Omega_{i}}^{\beta,c}(T_{i}^{-1}(Q_0))\cap R(2(1+M)^{2}r_{i}) \\
    &\subseteq F_{\varphi_{i}}^{-1}(\Gamma^{\beta}_{\frac{3}{2}}(x_0))\cap R(2(1+M)^{2}r_{i}).
\end{align*}
Consequently, if $X\in\Gamma_{\Omega,\loc}^{\beta}(Q_0)$, then $X=T_{i}(F_{\varphi_{i}}^{-1}(t,x))$ where $(t,x)\in\Gamma^{\beta}_{\frac{3}{2}}(x_0)$ and $(t,x)=F_{\varphi_{i}}(s,y)=(s-\varphi_{i}(y),y)$ for some $(s,y)\in R(2(1+M)^{2}r_{i})$. 
This implies that 
\begin{align*}
0<t&\leq \abs{s}+\abs{\varphi_{i}(y)}\leq 2(1+M)^{3}r_{i} + 2(1+M)^{2}Mr_{i}\leq 4(1+M)^{3}r_{i}< 6(1+M)^{3}r_{i}.
\end{align*}
It therefore follows from \eqref{local estimate on u in terms of infinite sum on dyadic annuli} that
\begin{equation}\label{main estimate localisation in bdd lipschitz tangential max}
    \sup_{X\in\Gamma_{\Omega,\loc}^{\beta}(Q_0)}\abs{u(X)}\lesssim \sup_{(t,x)\in\Gamma^{\beta}_{\frac{3}{2}}(x_0)} \left(\sum_{j=0}^{\infty} 2^{-\alpha_{L} j}\left(\dashint_{\Delta(x,2^{j+1}t)}\abs{(f\circ \Phi_{i})(y)}^{p_0}\dd y\right)^{1/p_0}\right),
\end{equation}
where we recall that $Q_{0}=\Phi_{i}(x_0)$ for $x_0\in \Delta(0, (1+M)r_{i})$.
The proof of Theorem~\ref{main theorem upper half space for Cpalpha} (see in particular Remark~\ref{remark main implication of main thm within proof}) then shows that the maximal function ${\widetilde{\mathcal{M}}_{p,p_0}^{\beta}(f\circ \Phi_{i}):\Rn\to [0,\infty]}$, which is the analogue for the regions $\Gamma^{\beta}_{\frac{3}{2}}(x_0)$ of the maximal function $\mathcal{M}_{p,p_0}^{\beta}(f\circ \Phi_{i})$ introduced in Remark~\ref{remark main implication of main thm within proof} (see also \eqref{introduction of the big maximal function for later use}) for the regions $\Gamma^{\beta}(x_0)$, is such that 
\begin{equation*}
\sup_{X\in\Gamma_{\Omega,\loc}^{\beta}(Q_0)}\abs{u(X)}\lesssim \widetilde{\mathcal{M}}_{p,p_0}^{\beta}(f\circ \Phi_{i})(x_0)
\end{equation*}
and 
\begin{equation*}
    \smallnorm{\widetilde{\mathcal{M}}_{p,p_0}^{\beta}(f\circ \Phi_{i})}_{\Ell{p}}
    \lesssim \norm{f\circ \Phi_{i}}_{\mathrm{W}^{s,p}(\Rn)}\eqsim \norm{f}_{\mathrm{W}^{s,p}(\partial\Omega)},
\end{equation*}
where we have used Remark~\ref{remark main implication of main thm within proof} and the fact that $\supp{f}\subseteq T_{i}(R(r_{i}/2))\cap\partial\Omega$.

By combining these observations with the interior estimate \eqref{interior estimate for u(X) epsilon away from boundary} we obtain that 
\begin{equation*}
    (N_{*,\beta,\frac{1}{2}}u)(Q_0)=\sup_{X\in\Gamma^{\beta,1/2}_{\Omega}(Q_0)}\abs{u(X)}\lesssim \norm{f}_{\El{p}(\partial\Omega)} + \widetilde{\mathcal{M}}_{p,p_0}^{\beta
    }(f\circ \Phi_{i})(x_0)
\end{equation*}
for $Q_{0}=\Phi_{i}(x_0)$, where $x_0\in \Delta(0, (1+M)r_{i})$. It follows that 
\begin{align*}
    \smallnorm{N_{*,\beta,\frac{1}{2}}u}_{\El{p}(\partial\Omega \cap T_{i}(R((1+M)r_{i})))}&=\left(\int_{\Delta(0,(1+M)r_{i})}(N_{*,\beta,\frac{1}{2}}u)(\Phi_{i}(y))^{p}(1+\abs{\nabla\varphi_{i}(y)}^{2})^{1/2}\dd y\right)^{1/p}\\
    &\lesssim \left(\int_{\Delta(0,(1+M)r_{i})}(\widetilde{\mathcal{M}}_{p,p_0}^{\beta}(f\circ \Phi_{i})(y))^{p}\dd y\right)^{1/p} + \norm{f}_{\El{p}(\partial\Omega)}\\
    &\lesssim \norm{f}_{\mathrm{W}^{s,p}(\partial\Omega)} + \norm{f}_{\El{p}(\partial\Omega)}.
\end{align*}
Combining this estimate with \eqref{tangential max function estimate for points Q_0 away from support of f} yields $\smallnorm{N_{*,\beta,\frac{1}{2}}u}_{\El{p}(\partial\Omega)}\lesssim \norm{f}_{\mathrm{W}^{s,p}(\partial\Omega)}$. We can now use Lemma~\ref{density lemma in boundary sobolev spaces} and the usual approximation argument to obtain the $\sigma$-almost everywhere convergence in~\eqref{convergence sigma a.e. at the boundary of Lipschitz domain}.
\end{proof}

\subsection{Second part of the proof of Theorem~\ref{second main theorem bounded Lipschitz domain}}\label{section second part of the proof of second main thm}
We shall now turn to the proof of the estimate \eqref{upper bound hausdorff dimension of divergence set boundary lipschitz domain} on the Hausdorff dimension of the divergence set. In the context of the space $\partial\Omega$, we work directly with the Hausdorff content $\mathcal{H}^{\gamma}_{\infty}$ instead of working with a general fractional dimensional measure $\mu$ and appealing to Frostman's lemma. It may be possible to use the version of Frostman's lemma on compact metric spaces stated in \cite[Theorem 8.17]{Mattila_1995}, but we do not pursue this here.

The first part of this section is devoted to  showing that for any $f\in\mathrm{W}^{s,p}(\partial\Omega)$, the function $\tilde{f} :\partial\Omega \to (-\infty,\infty]$ defined by 
\begin{equation}\label{preferred representative for sobolev on boundary lipschitz}
    \tilde{f}(Q)=\left\{\begin{array}{ll}
    \lim_{r\to 0^{+}}\dashint_{\Delta_{\partial\Omega}(Q,r)} f\dd \sigma, & \textup{if the limit exists in }\R;\\
    \infty, & \textup{otherwise},
    \end{array}\right.
\end{equation}
is a representative of $f\in\mathrm{W}^{s,p}(\partial\Omega)$ for which the estimate \eqref{upper bound hausdorff dimension of divergence set boundary lipschitz domain} makes sense (see Lemma~\ref{lemma important properties of the preferred representative on boundary lipschitz} below). Recall that the surface measure $\sigma$ on $\partial\Omega$ is $n$-Ahlfors--David regular by \eqref{surface measure sigma is ahlfors regular}, and is therefore doubling. In particular, $\partial\Omega$ is a (compact) space of homogeneous type, thus the limit in \eqref{preferred representative for sobolev on boundary lipschitz} exists for $\sigma$-almost every $Q\in\partial\Omega$ (see, e.g., \cite[Chapter \RNum{1}, Theorem 1]{Stein_harmonic_analysis_1993}) and $\tilde{f}$ is indeed a representative of $f\in\mathrm{W}^{s,p}(\partial\Omega)$. 

We need some preliminaries about  classical maximal functions on $\partial\Omega$ and their relation with Hausdorff content. The (centred) Hardy--Littlewood maximal function of $f\in\El{1}_{\loc}(\partial\Omega)$ is defined by 
\begin{equation*}
    M_{\partial\Omega}f(Q)=\sup_{0<r\leq \diam{\Omega}}\dashint_{\Delta_{\partial\Omega}(Q,r)}\abs{f}\dd \sigma\quad \textup{for all }Q\in\partial\Omega.
\end{equation*}
More generally, for $0\leq \alpha<n$ we shall need the fractional maximal function defined by
\begin{equation*}
    M_{\partial\Omega}^{\alpha}f(Q)=\sup_{0<r\leq \diam{\Omega}}\sigma(\Delta(Q,r))^{\tfrac{\alpha}{n}}\dashint_{\Delta_{\partial\Omega}(Q,r)}\abs{f}\dd \sigma\quad \textup{for all }Q\in\partial\Omega.
\end{equation*} 
For $\gamma\geq 0$ and $d\geq1$, the $\gamma$-dimensional (unlimited) Hausdorff content of $A\subseteq\R^{d}$ is defined by
\begin{equation*}
    \mathcal{H}^{\gamma}_{\infty}(A)=\inf\set{\sum_{i=1}^{\infty}\diam{E_i}^{\gamma} : A\subseteq \bigcup_{i=1}^{\infty}E_{i}},
\end{equation*}
where the infimum is taken over all countable covers of $A$ by arbitrary subsets of $\R^{d}$. Note that $\mathcal{H}^{\gamma}(A)=0$ if and only if $\mathcal{H}^{\gamma}_{\infty}(A)=0$ (see, e.g., \cite[Lemma 4.6]{Mattila_1995}).

A Vitali covering argument analogous to that used in the proof of \cite[Lemma 1]{Dorronsoro_1986} yields the following result.
\begin{lem}\label{lemma hausdorff measure estimate on fractional max function on lipschitz boundary}
    Let $p\in[1,\infty)$, $\alpha\in (0,\frac{n}{p}]$ and $f\in\El{p}(\partial\Omega)$. For all $\lambda>0$ and $\gamma\geq n-\alpha p$ it holds that 
    \begin{equation*}
        \mathcal{H}^{\gamma}_{\infty}\left(\set{Q\in\partial\Omega : M^{\alpha}_{\partial\Omega}f(Q)>\lambda}\right)\lesssim \lambda^{-p}\norm{f}_{\El{p}(\partial\Omega)}^{p}.
    \end{equation*}
    The same estimate also holds with $M^{\alpha}_{\partial\Omega}f$ replaced by $(M^{\alpha s}_{\partial\Omega}\abs{f}^{s})^{1/s}$ for any $1\leq s\leq p$.
\end{lem}
A proof of the following lemma can be found, for example, in \cite[Lemma 2]{Dorronsoro_1986}.
\begin{lem}\label{lemma Hausdorff measure estimates for riesz potentials}
    Let $p\in[1,\infty)$, $\alpha\in(0,\frac{n}{p})$ and $f\in\Ell{p}$. For all $\lambda>0$ and $\gamma > n-\alpha p$ it holds that 
    \begin{equation*}
        \mathcal{H}^{\gamma}_{\infty}\left(\set{x\in\Rn  : I_{\alpha}\ast f(x)>\lambda}\right)\lesssim (\lambda^{-1}\norm{f}_{\Ell{p}})^{\frac{p\gamma}{n-\alpha p}}.
    \end{equation*}
\end{lem}
The following lemma is the main ingredient needed for the proof of Lemma~\ref{lemma important properties of the preferred representative on boundary lipschitz}. Coincidentally, its proof contains the essence of the argument that we shall use in the remaining part of the proof of Theorem~\ref{second main theorem bounded Lipschitz domain}.
\begin{lem}\label{lemma main estimate maximal function sobolev on bdry lipschitz}
    Let $p\in[1,\infty)$, $s\in(0,\frac{n}{p})$ and $f\in\mathrm{W}^{s,p}(\partial\Omega)$. For all $\lambda>0$ and $\gamma>n-sp$ it holds that 
    \begin{equation*}
        \mathcal{H}^{\gamma}_{\infty}\left(\set{Q\in\partial\Omega : M_{\partial\Omega}f(Q)>\lambda}\right)\lesssim (\lambda^{-1}\norm{f}_{\El{p}(\partial\Omega)})^{p} + (\lambda^{-1}\norm{f}_{\mathrm{W}^{s,p}(\partial\Omega)})^{\frac{\gamma p}{n-sp}}.
    \end{equation*}
\end{lem}
\begin{proof}
    As in the first part of the proof of Theorem~\ref{second main theorem bounded Lipschitz domain}, we can use the covering \eqref{smallest covering of boundary of Omega} and a subordinate partition of unity to decompose $f=\sum_{i=1}^{N}f_{i}$, with $\supp{f_{i}}\subseteq \partial\Omega\cap T_{i}(R(r_{i}/2))$ and $\norm{f_{i}}_{\mathrm{W}^{s,p}(\partial\Omega)}\lesssim \norm{f}_{\mathrm{W}^{s,p}(\partial\Omega)}$ for all $i\in\set{1,\ldots, N}$. The sub-additivity of the maximal operator $M_{\partial\Omega}$ shows that for all $\lambda>0$ it holds that 
    \begin{equation*}
        \set{Q\in\partial\Omega : M_{\partial\Omega}f(Q)>\lambda }\subseteq \bigcup_{i=1}^{N}\set{Q\in\partial\Omega : M_{\partial\Omega}f_{i}>\lambda/N}.
    \end{equation*}
    We can therefore assume from now on that $\supp{f}\subseteq \partial\Omega \cap T_{i}(R(r_i /2))$ for some $i\in\set{1,\ldots ,N}$. 
    
    If $Q\in\partial\Omega\setminus T_{i}(R((1+M)r_{i}))$ (i.e. if $Q$ is far from the support of $f$), then $\Delta_{\partial\Omega}(Q, \delta_{i})\cap\supp{f}=\emptyset$, where $\delta_{i}=\frac{r_{i}}{2}(1+M)$ (see \eqref{equation support interesect ball is empty} in the first part of the proof of Theorem~\ref{second main theorem bounded Lipschitz domain}). Consequently, it follows from the $n$-Ahlfors--David regularity of $\sigma$ that
    \begin{align}
    \begin{split}
    \label{far away estimate for hl max on lipschitz boundary}
        M_{\partial\Omega}f(Q)&=\sup_{\delta_{i} <r\leq \diam{\Omega}}\dashint_{\Delta_{\partial\Omega}(Q,r)}\abs{f}\dd \sigma \\
        &\lesssim \sup_{\delta_{i} <r\leq \diam{\Omega}} \sigma (\Delta_{\partial\Omega}(Q,r))^{s/n}\dashint_{\Delta_{\partial\Omega}(Q,r)}\abs{f}\dd
        \sigma \leq M^{s}_{\partial\Omega}f(Q).
        \end{split}
    \end{align}
    If $Q\in\partial\Omega\cap T_{i}(R((1+M)r_{i}))$ (i.e. if $Q$ is close to the support of $f$), then $Q=\Phi_{i}(x_0)=T_{i}((\varphi_{i}(x_0),x_0))$ for some $x_{0}\in\Delta(0, (1+M)r_{i})$ (using the notation introduced at the start of the proof of Theorem~\ref{second main theorem bounded Lipschitz domain}) and we may use \eqref{formula for local integral wrt surface measure from area formula} and \eqref{surface measure sigma is ahlfors regular} to obtain that
    \begin{align}
    \begin{split}
    \label{local estimate for hl max on lipschitz boundary}
        M_{\partial\Omega}f(Q)&=\sup_{0<r\leq \diam{\Omega}}\dashint_{\Delta_{\partial\Omega}(Q,r)}\abs{f}\dd \sigma\\
        &\eqsim \sup_{0<r\leq \diam{\Omega}}\frac{1}{\sigma(\Delta_{\partial\Omega}(Q,r))}\int_{\Delta(0,r_{i}/2)}\abs{f\circ \Phi_{i}(x)}\ind{\Delta_{\partial\Omega}(\Phi_{i}(x_0),r)}(\Phi_{i}(x))\dd x\\
        &\lesssim \sup_{0<r\leq \diam{\Omega}} r^{-n}\int_{\Delta(0,r_{i}/2)\cap \Delta(x_0,r)}\abs{f\circ \Phi_{i}(x)}\dd x\\
        &\lesssim M(f\circ \Phi_{i})(x_0),
        \end{split}
    \end{align}
    where we have also used the fact that $\abs{x-x_{0}}<r$ if $\abs{\Phi_{i}(x)-\Phi_{i}(x_0)}<r$ in the third line.
    Since $f\circ \Phi_{i} \in\mathrm{W}^{s,p}(\Rn)\subseteq \mathrm{C}^{p}_{s}(\Rn)$ (see Remark~\ref{remark main implication of main thm within proof}), the estimate \eqref{domination maximal function by riesz potential of sharph function} and Lemma~\ref{lemma Hausdorff measure estimates for riesz potentials} show that for all $\lambda>0$ and $\gamma > n-sp$ it holds that
    \begin{align}
    \begin{split}\label{estimate of the local part of the maximal function of local coordinate rep}
        \mathcal{H}^{\gamma}_{\infty}\left(\set{x_0\in\Rn : M(f\circ\Phi_{i})(x_0)>\lambda }\right)&\leq \mathcal{H}^{\gamma}_{\infty}\left(\set{x_0\in\Rn : I_{s}\ast S_{s}(f\circ\Phi_{i})(x_0)\gtrsim\lambda }\right)\\
        &\lesssim (\lambda^{-1}\norm{S_{s}(f\circ\Phi_{i})}_{\Ell{p}})^{\frac{\gamma p}{n-sp}}\\
        &\lesssim (\lambda^{-1}\norm{f\circ \Phi_{i}}_{\mathrm{C}^{p}_{s}(\Rn)})^{\frac{\gamma p}{n-sp}}\\
        &\lesssim (\lambda^{-1}\norm{f\circ \Phi_{i}}_{\mathrm{W}^{s,p}(\Rn)})^{\frac{\gamma p}{n-sp}}\\
        &\lesssim (\lambda^{-1}\norm{f}_{\mathrm{W}^{s,p}(\partial\Omega)})^{\frac{\gamma p}{n-sp}}.
        \end{split}
    \end{align}
    Note that since $\Phi_{i}:\Rn\to\R^{1+n}$ is Lipschitz, it holds that $\mathcal{H}^{\gamma}_{\infty}(\Phi_{i}(A))\lesssim \mathcal{H}^{\gamma}_{\infty}(A)$ for all $A\subseteq\Rn$.
    As a consequence, it follows from the estimates \eqref{far away estimate for hl max on lipschitz boundary}, \eqref{local estimate for hl max on lipschitz boundary} and Lemma~\ref{lemma hausdorff measure estimate on fractional max function on lipschitz boundary} that
    \begin{align*}
        \mathcal{H}^{\gamma}_{\infty}\left(\set{Q\in\partial\Omega : M_{\partial\Omega}f(Q)>\lambda}\right)&\leq \mathcal{H}^{\gamma}_{\infty}\left(\set{Q\in\partial\Omega\cap T_{i}(R((1+M)r_{i})) : M_{\partial\Omega}f(Q)>\lambda}\right)\\
        &\quad + \mathcal{H}^{\gamma}_{\infty}\left(\set{Q\in\partial\Omega\setminus T_{i}(R((1+M)r_{i})) : M_{\partial\Omega}f(Q)>\lambda}\right)\\
        &\lesssim \mathcal{H}^{\gamma}_{\infty}\left(\Phi_{i}\left(\set{x_{0}\in\Delta(0, (1+M)r_{i}) : M(f\circ\Phi_{i})(x_0)\gtrsim \lambda }\right)\right)\\
        &\quad + \mathcal{H}^{\gamma}_{\infty}\left(\set{Q\in\partial\Omega\setminus T_{i}(R((1+M)r_{i})) : M_{\partial\Omega}^{s}f(Q)\gtrsim \lambda}\right)\\
        &\lesssim (\lambda^{-1}\norm{f}_{\mathrm{W}^{s,p}(\partial\Omega)})^{\frac{\gamma p}{n-sp}} + (\lambda^{-1}\norm{f}_{\El{p}(\partial\Omega)})^{p}.\qedhere
    \end{align*}
\end{proof}

We can now use Lemma~\ref{lemma main estimate maximal function sobolev on bdry lipschitz} to obtain the following analogue of Lemma~\ref{lemma existence of limit of averages away from a small set for Cpalpha} which, as we will show, guarantees that the representative $\tilde{f}:\partial\Omega \to (-\infty,\infty]$ defined in \eqref{preferred representative for sobolev on boundary lipschitz} has the required properties for the proof of the remaining part of Theorem~\ref{second main theorem bounded Lipschitz domain} to go through.
\begin{lem}\label{lemma important properties of the preferred representative on boundary lipschitz}
    Let $p\in[1,\infty)$, $s\in(0,\frac{n}{p}]$ and $f\in\mathrm{W}^{s,p}(\partial\Omega)$. There exists $E\subseteq\partial\Omega$ such that the limit
    \begin{equation}\label{limit of averages for sobolev functions on lip domain}
        \lim_{r\to 0^{+}}\dashint_{\Delta_{\partial\Omega}(Q,r)}f\dd \sigma 
    \end{equation}
    exists for all $Q\in\partial\Omega\setminus E$, and $\mathcal{H}^{\gamma}_{\infty}(E)=0$ for all $\gamma>n-sp$. In particular, the function $\tilde{f} :\partial\Omega \to (-\infty,\infty]$ defined by \eqref{preferred representative for sobolev on boundary lipschitz} satisfies $\smallabs{\tilde{f}(Q)}\leq M_{\partial\Omega}f(Q)$ for all $Q\in\partial\Omega\setminus E$. 
\end{lem}
\begin{proof}
    If $s\in(0,\frac{n}{p})$, then the argument of the proof of Lemma~\ref{lemma existence of limit of averages away from a small set for Cpalpha} (replacing the measure $\mu$ with the Hausdorff content $\mathcal{H}^{\gamma}_{\infty}$) shows that the existence of the limit \eqref{limit of averages for sobolev functions on lip domain} follows from the density result in Lemma~\ref{density lemma in boundary sobolev spaces} and the estimate from Lemma~\ref{lemma main estimate maximal function sobolev on bdry lipschitz}. If $s=\frac{n}{p}$, then $\gamma>0$ and we show how the proof of Lemma~\ref{lemma main estimate maximal function sobolev on bdry lipschitz} needs to be adapted. We may choose $\delta\in(0,\frac{n}{p})$ and $\eps>0$ such that $s=\delta +\eps$ and $\gamma>n-\delta p$. Since $f\circ \Phi_{i}\in\mathrm{W}^{s,p}(\Rn)\subseteq\mathrm{C}^{p}_{s}(\Rn)$ (see Remark~\ref{remark main implication of main thm within proof}), it follows from Lemma~\ref{lemma lifting property bessel operator} that there is $g_{i}\in\mathrm{C}^{p}_{\delta}(\Rn)$ such that $f\circ \Phi_{i}=G_{\eps}\ast g_{i}$ and $\norm{g_{i}}_{\mathrm{C}^{p}_{\delta}(\Rn)}\eqsim \norm{f\circ \Phi_{i}}_{\mathrm{C}^{p}_{s}(\Rn)}\lesssim \norm{f\circ\Phi_{i}}_{\mathrm{W}^{s,p}(\Rn)}\lesssim \norm{f}_{\mathrm{W}^{s,p}(\partial\Omega)}$. Consequently, proceeding as in the proof of Lemma~\ref{lemma estimate Cpalpha functions against fractional measures}, we obtain that $$M(f\circ\Phi_{i})(x_0)\lesssim (I_{\delta}\ast G_{\eps}\ast S_{\delta}g_{i})(x_0)$$ for all $x_{0}\in\Rn$, and this estimate can be used to adapt the proof of the local estimate \eqref{estimate of the local part of the maximal function of local coordinate rep} in the proof of Lemma~\ref{lemma main estimate maximal function sobolev on bdry lipschitz} (using Lemma~\ref{lemma Hausdorff measure estimates for riesz potentials}) to obtain that
    \begin{equation*}
        \mathcal{H}^{\gamma}_{\infty}\left(\set{Q\in\partial\Omega : M_{\partial\Omega}f(Q)>\lambda}\right)\lesssim (\lambda^{-1}\norm{f}_{\El{p}(\partial\Omega)})^{p} + (\lambda^{-1}\norm{f}_{\mathrm{W}^{s,p}(\partial\Omega)})^{\frac{\gamma p}{n-\delta p}}
    \end{equation*}
    for all $\lambda>0$. The conclusion follows from this estimate as before.
\end{proof}

We now have all the ingredients needed for the proof of the estimate~\eqref{upper bound hausdorff dimension of divergence set boundary lipschitz domain} on the Hausdorff dimension of the divergence set in  Theorem~\ref{second main theorem bounded Lipschitz domain}. We shall use the same notation that was introduced in the first part of the proof in Section~\ref{section first part of the proof of second main thm}. In particular, recall that $p_{0}\in (1,p)$ is such that $\left(\mathcal{D}\right)^{L}_{p_0}$ holds, and that the covering of $\partial\Omega$ given by~\eqref{smallest covering of boundary of Omega} satisfies~\eqref{local representation of Omega close to the boundary}.

\begin{proof}[Proof of the estimate \eqref{upper bound hausdorff dimension of divergence set boundary lipschitz domain} in Theorem~\ref{second main theorem bounded Lipschitz domain}]
As before, we only prove the case when $c=\tfrac{1}{2}$. Let $f\in\mathrm{W}^{s,p}(\partial\Omega)$. As in the proof of Lemma~\ref{lemma main estimate maximal function sobolev on bdry lipschitz} we may use the covering \eqref{smallest covering of boundary of Omega}, the linearity of the solution operator $f\mapsto u_{f}$ and the sub-additivity of tangential maximal operators to assume that $\supp{f}\subseteq \partial\Omega \cap T_{i}(R(r_{i}/2))$ for some $i\in\set{1,\ldots , N}$. 

If $\beta'\in(\beta,1)$, then there is $s'\in (0,s)$ such that $\beta'=1-\frac{s' p}{n}$. 
If $Q\in\partial\Omega \setminus T_{i}(R((1+M)r_{i}))$, then the estimates~\eqref{interior estimate for u(X) epsilon away from boundary} and \eqref{tangential max function estimate for points Q_0 away from support of f} in the first part of the proof of Theorem~\ref{second main theorem bounded Lipschitz domain} (with $\beta'$ in place of $\beta$ and $p_0$ in place of $p$), the boundedness of $\partial\Omega$ and the (crude) argument used to prove the estimate~\eqref{far away estimate for hl max on lipschitz boundary} show that
\begin{align*}
    N_{*,\beta', \tfrac{1}{2}}u(Q)\lesssim \norm{f}_{\El{p_{0}}(\partial\Omega)}\lesssim M_{\partial\Omega}(\abs{f}^{p_{0}})(Q)^{1/p_{0}}\lesssim (M_{\partial\Omega}^{(s-s')p_{0}}\abs{f}^{p_0})(Q)^{1/p_0}.
\end{align*}
It follows from Lemma~\ref{lemma hausdorff measure estimate on fractional max function on lipschitz boundary} that for all $\lambda>0$ and $\gamma>n-(s-s')p$ it holds that 
\begin{align*}
    \mathcal{H}^{\gamma}_{\infty}\left(\set{Q\in\partial\Omega \setminus T_{i}(R((1+M)r_{i})) : N_{*,\beta', \tfrac{1}{2}}u(Q) >\lambda }\right)\lesssim (\lambda^{-1}\norm{f}_{\El{p}(\partial\Omega)})^{p}.
\end{align*}

If $Q\in\partial\Omega \cap T_{i}(R((1+M)r_{i}))$, then there is $x_{0}\in\Delta(0,(1+M)r_{i})$ such that $Q=\Phi_{i}(x_0)$ (recall the notation introduced at the start of the proof of the first part of Theorem~\ref{second main theorem bounded Lipschitz domain}). 
Since $f\circ\Phi_{i}\in\mathrm{W}^{s,p}(\Rn)\subseteq\mathrm{C}^{p}_{s}(\Rn)$ (see Remark~\ref{remark main implication of main thm within proof}), it follows from Lemma~\ref{lemma lifting property bessel operator} that there is $g_{i}\in\mathrm{C}^{p}_{s'}(\Rn)$ such that $f\circ \Phi_{i}=G_{s-s'}\ast g_{i}$ and $\norm{g_{i}}_{\mathrm{C}^{p}_{s'}(\Rn)}\eqsim\norm{f\circ \Phi_{i}}_{\mathrm{C}^{p}_{s}
(\Rn)}\lesssim \norm{f\circ \Phi_{i}}_{\mathrm{W}^{s,p}(\Rn)}\lesssim \norm{f}_{\mathrm{W}^{s,p}(\partial\Omega)}$. The estimates~\eqref{interior estimate for u(X) epsilon away from boundary} and \eqref{main estimate localisation in bdd lipschitz tangential max} in the first part of the proof of Theorem~\ref{second main theorem bounded Lipschitz domain} (again, with $\beta'$ in place of $\beta$ and $p_0$ in place of $p$), the argument used in \eqref{local estimate for hl max on lipschitz boundary} and Remark~\ref{remark 2 - convolution commute  with general maximal functions approach} show that 
\begin{align}
\begin{split}\label{local estimate on the tangential maximal function in lipschitz domain}
    N_{*,\beta',\tfrac{1}{2}}u_{f}(Q)&\lesssim \norm{f}_{\El{p_{0}}(\partial\Omega)} + \sup_{(t,x)\in\Gamma^{\beta'}_{\frac{3}{2}}(x_0)} \left(\sum_{j=0}^{\infty} 2^{-\alpha_{L} j}\left(\dashint_{\Delta(x,2^{j+1}t)}\abs{(f\circ \Phi_{i})(y)}^{p_0}\dd y\right)^{1/p_0}\right)\\
    &\lesssim M_{\partial\Omega}(\abs{f}^{p_{0}})(Q)^{1/p_{0}} + \sup_{(t,x)\in\Gamma^{\beta'}_{\frac{3}{2}}(x_0)} \left(\sum_{j=0}^{\infty} 2^{-\alpha_{L} j}\left(\dashint_{\Delta(x,2^{j+1}t)}\abs{G_{s-s'}\ast g_{i}}^{p_0}\right)^{1/p_0}\right)\\
    &\lesssim M_{p_{0}}(f \circ \Phi_{i})(x_0) + (G_{s-s'}\ast N_{*,\beta'}^{3/2}w_{i})(x_0),
    \end{split}
\end{align}
where the function $w_{i}:\Hn\to [0,\infty]$ is defined by
\begin{equation*}
        w_{i}(t,x)=\sum_{j=0}^{\infty} 2^{-\alpha_{L}j} \left(\dashint_{\Delta(x,2^{j+1}t)}\abs{g_{i}(y)}^{p_0}\dd y\right)^{1/p_{0}} \quad \textup{ for all } (t,x)\in\Hn.
\end{equation*}
Since $G_{s-s'}(x)\leq I_{s-s'}(x)$ for all $x\in\Rn$, it follows from Lemma~\ref{lemma maximal function commutes with convolution} and Lemma~\ref{lemma Hausdorff measure estimates for riesz potentials} that for all $\lambda>0$ and $\gamma> n-(s-s')p$ it holds that 
\begin{align*}
    \mathcal{H}^{\gamma}_{\infty}\left(\set{x_{0}\in\Rn : M_{p_0}(f\circ \Phi_{i})(x_0) \gtrsim \lambda }\right)&\lesssim \mathcal{H}^{\gamma}_{\infty}\left(\set{x_{0}\in\Rn : (G_{s-s'}\ast  M_{p_0}g_{i})(x_0) \gtrsim \lambda }\right)\\
    &\lesssim (\lambda^{-1}\norm{M_{p_0}g_{i}}_{\Ell{p}})^{\frac{p\gamma}{n-(s-s')p}}\\
    &\lesssim (\lambda^{-1}\norm{g_{i}}_{\Ell{p}})^{\frac{p\gamma}{n-(s-s')p}}\\
    &\lesssim (\lambda^{-1}\norm{f}_{\mathrm{W}^{s,p}(\partial\Omega)})^{\frac{p\gamma}{n-(s-s')p}},
\end{align*}
and similarly, it follows from Lemma~\ref{lemma Hausdorff measure estimates for riesz potentials} and Remark~\ref{remark main implication of main thm within proof} that
\begin{align*}
    \mathcal{H}^{\gamma}_{\infty}\left(\set{x_{0}\in\Rn : (G_{s-s'}\ast N_{*,\beta'}^{3/2}w_{i})(x_0) \gtrsim \lambda }\right)&\lesssim (\lambda^{-1}\smallnorm{N_{*,\beta'}^{3/2}w_{i}}_{\Ell{p}})^{\frac{p\gamma}{n-(s-s')p}}\\
    &\lesssim (\lambda^{-1}\norm{g_{i}}_{\mathrm{C}^{p}_{s'}(\Rn)})^{\frac{p\gamma}{n-(s-s')p}}\\
    &\lesssim(\lambda^{-1}\norm{f}_{\mathrm{W}^{s,p}(\partial\Omega)})^{\frac{p\gamma}{n-(s-s')p}}.
\end{align*}
Consequently, since $\mathcal{H}^{\gamma}_{\infty}(\Phi_{i}(A))\lesssim \mathcal{H}^{\gamma}_{\infty}(A)$ for all $A\subseteq\Rn$, we obtain that for all $\lambda>0$ and $\gamma>n-(s-s')p$ it holds that 
\begin{align}
\begin{split}\label{global estimate for hausdorff content of tangential max in lipschitz}
    \mathcal{H}^{\gamma}_{\infty}&\left(\set{Q\in\partial\Omega : N_{*,\beta',\tfrac{1}{2}}u_{f}(Q) >\lambda}\right)\\
    &\quad\quad\quad\quad\quad\lesssim
     (\lambda^{-1}\norm{f}_{\El{p}(\partial\Omega)})^{p}\\ 
     &\quad\quad\quad\quad\quad\quad + \mathcal{H}^{\gamma}_{\infty}\left(\Phi_{i}\left(\set{x_{0}\in\Delta(0, (1+M)r_{i}) : M_{p_{0}}(f\circ \Phi_{i})(x_0)\gtrsim \lambda }\right)\right)\\
     &\quad\quad\quad\quad\quad\quad + \mathcal{H}^{\gamma}_{\infty}\left(\Phi_{i}\left(\set{x_{0}\in\Delta(0, (1+M)r_{i}) : (G_{s-s'}\ast N_{*,\beta'}^{3/2}w_{i})(x_0) \gtrsim \lambda }\right)\right)\\
     &\quad\quad\quad\quad\quad\lesssim (\lambda^{-1}\norm{f}_{\El{p}(\partial\Omega)})^{p} + (\lambda^{-1}\norm{f}_{\mathrm{W}^{s,p}(\partial\Omega)})^{\frac{p\gamma}{n-(s-s')p}},
     \end{split}
\end{align}
which is the estimate that we need.

If $\beta'=1$, then $s'=0$ and we treat two cases: First, if $s\in(0,\frac{n}{p})$, then we use the fact that $\abs{(f\circ \Phi_{i})(y)}\lesssim I_{s}\ast S_{s}(f\circ\Phi_{i})(y)$ for almost every $y\in\Rn$ (see the proof of Lemma~\ref{lemma estimate Cpalpha functions against fractional measures}) as well as Lemma~\ref{lemma maximal function commutes with convolution} to obtain that 
\begin{equation*}
    M_{p_0}(f\circ \Phi_{i})(x_0)\lesssim I_{s}\ast M_{p_0}(S_{s}(f\circ \Phi_{i}))(x_0) \quad \textup{ for all }x_{0}\in\Rn.
\end{equation*}
Remark~\ref{remark 2 - convolution commute  with general maximal functions approach} shows that a similar inequality also holds for the tangential maximal function in the right-hand side of the first line of \eqref{local estimate on the tangential maximal function in lipschitz domain}. Since $\norm{S_{s}(f\circ \Phi_{i})}_{\Ell{p}}\lesssim \norm{f\circ\Phi_{i}}_{\mathrm{W}^{s,p}(\Rn)}\lesssim \norm{f}_{\mathrm{W}^{s,p}(\partial\Omega)}$, the same argument as above yields the conclusion of \eqref{global estimate for hausdorff content of tangential max in lipschitz} with $s'=0$. 

Second, if $s=\frac{n}{p}$, then we have $\gamma>0$ and we can write $s=\delta +\eps$ with $\gamma>n-\delta p$. Again, by Remark~\ref{remark main implication of main thm within proof} and Lemma~\ref{lemma lifting property bessel operator} there is $h_{i}\in\mathrm{C}^{p}_{\delta}(\Rn)$ such that $f\circ \Phi_{i}= G_{\eps}\ast h_{i}$ and $\norm{h_{i}}_{\mathrm{C}^{p}_{\delta}(\Rn)}\eqsim \norm{f\circ\Phi_{i}}_{\mathrm{C}^{p}_{s}(\Rn)}\lesssim \norm{f\circ \Phi_{i}}_{\mathrm{W}^{s,p}(\Rn)}\lesssim \norm{f}_{\mathrm{W}^{s,p}(\partial\Omega)}$. As in the proof of Lemma~\ref{lemma estimate Cpalpha functions against fractional measures}, it holds that 
\begin{equation*}
    \abs{(f\circ \Phi_{i})(y)}\lesssim (I_{\delta}\ast G_{\eps}\ast S_{\delta}h_{i})(y)
\end{equation*}
for almost every $y\in\Rn$, and therefore $M_{p_0}(f\circ \Phi_{i})(x_0)\lesssim (I_{\delta}\ast G_{\eps}\ast M_{p_0}(S_{\delta}h_{i}))(x_0)$ for all $x_{0}\in\Rn$. As before, a similar estimate holds for the tangential maximal function in the first line of \eqref{local estimate on the tangential maximal function in lipschitz domain}. We can therefore use the same argument as above to obtain for all $\lambda>0$ that 
\begin{align}\label{second case global estimate for hausdorff content of tangential max in lipschitz}
    \mathcal{H}^{\gamma}_{\infty}\left(\set{Q\in\partial\Omega : N_{*,\beta',\tfrac{1}{2}}u_{f}(Q) >\lambda}\right)
     \lesssim (\lambda^{-1}\norm{f}_{\El{p}(\partial\Omega)})^{p} + (\lambda^{-1}\norm{f}_{\mathrm{W}^{s,p}(\partial\Omega)})^{\frac{p\gamma}{n-\delta p}}.
\end{align}

Finally, for all $\beta'\in(\beta,1]$ and $s'\in[0,s)$ such that $\beta'=1-\frac{s'p}{n}$, Lemma~\ref{lemma main estimate maximal function sobolev on bdry lipschitz} and (the proof of) Lemma~\ref{lemma important properties of the preferred representative on boundary lipschitz} show that for all $\gamma>n-(s-s')p\geq n-sp$ there is a constant $c(n,p,\gamma)>0$ such that
\begin{equation*}
    \mathcal{H}^{\gamma}_{\infty}\left(\set{Q\in\partial\Omega : \smallabs{\tilde{f}(Q)}>\lambda}\right)\lesssim  (\lambda^{-1}\norm{f}_{\El{p}(\partial\Omega)})^{p} +(\lambda^{-1}\norm{f}_{\mathrm{W}^{s,p}(\partial\Omega)})^{c(n,p,\gamma)}
\end{equation*}
for all $\lambda>0$. We may now use this estimate, the density result in Lemma~\ref{density lemma in boundary sobolev spaces}, the estimates \eqref{global estimate for hausdorff content of tangential max in lipschitz} and \eqref{second case global estimate for hausdorff content of tangential max in lipschitz}, and the argument of the proof of Proposition~\ref{proposition Hausdorff dimension div set for cones and Bessel} (replacing the measure $\mu$ with the Hausdorff content $\mathcal{H}^{\gamma}_{\infty}$) to show that 
\begin{equation*}
    \mathcal{H}^{\gamma}_{\infty}\left(\set{Q\in\partial\Omega : \limsup_{\Gamma^{\beta',1/2}_{\Omega}(Q)\ni X\to Q} \smallabs{u_{f}(X)-\tilde{f}(Q)}\neq 0}\right)=0.
\end{equation*}
for all $\gamma>n-(s-s')p$. The same holds for $\mathcal{H}^{\gamma}$ instead of $\mathcal{H}^{\gamma}_{\infty}$ (see, e.g., \cite[Lemma 4.6]{Mattila_1995}) and the estimate~\eqref{upper bound hausdorff dimension of divergence set boundary lipschitz domain} follows. This concludes the proof of Theorem~\ref{second main theorem bounded Lipschitz domain}.
\end{proof}

\section{Further contextual remarks}\label{Sect:context}
The tangential convergence problems considered in this paper are of course not confined to boundary value problems that are elliptic. In this section we connect with some similar statements concerning parabolic, dispersive and wave equations.

It was shown in the celebrated work of Aronson \cite{Aronson} that the fundamental solution $g(t,x)$ of the equation $\partial_tu=\dvg_x(A(t,x)\nabla_x u)$ enjoys Gaussian bounds on any strip $(0,T]\times\mathbb{R}^n$ provided the real matrix-valued coefficients $A$ are bounded and uniformly elliptic in the sense of \eqref{boundedness aanndd ellipticity conditions}, uniformly in $t\in (0,T]$. Specifically, the solution $u_f$ to the associated initial value problem satisfies $u_f(t,x)\lesssim (H_t*f)(x)$ where $H_t$ is a Euclidean heat kernel; moreover one has a similar lower bound, and all implicit constants depend only on $n$, $T$ and the bounds in \eqref{boundedness aanndd ellipticity conditions}. Consequently the convergence properties of $u_f$ to its initial data $f$ inherit those of the constant coefficient case, and these are easily seen to follow from the classical theory discussed in the introduction (e.g. pointwise domination of the heat kernel by the Poisson kernel and application of Theorem~\ref{theorem Nagel--Stein thm 5} shows that $u_f(t,x)$ tends to $f(x_0)$ in the parabolic tangential approach region $\{(t,x)\in\Hn : \abs{x-x_{0}}<t^{\beta/2} \textup{ if }0<t\leq 1, \textup{ and } \abs{x-x_0}<t^{1/2} \textup{ if } t\geq 1\}$). In certain circumstances at least, similar Gaussian bounds may also be obtained for equations with complex coefficients; see \cite{ATP98}. 

The parabolic setting is somewhat akin to that for elliptic equations with $t$-independent coefficients in the (block) form $-\partial_{t}^2u-\dvg_x(A(x)\nabla_x u)=0$ on $\Hn$. In that case, the unique solution $u_f$ to the Dirichlet problem satisfies $u_{f}(t,x)\lesssim (P_{t}\ast f)(x)$, where $P_{t}$ is the classical Poisson kernel. This follows from the Gaussian upper bound for the kernel of the heat semigroup $e^{-tL_{A}}$ of the elliptic operator $L_{A}=-\dvg_{x}(A\nabla_{x})$ (see, e.g., \cite[Theorem 6.10]{Ouhabaz_Heat_Equations_2005}), since  functional calculus then yields the required bounds for the kernel of the Poisson semigroup $e^{-tL_{A}^{1/2}}$.

Problems of the type that we address in this paper have also attracted interest in the setting of dispersive equations, most notably the free Schr\"odinger equation $i\partial_t u=\Delta_xu$ on $\mathbb{R}^{1+n}_+$. Here matters are very different in that 
some smoothness is \textit{required} of the initial datum $u_0$ in order for $u(t,x)\rightarrow u_0(x)$ to be true for almost every $x$ as $t\rightarrow 0$. Identifying the regularity threshold for $u_0$ in the scale of $\mathrm{L}^2$ Sobolev spaces $\mathrm{H}^s$ was a longstanding problem of Carleson, and one that has only recently been resolved for general dimensions \cite{Du_Zhang} -- the most basic point being that the Schr\"odinger kernel, unlike the Poisson and  heat kernels, is not a classical approximate identity. As a result there are significant phenomenological differences from the elliptic and parabolic settings. For example, it is shown in \cite{BBCR} that the dimension of the divergence set as $t\rightarrow 0$ has discontinuities as a function of $s$.
For tangential convergence results of the type $u(t, x+\gamma(t))\rightarrow u_0(x)$ we refer to \cite{Minguillon} and the references there, and for bounds on the Hausdorff dimensions of sets of divergence we refer to \cite{Du_Zhang, Mattila_Fourier} and the references there. Some of the methods employed in this setting have also proved effective in addressing similar questions for the wave equation $\partial_t^2u=\Delta_x u$ on $\mathbb{R}^{1+n}_+$ -- see for example \cite{BBCR}.

Finally we briefly point out that tangential approach maximal functions similar to $N_{*,\beta}$ arise naturally in the study of the H\"ormander symbol classes, along with other families of oscillatory Fourier multipliers. In that setting the relevant maximal functions are essentially the auxiliary maximal functions of Section \ref{Sect:tang}, taking the form
$$
\mathfrak{M}_{\alpha,\beta}u(x_0):=
\sup_{\substack{(t,x)\in \Gamma^\beta(x_0)\\ 0<t\leq 1}}t^\alpha |u(t,x)|,
$$
where $u(t,x)=P_t*f(x)$, or some other suitable average of a given function $f$.
For $\alpha> 0$ the operator $f\mapsto \mathfrak{M}_{\alpha,\beta}u$ may be interpreted as a fractional maximal averaging operator, effectively generalising that given by \eqref{general tangential maximal function on the tangential approach regions defined by Gamma beta}.
This connection bears some analogy with the fundamental relationship between Calder\'on--Zygmund singular integrals and the Hardy--Littlewood maximal function; see \cite{Beltran, Beltran_Bennett}.
\bibliographystyle{abbrv}
\bibliography{refs}
\end{document}